\documentclass{amsart}

\usepackage{xspace,amssymb,amsfonts,euscript,eufrak,mathrsfs}

\usepackage{amsthm,amsmath}
\usepackage{url}
\usepackage{graphicx}
 \usepackage{epstopdf}
\usepackage{xcolor}
\usepackage{color}
\usepackage[all]{xy}
\usepackage{geometry}

\renewcommand{\subsectionmark}[1]{\markright{}{}}

\newcommand{\CC}{\mathbb{C}}
\newcommand{\NN}{\mathbb{N}}
\newcommand{\RR}{\mathbb{R}}
\newcommand{\ZZ}{\mathbb{Z}}

\newcommand{\g}{\mathfrak{g}}
\newcommand{\gaff}{\widehat{\mathfrak{g}}}
\renewcommand{\b}{\mathfrak{b}}
\newcommand{\baff}{\widehat{\mathfrak{b}}}
\newcommand{\h}{\mathfrak{h}}
\newcommand{\haff}{\widehat{\mathfrak{h}}}

\newcommand{\coQ}{\check{Q}}
\newcommand{\coX}{\check{X}}
\newcommand{\Dre}{\widehat{\Delta}^{\text{re}}}
\newcommand{\al}{\alpha}
\newcommand{\Dw}{\Delta^{w}_{+}}

\newcommand{\A}{\mathcal{A}}
\newcommand{\C}{\mathcal{C}}
\newcommand{\SR}{\mathcal{S}}
\newcommand{\SRA}{\widehat{\mathcal{S}}}
\newcommand{\T}{\widehat{\mathcal{T}}}
\newcommand{\W}{\mathcal{W}}
\newcommand{\WA}{\widehat{\mathcal{W}}}
\newcommand{\WAp}{\widehat{\mathcal{W}}^{\text{par}}}

\newcommand{\He}{\mathbf{H}}
\newcommand{\Me}{\mathbf{M}}

\newcommand{\E}{\mathcal{E}}
\newcommand{\MG}{\mathcal{G}}
\newcommand{\MGa}{\widehat{\mathcal{G}}}
\newcommand{\MGJ}{\widehat{\mathcal{G}}^J}
\newcommand{\MGp}{\MGa^{\text{par}}}
\newcommand{\MGper}{\MGa^{\text{per}}}
\newcommand{\MGst}{\MGa^{\text{stab}}}
\newcommand{\MGYk}{\textbf{MG}(Y_k)}
\newcommand{\V}{\mathcal{V}}

\newcommand{\BMP}{\mathscr{B}}
\newcommand{\BMPp}{\mathscr{B}^{\text{par}}}
\newcommand{\F}{\mathcal{F}}
\renewcommand{\H}{\mathcal{H}}
\newcommand{\I}{\mathcal{I}}
\newcommand{\ShMGk}{\textbf{Sh}_k(\MG)}
\newcommand{\ShMGJk}{\textbf{Sh}_k(\MG^J)}
\newcommand{\ShMGI}[1]{\textbf{Sh}_k(\MG^{\text{#1}}_{|\I})}
\newcommand{\ShZ}{\mathscr{Z}}
\newcommand{\st}[1]{#1^{\text{stab}}}
\newcommand{\stab}[2]{{#1}^{\text{stab},#2}}

\newcommand{\ma}{m_{\alpha}}
\newcommand{\mc}{m_{c}}
\newcommand{\Z}{\mathcal{Z}}
\newcommand{\Zp}{\mathcal{Z}^{\text{par}}}

\newcommand{\Zs}{\mathcal{Z}^s}
\newcommand{\sZ}{{}^s\!\mathcal{Z}}
\newcommand{\Zmod}{\mathcal{Z}-\text{mod}^{\text{f}}}
\newcommand{\Zsmod}{\mathcal{Z}^s-\text{mod}^{\text{f}}}
\newcommand{\sZmod}{{}^s\mathcal{Z}-\text{mod}^{\text{f}}}
\newcommand{\Zpmod}{\mathcal{Z}^{\text{par}}\!-\!\text{mod}^{\text{f}}}
\newcommand{\Zpermod}{\mathcal{Z}^{\text{per}}-\text{mod}^{\text{f}}}

\renewcommand{\H}{\mathcal{H}}
\newcommand{\Hp}{\mathcal{H}^{\text{par}}}

\newcommand{\sth}{{}^s\theta^{\text{par}}}

\newcommand{\tle}{\trianglelefteq}
\newcommand{\dottl}{\stackrel{\centerdot}{\triangleleft}}

\renewcommand{\check}[1]{{#1}^{\vee}}

\newcommand{\Hom}{\text{Hom}}
\newcommand{\im}{\text{im}}
\newcommand{\rk}{\underline{\text{rk}}}
\newcommand{\id}{\text{id}}

\newtheorem{prop}{Proposition}[section]
\newtheorem{theor}{Theorem}[section]
\newtheorem{defin}{Definition}[section]
\newtheorem{cor}{Corollary}[section]
\newtheorem{lem}{Lemma}[section]
\newtheorem{expl}{Example}[section]

\newtheorem{rem}{Remark}[section]
\newtheorem{quest}{Question}[section]

\begin{document}
\title{On the stable moment graph of an affine Kac--Moody algebra}
\author{Martina Lanini}
\address{
Department of  Mathematics \& Statistics, University of Melbourne, Parkville, VIC 3010, Australia}
\email{martina.lanini@unimelb.edu.au}
\maketitle
\vspace{-9mm}
\begin{abstract}
In 1980 Lusztig proved a stabilisation property of the affine Kazhdan-Lusztig polynomials. In this paper we give a categorical 
version of such a result using the theory of sheaves on moment graphs.  This leads us to associate with any Kac-Moody algebra its stable moment graph. 
\end{abstract}

\section{Introduction}

Finding the characters of the irreducible modular representations of semisimple algebraic groups is a fundamental problem in 
the representation theory of finite groups. In \cite{Lu80c} Lusztig stated a conjecture which
related these characters to the affine Kazhdan-Lusztig polynomials and which was similar to a conjecture by Kazhdan and Lusztig
(cf. \cite{KL}) on the composition series of Verma modules for semisimple complex Lie algebras. Kazhdan-Lusztig's conjecture has
 been solved using geometric methods and the theory of $\mathcal{D}$-modules (cf. \cite{BeBe}, \cite{BK}), while Lusztig's conjecture
 has not been completely settled yet. Indeed, by  \cite{AJS}, \cite{KL93}, \cite{KT} it is possible to recover it as a limit from the
 characteristic zero case, but this means that Lusztig's conjecture is not known  to be true yet for all the characteristics
 it is supposed to hold.

The attempt of finding a new approach to these character conjectures brought Fiebig to apply moment graph techniques in 
representation theory. In particular, he associated a particular moment graph to any complex symmetrizable Kac-Moody algebra $\g$: its
 Bruhat graph, that is an oriented graph with labeled edges and whose set of vertices is given by the elements of the Weyl group $W$ 
of $\g$. 
A central role in this theory is played by indecomposable Braden-MacPherson sheaves on the Bruhat graph of $\g$, 
which correspond to indecomposable projective objects admitting a Verma flag in an equivariant version of the representation category 
$\mathscr{O}$ of $\g$ (cf. \cite{Fie08a}). Moreover, moment graphs and Braden-MacPherson sheaves can be constructed also in the 
positive characteristic setting, in which case the moment graph is attached to a semisimple reductive algebraic group $G$.  In this way, 
Fiebig could relate Kazhdan-Lusztig and Lusztig conjectures to a question on the
 graded rank of the stalks of indecomposable Braden-MacPherson sheaves. In particular, he stated a conjectural formula connecting 
this rank explicitly to Kazhdan-Lusztig polynomials. Such a formula is known to hold in characteristic zero and, in this case, to 
be equivalent to the character formula conjectured by Kazhdan and Lusztig (cf. \cite{Fie08a}), while in positive characteristic it 
is expected to be true for characteristics bigger than the Coxeter number of $G$. 

Fiebig's conjectural formula motivated our paper \cite{L11}, where we lifted at the categorical level of sheaves on moment graphs 
certain properties of Kazhdan-Lusztig polynomials and parabolic Kazhdan-Lusztig polynomials. The methods we developed in \cite{L11} 
work in any characteristic under certain technical assumptions, even in cases in which the formula has been proven to fail. In \cite{FW}
Fiebig and Williamson related indecomposable Braden-MacPherson sheaves on a Bruhat graph to parity sheaves on the corresponding flag 
variety, a modular counterpart of intersection cohomology complexes introduced by  Juteau, Mautner and Williamson (cf. \cite{JMW}).
 Once this connection
 is established, by \cite[ Theorem 9.2]{FW} the failure of the formula corresponds to the failure of a parity sheaf to be perverse. 
In this setting, our results proved that even in the cases in which parity sheaves are not perverse, they still satisfy
certain elementary properties that intersection cohomology complexes have in characteristic zero. Anyway, giving a categorical version of properties of Kazhdan-Lusztig 
polynomials is not at all trivial and rather interesting, also in the case in which the connection between Braden-MacPherson sheaves 
and these polynomials
 is known. Indeed, the lifting at a categorical level provides us with extra structure which may help in the understanding of
 objects related to these ubiquitous polynomials.

While in \cite{L11} we considered rather basic equalities involving Kazhdan-Lusztig polynomials, in this paper the result we want to deal
with is much more complex. It concerns a stabilisation property of the affine Kazhdan-Lusztig polynomials, which has been proven by Lusztig in 
\cite{Lu80}. More precisely, for any affine Weyl group $\WA$, Lusztig defined in \cite{Lu80} its \emph{periodic module} $\Me$, that is a free 
$\ZZ[q^{\pm \frac{1}{2}}]$-module with a standard basis indexed by the set $\A$ of alcoves of $\WA$ and also equipped with a structure of
 a module over the Hecke algebra $\He$ of $\WA$. By extending methods of \cite{KL}, he found a nice basis of $\Me$ and showed that
 the change of
basis matrix from the standard basis to this one is given by certain polynomials in $\ZZ[q]$: the generic polynomials, which we will 
denote by $\{Q_{A,B}\}_{A,B\in \A}$. A fundamental property of these polynomials is that they are invariant under finite coweight translation,
 that is $Q_{A,B}=Q_{A+\lambda,B+\lambda}$ for any finite coweight $\lambda$.
 On the other hand, we have another family of polynomials which are attached to the affine Hecke 
algebra: the Kazhdan-Lusztig polynomials $\{P_{x,y}\}_{x,y\in\WA}$. Since it is possible to identify the Weyl group
 $\WA$ with its set of alcoves, we may index any affine Kazhdan-Lusztig polynomial by a pair of alcoves and ask whether $P_{A,B}$ 
and $Q_{A,B}$ are related. This is precisely the point of the  theorem by Lusztig we want to lift at the level of sheaves on moment 
graphs. Indeed, what he proved is that if the alcoves $A,B$ are \emph{far enough} in the fundamental chamber $\C^{+}$, the generic polynomial
 $Q_{A,B}$ coincides with the corresponding affine Kazhdan-Lusztig $P_{A,B}$. The goal of this paper is hence to describe how the
 stalks of Braden-MacPherson sheaves on certain finite intervals of the Bruhat graph behave and in particular to show this
 stabilisation property.

Before outlining the structure of the paper and illustrating more in details in what way we obtain this categorical analogue
 of Lusztig's theorem, we want to motivate our interest in such a result. By linking the affine Hecke algebra and its periodic module,
 the moment graph version of the stabilisation property we mentioned above should allow us to investigate both $\He$ and $\Me$ from a 
categorical point of view. Let $\gaff$ be an affine Kac-Moody algebra and $\WA$ the corresponding (affine) Weyl group. The affine 
Hecke algebra  of $\WA$ controls the representation theory of $\gaff$ at a non-critical level and in this case moment graph
 techniques have already been applied (cf. \cite{Fie08a}), while the periodic module, according to a conjecture by Feigin and Frenkel 
(cf. \cite{FF92}), governs the  the representation theory of $\gaff$ at a critical level, where a moment graph picture has been 
 missing so far.  We believe that the \emph{stable moment graph} $\MGst$ we associate with $\gaff$ in this paper  will be relevant in 
this direction. From the geometric point of view, $\Me$ is related to the semi-infinite flags (cf. \cite{FF}) --also referred to as 
periodic affine Schubert varieties by Lusztig (cf. \cite{Lu90})-- and we expect $\MGst$ to be connected to these objects too. It would 
be very interesting to understand the link between the stable moment graph and this geometric counterpart. Finally, we want to mention
that the definition of $\MGst$ arises from rather intriguing combinatorics, which have been our starting point and which we are
 going to address in this work. 

The paper is organized as follows.

 The aim of the second section is to recollect the basics of the theory of affine Kac-Moody algebras,
 in order to fix the notation. 

In Section 3 we recall the definition of the category of $k$-moment graphs on a lattice,
 stand $k$ for a local ring with $2\in k^{\times}$, we introduced in \cite{L12}.  Let us fix once and for all  an affine Kac-Moody
 algebra $\gaff$, a Borel subalgebra $\baff$ and a Cartan subalgebra $\haff$. Let us moreover denote by $\WA$ the Weyl 
group of $\gaff$ and $\SRA$ the set of simple reflections corresponding to the data we fixed. For any quadruple
 $(\gaff\supseteq \baff\supseteq \haff, J)$, where $J\subseteq \SRA$, we define the Bruhat moment graph $\MGJ$ (see \S\ref{bruhatMG}). 
Moreover, we associate with $\gaff\supseteq \baff\supseteq \haff$ another $k$-moment graph: the periodic one, that we denote by $\MGper$  
and which coincides with $\MGa^{\emptyset}$ up to the orientation of the edges (see \S\ref{ssec_perMG}).

In Section 4 we focus on the parabolic Bruhat graph $\MGp$, that is the one corresponding to the choice of $J$ being $\SR$, 
the set of finite simple reflections. The vertices of this graph are given by the alcoves in the fundamental 
chamber and the goal of this section is to study the behaviour of intervals $\MGp_{|[A,B]}$, with $A$ and $B$ \emph{far enough} 
in  $\C^{+}$.  Our first hope was that the stabilisation phenomenon described by Lusztig would have been visible also at the moment 
graph level. This is unluckily not completely true. Indeed, for any fixed pair of alcoves $A,B$ in the fundamental chamber it is 
possible to translate them in such a way that $\MGp$ restricted to the corresponding interval and considered as \emph{unlabeled} 
graph is invariant under translation by positive multiples of $\rho$, the coweight defined as half the sum of the finite positive
 coroots (see Lemma \ref{iso_Gp}). 
 Since the label function is a fundamental data a moment graph comes with, next step is to discuss the labeling of the edges of
 such intervals. It turns out that the set of edges is in this case bipartite in \emph{stable} and \emph{non-stable} edges, whose
 labels have a good and a bad behaviour, respectively (see Lemma \ref{lem_refl_label_root} and Lemma \ref{lem_transl_label}). 
This gives rise to the definition of a new $k$-moment graph attached to $\gaff$: the \emph{stable} moment graph $\MGst$
(see \S\ref{ssec_stableMG}), that is a subgraph of $\MGp$ having the stabilisation property we were looking for. It is
 the main character of this paper. 

Section 5 deals with the category of sheaves on a $k$-moment graph. After recalling definitions and basic results concerning 
these objects, we introduced the new notion of pushforward functor and we prove that it is right adjoint to the pullback functor 
we introduced in \cite{L11}. Braden-MacPherson sheaves make their appearance in the last part of this section. 
We recall a lemma of 
\cite{L11} telling us that the pullback of isomorphism preserves indecomposable Braden Mac-Pherson sheaves 
(see Lemma \ref{pullbackBMP}). We will refer to this result as ``property of the pullback'' later, since it will be an important tool in the proof
 of our main theorem.  

In the following section we are finally able to state  the moment graph analogue of Lusztig's theorem. Roughly speaking,  our claim is 
that the stalks of indecomposable Braden-MacPherson sheaves on finite intervals of $\MGp$ far enough in $\C^{+}$  stabilise too. Here 
the stable moment graph comes into play.  Suppose that $A,B$ are alcoves far enough in the fundamental chamber. 
Then  results of Section 4 show that the two $k$-moment graphs $\MGp_{|_{[A,B]}}$ and $\MGp_{|_{[A+\mu,B+\mu]}}$ are in general not
 isomorphic, while there is always an isomorphism of  $k$-moment graphs between $\MGst_{|_{[A,B]}}$ and $\MGst_{|_{[A+\mu,B+\mu]}}$,
 for $\mu\in\NN\rho$.  Since the stable moment graph is a subgraph of $\MGp$, there is a morphism $\MGst\hookrightarrow \MGp$. 
The following diagram summarises this situation:
\begin{displaymath}
\begin{xymatrix}{
\MGp_{|_{[A,B]}}&\MGp_{|_{[A+\mu,B+\mu]}}\\
\MGst_{|_{[A,B]}}\ar@{^{(}->}^{g}[u]\ar[r]&\MGst_{|_{[A+\mu,B+\mu]}}\ar@{^{(}->}_{g_{\mu}}[u]
}\end{xymatrix}
\end{displaymath}

We then get a functor ${\cdot}^{\text{stab}}:=g^*:\mathbf{Sh}_{\MGp_{|_{[A,B]}}}\rightarrow \mathbf{Sh}_{\MGst_{|_{[A,B]}}}$.
 The main theorem of the last two sections is the following one.

\begin{theor}The functor ${\cdot}^{\text{stab}}:\textbf{Sh}_k(\MGp_{|_{[A,B]}})\rightarrow \textbf{Sh}_k(\MGst_{|_{[A,B]}})$ preserves indecomposable Braden-MacPherson sheaves.
\end{theor}

Once this result is proven, the stabilisation property of the stalks of indecomposable Braden-MacPherson sheaves follows by applying
 the  property of the pullback we mentioned before.

We conclude the section by demonstrating the theorem above in the subgeneric case, that is the one corresponding to
 $\gaff=\widehat{\mathfrak{sl}_2}$. In particular, we show that for any finite interval of $\MGst$, the associated indecomposable 
Braden-MacPherson sheaf is isomorphic to its structure sheaf $\mathscr{Z}$ (see \ref{sssec_ShZ_Z} for the definition of structure sheaf).
This is done by proving the equivalent  claim (see Proposition \ref{prop_StrucShflabby}) that $\mathscr{Z}$ is flabby. On the other hand,
by Fiebig's multiplicity one result (cf. \cite{Fie06}), it is known that indecomposable Braden-MacPherson sheaves on finite intervals
 of $\MGp$ are isomorphic to the corresponding structure sheaves (in Appendix A we give an alternative proof of this fact). 
This finishes the proof of the subgeneric case.

In Section 7  we address the general case. Easily, we reduce us (see Lemma \ref{lem_reduction}) to prove that the image under ${\cdot}^{\text{stab}}$  of an indecomposable
Braden-MacPherson sheaf is flabby and indecomposable. In order to show that, we use a local-global approach,
 developed by Fiebig in \cite{Fie08a}, which enables us to translate our claim in terms of certain modules over the struc
(5,50\halfrootthree)**\dir{.};ture algebra $\Z$, 
that is the space of global sections of the structure sheaf. More precisely, we write the functor  ${\cdot}^{\text{stab}}$ as composition of 
five functors:
\begin{displaymath}\begin{xymatrix}{
\textbf{Sh}_{\MGp_{|_{\I}}}\ar[r]^{i_{*}}\ar@/_1.6pc/[rrrrr]<-1.2ex>_{{\cdot}^{\text{stab}}}&\textbf{Sh}_{\MGp}
\ar[r]^{p_{\text{par}}^{*}}&\textbf{Sh}_{\MGa} \ar[r]^{{\cdot}^\text{per}}&\textbf{Sh}_{\MGper}\ar[r]^{j^*}&\textbf{Sh}_{\MGper_{|_{\I}}}
\ar[r]^{{\cdot}^{\text{opp}}}&\textbf{Sh}_{\MGst_{|_{\I}}}\\
}\end{xymatrix}
\end{displaymath}

The functors $i_{*}$, $p_{\text{par}}^{*}$, $j^*$ and ${\cdot}^{\text{opp}}$ preserve indecomposable Braden-MacPherson sheaves 
(and hence flabbiness), by results of \cite{L11}. Therefore it is missing to prove that the functor ${\cdot}^\text{per}$ maps 
Braden-MacPherson sheaves to flabby sheaves. To get such a result, translation functor techniques are needed. Once recalled some definitions
 from \cite{Fie08a} and \cite{L12}, we consider any partial order $\trianglelefteq$ on $\WA$ satisfying certain properties 
(see \S\ref{ssec_propertiesBruhat}), we define the notion of $\trianglelefteq$-flabbiness (see Definition \ref{def_tle_flabby}) and,
generalising result of \cite{Fie08a}, we prove that translation functors preserve $\tle$-flabby objects (see Theorem 
\ref{theor_sthper_flabby}). The desired property of the functor ${\cdot}^\text{per}$ follows now by a rather easy inductive procedure (see
 Proposition \ref{prop_flabbyStab}).

Finally, we obtain the indecomposability by considering localisation at certain ideals of \emph{special} $\Z$-modules (see \S\ref{sssec_loc_SM}), which
 allows us to use the subgeneric case -which has been proven in Section 6- and hence to conclude (see Proposition \ref{prop_indp_BMPst}).

Appendix A deals with the subgeneric case. In particular, we prove in a very explicit way that indecomposable Braden-MacPherson sheaves
 on any finite interval of $\MGp$, satisfying the GKM-condition (cf. Definition \ref{Def_GKM}), are isomorphic to the corresponding structure sheaves.
As mentioned before, this result could be obtained by applying \cite[Theorem 6.4]{Fie06}, or --for $k=\mathbb{Q}$-- it could also be derived
 from topological facts, such as the rationally smoothness of the underlying varieties together with \cite[Theorem 1.6]{BM01}. The advantage of our
 proof is that it only uses the definition of Braden-MacPherson sheaves, together with the explicit description of finite intervals of $\MGp$
 given in Section 3.

\subsection{Acknowledgements} I wish to thank Peter Fiebig for encouraging me to look at the stable moment graph and for sharing
his ideas with me. I would like to acknowledge Rocco Chiriv\`i, whose suggestions were always very useful. I owe many thanks to
 Vladimir Shchigolev too, for his careful reading of a very preliminary version of this paper.

\section{Affine Kac-Moody algebras} We want to fix some notation relative to affine Kac-Moody algebras. The main reference is
 \cite[\S6]{Kac}.

\subsection{Basic notation}
Let us consider a finite-dimensional simple complex Lie algebra $\g$ and $\b\supseteq \h$ a Borel and a Cartan subalgebras. This data uniquely determines
 a set of simple roots $\Pi=\Pi(\b,\h)$ and a root system $\Delta=\Delta_{+} \sqcup \Delta_{-}$, where $\Delta_{+}$, resp. $\Delta_{-}$, denotes the 
set of positive, resp. negative, roots. 

We want now to consider the (untwisted) affine Kac-Moody algebra $\gaff$. As a vector space, it is defined as 
\begin{equation*}
\gaff=\g\otimes_{\CC}\CC[t^{\pm1}]\oplus \CC c\oplus \CC D.
\end{equation*}
If $\kappa:\g\times \g\rightarrow \CC$ denotes the Killing form of $\g$, then the commutation relations are as follows
\begin{align*}
 [c,\gaff] & = 0, \\
[D,x\otimes t^n] &  = n x\otimes t^n, \\
[x\otimes t^n, y\otimes t^m] & = [x,y]\otimes t^{m+n}+
n\delta_{m,-n}\kappa(x,y) c
\end{align*}
for $x,y\in \g$, $n,m\in \ZZ$.

Moreover, there is a Cartan subalgebra of $\gaff$,\ which corresponds to $\h$, namely
\begin{equation*}
\haff  = \h\oplus\CC c\oplus \CC D
\end{equation*}
and we also have
\begin{equation*}
\haff^{*}  = \h^{*}\oplus\CC \delta \oplus \CC \Lambda_0,
\end{equation*}
where, if $\langle \cdot,\cdot \rangle: \haff\times \haff^{*}\rightarrow \CC$ is the canonical pairing, it holds
\begin{align*}
\langle\delta,\haff\oplus \CC c\rangle & = \{0\}, \\ 
\langle\delta,D\rangle & = 1, \\
\langle\Lambda_0,\haff\oplus \CC D\rangle & = \{0\},\\
\langle\Lambda_0,c\rangle & = 1.
\end{align*}

Thus the set of real roots $\Dre$ of $\gaff$ has a nice description in terms of the root system of $\g$. Indeed, we have 
(cf. \cite[Proposition 6.3]{Kac})
\begin{equation*}
 \Dre=\left\{ \alpha+n\delta \mid \alpha\in \Delta, \, n\in\ZZ \right\}
\end{equation*}
while the set of positive real roots is given by
\begin{equation*}
 \Dre_{+}=\left\{ \alpha+n\delta \mid \alpha\in \Delta, \, n\in\ZZ_{>0} \right\}\cup \Delta_{+}.
\end{equation*}

Finally, if $\theta\in \Delta_{+}$ is the highest root, then 
\begin{equation*}
 \widehat{\Pi}=\Pi\cup \{-\theta+\delta\}
\end{equation*}
is the set of (affine) simple roots.

\subsection{The affine Weyl group and its set of alcoves}\label{ssec_WA_alcoves} For any $\alpha\in\Dre$, let us denote by $s_{\alpha}\in GL(\haff^{*})$
 the reflection whose action on $v\in \haff^{*}$ is defined by
\begin{equation*} s_{\alpha}(v)=v-\langle v, \check{\alpha} \rangle \alpha.
\end{equation*}
The affine Weyl group $\WA$ is then generated by the $s_{\alpha}$'s with $\alpha\in \Dre_{+}$, while we may identify 
the finite Weyl group $\W$ with the subgroup of $\WA$ generated by reflections which are indexed by finite positive real  roots.

 Let us denote by 
\begin{equation*}
\SRA=\left\{ s_{\alpha}\mid \alpha\in \widehat{\Pi}\right\}.
 \end{equation*}
the set of \emph{simple reflections} of $\WA$. Then $(\WA, \SRA)$ is a Coxeter system and the set $\T$ of \emph{reflections} of $\WA$ can be also obtained by conjugating $\SRA$, i.e.
\begin{equation*}
\T=\left\{s_{\alpha}\mid \alpha\in \Dre_{+}\right\}=\left\{wsw^{-1}\mid w\in\WA,\, s\in\SRA\right\}
\end{equation*}
Moreover, we denote by 
$\ell:\WA\rightarrow \mathbb{Z}_{\geq 0}$ the length function and by $\leq$ the Bruhat order on $\WA$.

Let $\haff^{*}_{\RR}$ and $\h^{*}_{\RR}$ be the $\RR$-span of $\widehat{\Pi}$ and of $\Pi$, respectively. We shall now recall another realization
 of $\WA$, as group of affine transformations of $\h^{*}_{\RR}$. This is obtained by identifying $\h^{*}_{\RR}$ with the affine space $\haff^{*}_{1}$ mod $\RR \delta$,
 where  
\begin{equation*}\haff^{*}_{1}:=\left\{\lambda\in\h^*_{\mathbb{R}}\,\big\vert \langle\lambda, c \rangle=1\right\}
\end{equation*}
Because $\g$ is symmetrisable, by \cite[Lemma 2.1]{Kac}, there is a bilinear form $(\cdot, \cdot):\gaff\times \gaff \rightarrow \CC$ 
that induces an isomorphism $\nu:\h \rightarrow \h^{*}$ such that we may identify $\check{\al}$ and $\frac{2\al}{(\al,\al)}$. Then, 
the action of $\WA$ on
 $\lambda\in \h^{*}$, is given by 
\begin{equation}\label{eqn_aff_refl_h}
 s_{\alpha+n\delta}(\lambda)=\lambda-\left(\langle\lambda, \check{\alpha}\rangle+\frac{2n}{(\al,\al)}\right)\alpha=s_{\al}(\lambda)-n\check{\al}
\end{equation}

 Denote by $\coQ$ the coroot lattice of $\g$ and  by $T_{\mu}$ the translation by $\mu\in\coQ$, that is the transformation
 defined as $T_{\mu}(\lambda)=\lambda+\mu$ for any $\lambda\in \h^{*}_{\mathbb{R}}$. This is an element of the affine Weyl group, since
 $T_{-n\check{\al}}=s_{\alpha+n\delta}s_{\alpha}$. It is easy to check that for any $w\in \WA$ and  for any $\mu\in\coQ$ we have $wT_{\mu}w^{-1}=T_{w(\mu)}$, so the group  of translations by an element of the
 coroot lattice turns out to be a normal subgroup. A well known fact is that $\WA=\W\ltimes\coQ$ (cf. \cite[Proposition 4.2]{Humph}).

Denote by  
\begin{equation*}H_{\alpha,n}:=\left\{\lambda \in \h^{*}_{\mathbb{R}}\,|\,\langle \lambda,\check{\al}\rangle=
-2\frac{n}{(\al,\al)}\right\}=\left\{\lambda \in \h^{*}_{\mathbb{R}}\,|\,( \lambda,\al)=-n\right\}
\end{equation*}
and observe that the affine reflection $s_{\alpha+n\delta}$ fixes pointwise such a hyperplane.  We call \emph{alcoves} 
the connected components of
 \begin{equation*}\h^{*}_{\mathbb{R}}\setminus \bigcup_{\substack{\alpha+n\delta\in\Dre_{+}}} H_{\alpha,n}
\end{equation*}
and write $\A$ for the set of all alcoves.

 The \emph{fundamental (Weyl) chamber} is
\begin{equation*}\C^+:=\{  \lambda\!\in\!\h_{\mathbb{R}}^{*}\,\,|\,\, \langle \lambda,\check{\al}\rangle>0 \,\,\forall \al\in\Pi \}
\end{equation*} and an element $\lambda\in\C^+$ is called \emph{dominant weight}.
We denote by $\mathcal{A}^+$ the set of all  alcoves contained in $\C^+$.

We state now a 1-1 correspondence between $\WA$ and $\A$ (cf. \cite[Theorem 4.8]{Humph}). In order to do that, we fix an alcove $A^+$, 
that is the unique alcove in $\A^+$ which contains the null vector in its closure. $A^+$ is usually called \textit{the fundamental alcove}
 and it has the property that every element $\lambda\in A^+$ is such that $0<( \lambda, \al) < 1$ for all $\alpha\in\Delta_{+}$ 
(cf. \cite[\S4.3]{Humph}).

The affine Weyl group $\WA$ acts on the left (by (\ref{eqn_aff_refl_h})) simply transitively on $\A$ (cf. \cite[\S 4.5]{Humph}) and so we obtain
\begin{equation}\label{corr}
\begin{array}{ccl}
\WA &\longrightarrow\!\!\!\!\!\!\!\!\!\!^{1\!-\!1}  &\A\\
w          &   \mapsto          & wA^+.\\
\end{array}
\end{equation}

Let us observe that each wall of $A^+$ is fixed by exactly one reflection $s \in \SRA$. We say that such a wall is the $s$-wall 
of $A^+$. In general every $A\in \A$ has one and only one wall in the $\WA$-orbit of the $s$-wall of $A^+$. This is called
 \emph{$s$-wall} of $A$.

The affine Weyl group acts on itself by right multiplication, so it makes sense to define a right action of $\WA$ on $\A$.
 It is of course enough to define such an action for the generators of the group. Thus for each alcove $A$ let $As$ be the unique
 alcove having in common with $A$ the $s$-wall.

\section{Moment graphs}

Let us fix once and for all a local ring $k$ such that $2$ is an invertible element. In the first part of this section we recall the 
definition of the category of $k$-moment graphs on a lattice, while in the second part we 
 focus on certain moment graphs, which are relevant for the representation theory of affine Kac-Moody algebras.

\subsection{Category of $k$-moment graphs on a lattice}

\begin{defin}[cf. \cite{Fie07b}]\label{Def_MG} Let  $Y$ be a lattice of finite rank. A \emph{moment graph on the lattice} $Y$ is given  by $(\V,\E, \trianglelefteq, l)$, where:
\item[(MG1)] $(\V,\E)$ is a directed graph without directed cycles nor multiple edges,
\item[(MG2)] $ \trianglelefteq$ is a partial order on $\V$ such that if $x,y\in \V$ and $E:x\rightarrow y\,\in \E$, then $x \trianglelefteq y$,
\item[(MG3)] $l:\E\rightarrow Y\setminus \!\{0\}\,$  is a map  called the \emph{label function}.
\end{defin}

Following Fiebig's notation (cf. \cite{Fie07b}), we will write $x-\!\!\!-\!\!\!-y$ if we are forgetting about the orientation of 
the edge.

For any lattice $Y$ of finite rank, let us denote the extended lattice by  $Y_k=Y\otimes_{\mathbb{Z}} k$.

\begin{defin}[cf. \cite{L11}]\label{Def_kMG} Let $\MG$ be  a moment graph on the lattice $Y$. We say that $\MG$ is a
 \emph{$k$-moment graph} on $Y$ if all labels are non-zero in $Y_k$, that is the image $\overline{l(E)}:=l(E)\otimes 1$ is a
 non-zero element of $Y_k$.
\end{defin}

By abuse of notation, we will often denote by $l(E)$ also its image in $Y_k$.

\begin{defin}[cf. \cite{FW}]\label{Def_GKM} The pair $(\MG, k)$ is called a \emph{ $GKM$-pair}  if all pairs $E_1$, $E_2$ of 
distinct edges containing a common vertex are such that $k\cdot l(E_1)\cap k\cdot l(E_2)=\{0\}$.
\end{defin}

 This  property is very important and, in the next sections, it  will give a restriction on the ring $k$.

\subsubsection{Morphisms of $k$-moment graphs}\label{sec_morphMG}

\begin{defin}[cf. \cite{L11}]\label{Def_morphMG}A morphism between two $k$-moment graphs 
\begin{equation*}f:(\V,\E, \trianglelefteq, l)\rightarrow (\V',\E', \trianglelefteq', l')
\end{equation*} is given by $(f_{\V}, \{f_{l,x}\}_{x\in \V})$, where
\item[(MORPH1)] $f_{\V}:\V \rightarrow \V'$ is any  map of posets such that, if $x-\!\!\!-\!\!\!-y\in\E$, then 
either $f_{\V}(x)-\!\!\!-\!\!\!-f_{\V}(y)\in \E' $, or $f_{\V}(x)=f_{\V}(y)$.

For an edge $E:x\rightarrow y\in \E$ such that $f_{\V}(x)\neq f_{\V}(y)$, we will denote
 $f_{\E}(E):= f_{\V}(x)-\!\!\!-\!\!\!-f_{\V}(y)$.
\item[(MORPH2)] For all $x\in \V$, $f_{l,x}:Y_k\rightarrow Y_k\in\text{Aut}_k(Y_k)$  is such that,
 if $E:x-\!\!\!-\!\!\!-y\in \E$ and $f_{\V}(x)\neq f_{\V}(y)$, the following two conditions are verified:
\hspace{5mm}\item[(MORPH2a)] $f_{l,x}(l(E))=h\cdot l'(f_{\E}(E)) $, for some $h\in k^{\times}$
\hspace{5mm}\item[(MORPH2b)]  $\pi \circ f_{l,x} =\pi \circ f_{l,y}$,
where $\pi$ is the canonical quotient map $\pi:Y_k\rightarrow Y_k/l'(f_{\E}(E))Y_k$.
\end{defin}

Given two morphisms of $k$-moment graphs
\begin{equation*}f:\MG=(\V, \E, \tle, l)\rightarrow \MG'=(\V', \E', \tle', l')
\end{equation*}
\begin{equation*}
 g:\MG'\rightarrow \MG''=(\V'', \E'', \tle'', l'')
\end{equation*}
 We may set 
\begin{equation*}g\circ f:=(g_{\V'}\circ f_{\V}, \{g_{l', f_{\V}(x)}\circ f_{l,x}\}_{x\in \V}).
\end{equation*}

It is an easy check that this composition is well--defined and associative.

\begin{defin}\label{Def_catkMG}  We denote by $\MGYk$ the category whose objects are the $k$-moment graphs on $Y$ and corresponding
 morphisms. 
\end{defin}

\subsubsection{Isomorphisms of $k$-moment graphs}  In \cite{L11}, we gave an abstract definition of isomorphisms but we 
did not notice that the notion we were considering actually coincided with the one of invertible morphisms. 
The following lemma proves this fact.

\begin{lem}\label{Lem_IsoMG}Let us consider $\MG=(\V, \E, \tle, l), \MG'=(\V', \E', \tle', l')\in\MGYk$, and a morphism between them $f=(f_{\V}, \{f_{l,x}\}_{x\in\V}) \in\Hom_{\MGYk}(\MG, \MG')$. $f$ is an isomorphism in the categorical sense if and only if the following two conditions hold:
\item[(ISO1)] $f_{\V}$ is bijective
\item[(ISO2)] for all $u\rightarrow w\in \E'$, there exists exactly one $x\rightarrow y\in \E$ such that $f_{\V}(x)=u$ and $ f_{\V}(y)=w $. 
\end{lem}
\proof To begin with we show that a morphism satisfying \textit{(ISO1)} and \textit{(ISO2)} is invertible. 
Denote by $f^{-1}:=(f'_{\V'}, \{f'_{l',u}\}_{u\in\V'})$, where we set $f_{\V'}':=f_{\V}^{-1}$ and $f_{l',u}':=f_{l,f_{\V}^{-1}(u)}^{-1}$.
 We have to verify that $f^{-1}$ is well-defined, that is we have to check conditions \textit{(MORPH2a)} and \textit{(MORPH2b)}. 
Suppose there exists an edge $F:u\rightarrow w\in \E'$, then, by (ISO2), there is an edge $E:x\rightarrow y\in \E$ 
such that $f_{\V}(x)=u$ and $f_{\V}(y)=w$. Since $f$ satisfies \textit{(MORPH2a)}, $f_{l,x}(l(E))=h\cdot l'(F)$ for $h\in k^{\times}$ and we get
\begin{equation*}\begin{array}{c}
f_{l',u}'(l'(F))=f_{l,f_{\V}^{-1}(u)}^{-1}(l'(F))=f_{l,x}^{-1}(l'(F))=h^{-1}\cdot l(E)
\end{array}
\end{equation*}
Now, let $\mu\in Y_K$ and take $\lambda:=f_{l,y}^{-1}(\mu)$. By \textit{(MORPH2a)}, $\mu=f_{l,y}(\lambda)=f_{l,x}(\lambda)+r\cdot l'(F)$ for some $r\in k$. It follows
\begin{align*}
f_{l',u}'(\mu)&=f_{l,x}^{-1}(\mu))\\
&=f_{l,x}^{-1}(f_{l,x}(\lambda)+rl'(F))\\
&=\lambda+r\cdot f_{l,x}^{-1}(l'(F))\\
&=f_{l,y}^{-1}(\mu)+ r\cdot h^{-1}\cdot l(E)\\
&=f_{l',w}'(\mu)+ r'\cdot l(E)
\end{align*}

Suppose $f$ is an isomorphism. If  \textit{(ISO1)} is not satisfied,  then $f_{\V}$, and hence
  $f$, is not  invertible. Moreover, \textit{(ISO1)} implies that  for all $u\rightarrow v\in \E'$, there exists at most 
one $x\rightarrow y\in \E$ such that $f_{\V}(x)=u$ and $ f_{\V}(y)=v$ (otherwise $f_{\V}$ would not be injective).
 Now, let $f$ be the following homomorphism (we do not care about the $f_{l,x}$'s):
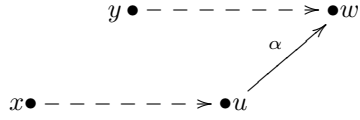
\begin{displaymath}
\xymatrix{   &  \!\!y\, \bullet \ar@{-->}[rr] & & \bullet w \\
x\bullet  \ar@{-->}[rr] &  & \bullet  \ar[ru]^{\alpha} u \\
}
\end{displaymath}
 Condition \textit{(ISO1)} holds, but $f$ is not invertible, since $f_{\V}^{-1}(u)\neq f_{\V}^{-1}(w)$ but $f_{\V}^{-1}(u)-\!\!\!-\!\!\!-f_{\V}^{-1}(w)\not\in \E$ (this contradicts \textit{(MORPH1)}). 
\endproof

Isomorphisms between $k$-moment graphs were an important tool in \cite{L11} in order to obtain a categorical 
analogue of certain equalities between Kazhdan-Lusztig polynomials.

\subsection{Bruhat moment graphs}\label{bruhatMG}

From now on we fix $\gaff\supseteq \baff \supseteq \haff$ and keep the notation of Section 2. For any subset $J\subset \SRA$ we define 
the  \emph{(affine) Bruhat moment graph}
 \begin{equation*}\MGJ=\MGa(\gaff\supseteq \baff\supseteq \haff,J)=(\V, \E, \leq, l)
 \end{equation*}
 associated to the data $(\gaff\supseteq \baff\supseteq \haff, J)$ as follows.

\begin{defin}\label{def_parBruhatMG} $\MGJ$ is  the moment graph on the affine coroot lattice $\widehat{\coQ}$ which is given by
\item[(i)] $\V=\WA^J$, stand $\WA^J$ for the set of minimal representatives of the equivalence classes of $\WA/\langle J \rangle$
\item[(ii)] $\E=\left\{ x\rightarrow y\,\vert\, x< y\,,\, \exists\, \alpha\!\in\!\Dre_+, \, \exists  w\in\langle J \rangle 
\text{ such that } ywx^{-1}=s_{\alpha} \right\}$
\item[(iii)] $l(x\rightarrow s_{\alpha}xw^{-1}):=\check{\alpha}$
\end{defin}

If $J=\emptyset$, we will simply write $\MGa$ instead of $\MGa^{\emptyset}$ and call it the \emph{regular Bruhat graph} of $\gaff$, while, 
if $J=\SR:=\{ s_{\alpha}\mid \alpha\in \Pi\}$, we will denote $\MGp:=\MGa^{\SR}$ and call it the \emph{parabolic Bruhat graph}.

\begin{expl}\label{expl_sl2_MGa}Let $\gaff=\widehat{\mathfrak{sl}_2}$. In this case, 
$\Dre_{+}=\{\pm\alpha+n\delta\,\vert\, n\in\mathbb{Z}_{>0}\}\cup \{\al\}$, where $\alpha$ is the (unique) positive root of 
$\mathfrak{sl}_2$ and $(\al,\al)=2$.  The corresponding regular Bruhat graph $\MGa$ is an infinite graph, whose vertices are given by 
the words in two letters ($s_1:=s_{\alpha}$ and $s_0$) without repetitions. Two elements are connected if and only if the difference
 between  their lengths is odd and in this case the edge is oriented from the shorter to the longer one. Thanks to the 
correspondence (\ref{corr}), we  may identify the set of vertices with the set of alcoves of $\gaff$. If we restrict $\MGa$  
to the interval $[A^+, s_1s_0s_1A^{+}]$, we get the following graph
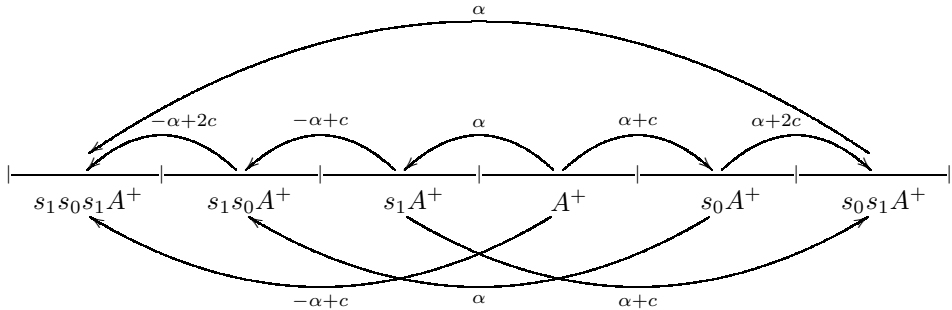
\begin{displaymath}
\begin{xymatrix}{
\ar@{}[rr]|{\,\,\,|}&
\ar@{-}[rrrr]|{|}\txt<-10pc>{\\ \\ $\hspace{4mm}s_1s_0s_1A^+$}&&\ar@{-}[rrrr]|{|}\txt<-10pc>{\\ \\ $\hspace{6mm}s_1s_0A^+$}&\ar@/_1.25pc/[ll]_(.41){-\al+2c}&\ar@{-}[rrrr]|{|}\txt<-10pc>{\\ \\ $\hspace{6mm}\,\,\,\,s_1A^+$}&\ar@/_1.25pc/[ll]_{-\al+c}\ar@/_2.35pc/[rrrrrr]<-3.3ex>_{\al+c}&\ar@{-}[rrrr]|{|}\txt<-10pc>{\\ \\ $\hspace{6mm}\,\,\,\,\,\,A^+$}&\ar@/_1.25pc/[ll]_{\al}\ar@/^1.25pc/[rr]^{\al+c}\ar@/^2.35pc/[llllll]<3.3ex>^{-\al+c}&\ar@{-}[rrrr]|{|}\txt<-10pc>{\\ \\ $\hspace{6mm}\,\,\,\,s_0A^+$}&\ar@/^1.25pc/[rr]^(.41){\al+2c}\ar@/^2.35pc/[llllll]<3.3ex>^{\al}&\ar@{}[rrrr]|{\!\!\!|}\txt<-10pc>{\\ \\ $\hspace{6mm}s_0s_1A^+$}&\ar@/_4.3pc/[llllllllll]<-1.5ex>_{\al}&&&
}
\end{xymatrix}
\end{displaymath}
\end{expl}

\subsection{The periodic moment graph}\label{ssec_perMG}
 We want now to associate another moment graph with $\gaff$. In order to do this, we need to recall the notion of generic order on 
the set of alcoves.

\subsubsection{Two partial orders on the set of alcoves}\label{ssec_2orders}

Following \cite{Lu80}, we provide the set of alcoves with two partial orders.

First of all, the Bruhat order on $\WA$ induces a partial order on  $\A$. Indeed, for all alcoves 
$A,B\in \A $ with $A=xA^+,\, B=yA^+$, $x,y\in \WA$ we may set
$$A\leq B\,\, \iff\,\, x\leq y.$$
We still call it \emph{Bruhat order}. 

 We observe that in general if we look at two alcoves it is not obvious at all if they are comparable with respect to the Bruhat order 
without knowing the corresponding elements in $\WA$.

Next, we recall Lusztig's definition of a nicer partial order $\preccurlyeq$ on  $\A$, in the sense that for all pair of alcoves
 we will be able to say whether they are comparable and, in case, to establish which one is the bigger one. 

Each $H\in \bigcup_{\substack{\alpha+n\delta\in \Dre_{+}}} H_{\alpha,n}$ divides $\h^{*}_{\RR}$ in two half spaces:
 $\h^{*}_{\RR}=H^+ \sqcup H \sqcup H^-$, where $H^+$ is the half space that intersects every translate of $\C^+$.
 Let $A\in \A$, if $H$ is the reflecting hyperplane between $A$ and $As$, $s\in\SRA$, we consider the partial order generated by 
 \begin{equation*}A\preccurlyeq As \,\,\,\,\text{ if }A\in H^-.
\end{equation*}

Notice that it is not clear in general how $\leq $ and $\preccurlyeq$ are related.
Let us denote by $\coX$ the \emph{lattice of  (finite) integral coweights}, that is (again under the identification of $\h_{\RR}$ and $\h^{*}_{\RR}$)
\begin{equation}\coX:=\left\{\lambda\in \h^{*}_{\RR}\mid ( \lambda, \alpha) \in \ZZ \,\, \forall \alpha\in \Delta\right\}.
\end{equation}

\begin{prop}[\cite{Soe97}, Claim 4.4] \label{gen_bruhat}Far enough inside $\A^+$, $\leq$ and $\preccurlyeq$ coincide, that is for all
 $\lambda\in \check{X}\cap \C^+$, $A,B\in \A$ the following are equivalent:
\begin{enumerate}
\item $A\preccurlyeq B$;
\item $n\lambda+A\leq n\lambda +B$ for $n>>0$.
\end{enumerate}
\end{prop}

Because of this result Lusztig called $\preccurlyeq$ \emph{generic Bruhat order}.  Remark that $\preccurlyeq$ is invariant 
under translation by finite coweights.

\begin{defin}The \emph{periodic moment graph} $\MGper=\MGper(\gaff \supseteq \baff \supseteq \haff)=(\V, \E, \preccurlyeq, l)$
 is a moment graph on $\widehat{\check{Q}}$ and it is given by
\item[(i)] $\V=\A$, the set of alcoves of $\WA$
\item[(ii)] $\E=\left\{ xA^+\rightarrow yA^+\mid xA^+\preccurlyeq yA^+\,,\, \exists\, \alpha\!\in\!\Dre_{+} \text{ such that } y=s_{\alpha}x \right\}$
\item[(iii)] $l(xA^+\rightarrow s_{\alpha}xA^+):=\check{\alpha}$
\end{defin}

\begin{rem}Let us observe that we identified $\WA$ and $\A$ by (\ref{corr}) and therefore $\MGa$ and $\MGper$ coincide as labeled 
 \emph{unoriented} graphs, as we can see by comparing Example \ref{expl_sl2_MGa} with the following one.
\end{rem}

\begin{expl}Let $\gaff=\widehat{\mathfrak{sl}_2}$. If we restrict the corresponding periodic moment  graph  to the interval
 $[s_1s_0s_1A^{+}, s_0s_1A^+]$, we get the following moment graph.
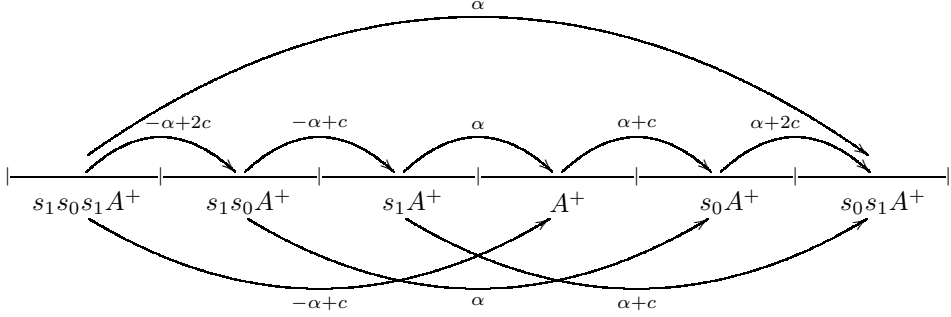
\begin{displaymath}
\begin{xymatrix}{
\ar@{}[rr]|{\,\,\,|}&
\ar@{-}[rrrr]|{|}\txt<-10pc>{\\ \\ $\hspace{4mm}s_1s_0s_1A^+$}&\ar@/^4.5pc/[rrrrrrrrrr]<1.5ex>^{\al}\ar@/^1.25pc/[rr]^(.57){-\al+2c}\ar@/_2.35pc/[rrrrrr]<-3.3ex>_{-\al+c}&\ar@{-}[rrrr]|{|}\txt<-10pc>{\\ \\ $\hspace{6mm}s_1s_0A^+$}&\ar@/^1.25pc/[rr]^{-\al+c}\ar@/_2.35pc/[rrrrrr]<-3.3ex>_{\al}&\ar@{-}[rrrr]|{|}\txt<-10pc>{\\ \\ $\hspace{6mm}\,\,\,\,s_1A^+$}&\ar@/_2.35pc/[rrrrrr]<-3.3ex>_{\al+c}\ar@/^1.25pc/[rr]^{\al}&\ar@{-}[rrrr]|{|}\txt<-10pc>{\\ \\ $\hspace{6mm}\,\,\,\,\,\,A^+$}&\ar@/^1.25pc/[rr]^{\al+c}&\ar@{-}[rrrr]|{|}\txt<-10pc>{\\ \\ $\hspace{6mm}\,\,\,\,s_0A^+$}&\ar@/^1.25pc/[rr]^(.41){\al+2c}&\ar@{}[rrrr]|{\!\!\!|}\txt<-10pc>{\\ \\ $\hspace{6mm}s_0s_1A^+$}&&&&
}
\end{xymatrix}
\end{displaymath}
\end{expl}

\section{Finite intervals of $\MGp$ far enough in the fundamental chamber}\label{ssec_stabMG}

This section is devoted to the description of certain intervals of the parabolic Bruhat graph $\MGp$, that is the Bruhat graph 
corresponding to the data $(\gaff\supseteq \baff\supseteq \haff, \SR)$.

\subsection{Two descriptions} There are actually two descriptions of the  graph $\MGp$: one identifies the set of vertices
 with the finite coroot lattice $\coQ$, while the other identifies the 
set of vertices with $\A^+$, the set of alcoves in the fundamental chamber. 

As the affine Weyl group acts on $\h^{*}_{\RR}$, we can consider the $\WA$-orbit of $0$, that is the finite coroot lattice $\coQ$. Moreover
  $\text{Stab}_{\W}(0)=\W=\langle\SR \rangle$, the finite Weyl group, and hence $\WA/ \W$ is in bijection with the coroot lattice via the 
mapping $w\rightarrow w(0)$. Clearly, for any pair of minimal length representatives $x,y\in \WA/ \W$, there exist $w\in \W$ and 
$\alpha+n\delta \in \Dre_{+}$ such that $y=s_{\alpha+n\delta}x w$ if and only if $y(0)=s_{\alpha+ n \delta}x(0)$, 
that is $y(0)=x(0)-n\check{\alpha}$.

On the other hand,  $\W\setminus \WA$ is clearly  in bijection with $\WA/\W$  via the mapping $x\mapsto x^{-1}$. The set of minimal representatives for the equivalence classes, under the correspondence (\ref{corr}), is given by the
 set $\A^+$ of the alcoves in the fundamental chamber. Moreover, we will connect  $xA^+,yA^+\in\A^+$ if and only if  there exist
 an element of the finite Weyl group $w\in\W$ and an affine positive root $\alpha\in\Dre_{+}$ such that $x=wys_{\alpha}$, that is
 $x^{-1}=s_{\alpha}y^{-1}w^{-1}$.

\begin{expl}Let $\gaff=\widehat{\mathfrak{sl}_3}$ and let $\W^{\text{par}}$ be the set of minimal length representatives of $\WA/ \W$.
 Let us consider the interval $[e,s_{\beta}s_{\al}s_{\beta}s_{0}]\subset \W^{\text{par}}$ then the two descriptions 
of $\MGp$ are as follows (we omit the labels).
\item[(i)]Description via the finite coroot lattice
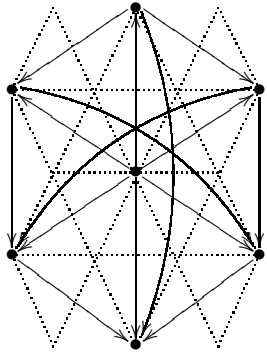
\begin{displaymath}
\begin{xy}0;/r.26pc/:
(-15,10\halfrootthree);(15,10\halfrootthree)**\dir{.};
(-10,20\halfrootthree);(10,20\halfrootthree)**\dir{.};
(-15,30\halfrootthree);(15,30\halfrootthree)**\dir{.};
(-10,0);(-15,10\halfrootthree)**\dir{.};
(0,0);(-15,30\halfrootthree)**\dir{.};
(10,0);(-10,40\halfrootthree)**\dir{.};
(10,0);(15,10\halfrootthree)**\dir{.};
(0,0);(15,30\halfrootthree)**\dir{.};
(-10,0);(10,40\halfrootthree)**\dir{.};
(-15,10\halfrootthree);(0,40\halfrootthree)**\dir{.};
(-15,30\halfrootthree);(-10,40\halfrootthree)**\dir{.};
(15,10\halfrootthree);(0,40\halfrootthree)**\dir{.};
(15,30\halfrootthree);(10,40\halfrootthree)**\dir{.};
(0,20\halfrootthree)*{\bullet}="A";
(0,40\halfrootthree)*{\bullet}="A1";
(-15,30\halfrootthree)*{\bullet}="A2";
(15,30\halfrootthree)*{\bullet}="A3";
(-15,10\halfrootthree)*{\bullet}="A4";
(15,10\halfrootthree)*{\bullet}="A5";
(0,0)*{\bullet}="B";
"A",{\ar"A1"};
"A",{\ar"A5"};
"A",{\ar"A4"};
"A",{\ar"A2"};
"A",{\ar"A3"};
"A",{\ar"B"};
"A1"+<-.25pc,-.05pc>,{\ar"A2"+<.2pc,.22pc>};
"A1"+<.25pc,-.05pc>,{\ar"A3"+<-.2pc,.22pc>};
"A1"+<.18pc,-.15pc>,{\ar@/^12pt/"B"+<.2pc,.3pc>};
"A2"+<.27pc,.08pc>,{\ar@/^12pt/"A5"+<-.2pc,.3pc>};
"A3"+<-.27pc,.08pc>,{\ar@/_12pt/"A4"+<.2pc,.3pc>};
"A4"+<.2pc,-.15pc>,{\ar"B"+<-.25pc,.1pc>};
"A5"+<-.2pc,-.15pc>,{\ar"B"+<.25pc,.1pc>};
"A3"+<.0pc,-.25pc>,{\ar"A5"+<.0pc,.25pc>};
"A2"+<.0pc,-.25pc>,{\ar"A4"+<.0pc,.25pc>};
\end{xy}
\end{displaymath}
\item[(ii)] Description via $\A^{+}$, the set of alcoves in $\C^+$
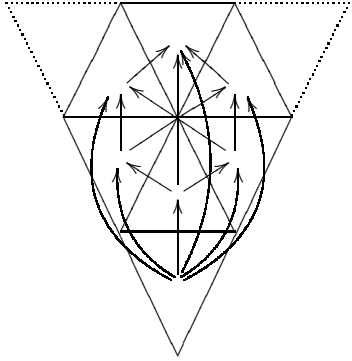
\begin{displaymath}
\begin{xy}0;/r.36pc/:
(-10,20\halfrootthree);(-15,30\halfrootthree)**\dir{.};
(-15,30\halfrootthree);(-5,30\halfrootthree)**\dir{.};
(10,20\halfrootthree);(15,30\halfrootthree)**\dir{.};
(15,30\halfrootthree);(5,30\halfrootthree)**\dir{.};
  (0,0);(-10,20\halfrootthree)**\dir{-};
   (0,0);(10,20\halfrootthree)**\dir{-};
    (-10,20\halfrootthree);(-5,30\halfrootthree)**\dir{-};
  (10,20\halfrootthree);(5,30\halfrootthree)**\dir{-};
   (-10,20\halfrootthree);(10,20\halfrootthree)**\dir{-};
   (-5,10\halfrootthree);(5,10\halfrootthree)**\dir{-};
(-5,30\halfrootthree);(5,30\halfrootthree)**\dir{-};
(-5,30\halfrootthree);(5,10\halfrootthree)**\dir{-};
(5,30\halfrootthree);(-5,10\halfrootthree)**\dir{-};
(0,6\halfrootthree)="A";(0,13\halfrootthree)="A1";
(-5,16\halfrootthree)="A2";(5,16\halfrootthree)="A3";
(-5,22\halfrootthree)="A5";(5,22\halfrootthree)="A4";
(0,26\halfrootthree)="B";
"A"+<0pc,.1pc>,{\ar"A1"+<0pc,-.1pc>};
"A"+<-.2pc,-.1pc>,{\ar@/^17pt/"A5"+<-.4pc,-.1pc>};
"A"+<.2pc,-.1pc>,{\ar@/_17pt/"A4"+<.4pc,-.1pc>};
"A"+<-.1pc,.0pc>,{\ar@/^7pt/"A2"+<-.1pc,-.2pc>};
"A"+<.1pc,.pc>,{\ar@/_7pt/"A3"+<.1pc,-.2pc>};
"A"+<.13pc,.15pc>,{\ar@/_11pt/"B"+<.1pc,-.08pc>};
"A1"+<-.2pc,.2pc>,{\ar"A2"+<.2pc,0pc>};
"A1"+<.2pc,.2pc>,{\ar"A3"+<-.2pc,0pc>};
"A1"+<0pc,.4pc>,{\ar"B"+<0pc,-.2pc>};
"A2"+<0pc,.4pc>,{\ar"A5"+<0pc,.pc>};
"A3"+<0pc,.4pc>,{\ar"A4"+<0pc,.pc>};
"A5"+<.2pc,.35pc>,{\ar"B"+<-.27pc,.1pc>};
"A4"+<-.2pc,.35pc>,{\ar"B"+<.27pc,.1pc>};
"A3"+<-.3pc,.4pc>,{\ar"A5"+<.3pc,.23pc>};
"A2"+<.3pc,.4pc>,{\ar"A4"+<-.3pc,.23pc>};
\end{xy}
\end{displaymath}

\end{expl}

As we can see in the previous example, in the description of $\MGp$ via the alcoves in the fundamental chamber, the set of edges 
seems to have a very complex structure, while in the other one the order on the set of vertices is hard to understand.
 Since we are interested in the study of intervals, the description via the finite coroot lattice turns out to be not that useful for 
our purposes, unless $\gaff=\widehat{\mathfrak{sl}_2}$.

\subsubsection{The $\widehat{\mathfrak{sl}_2}$ case}\label{sssec_sl2PMG}

If $\gaff=\widehat{\mathfrak{sl}_2}$, it is actually possible to give a very explicit description of $\MGp$.  In this
 case we  may identify the finite root lattice with the finite coroot lattice and then the set of vertices is $\V=\mathbb{Z}\al$. 
For any pair  $n, m\in\mathbb{Z}$, it is immediate to check that 
\begin{equation}\label{eqn_A1_edges}s_{\alpha-(n+m)\delta}(n\alpha)=m\alpha,
\end{equation}
then $\MGp$ is a fully connected graph. Notice that, even if (\ref{eqn_A1_edges}) holds for any pair of integers $n$ and $m$, we
 do not allow loops, so $n\neq m$ always. It follows
\begin{equation}\label{eqn_A1_label}l(n\al-\!\!\!-\!\!\!-m\al)=\left\{\begin{array}{ll}
-\al+(n+m)c&\text{if }n+m\geq 0\\
\al-(n+m)c&\text{if }n+m< 0
\end{array}
\right.
\end{equation}

Finally, observe that $\al=s_{0}(0)$ and $-\al=s_{\al}s_0(0)$; so, for any pair of $n\neq m\in\mathbb{Z}$, $n\al<m\al$ if and only if
either $|n|<|m|$ or $n=-m>0$.

\begin{expl}The interval $[0,-2\al]$ of $\MGp$ looks like in the following picture
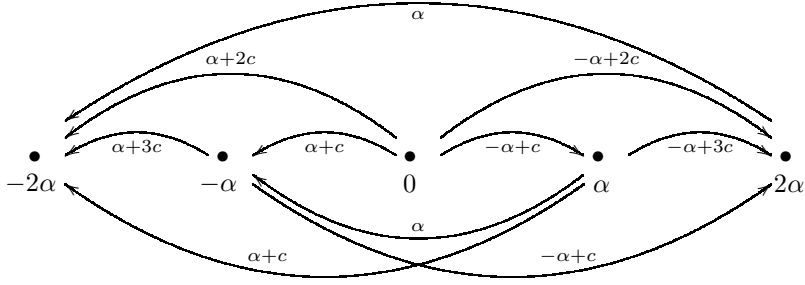
\begin{displaymath}
\begin{xymatrix}{
\bullet\txt<6pt>{\\ \\ \hspace{-5mm}$-2\al$}
&&\bullet\ar@/_1.25pc/@< 6pt>[ll]^{\al+3c}\txt<6pt>{\\ \\ \hspace{-5mm}$-\al$}\ar@/_3.45pc/@<-1ex>[rrrrrr]^(.61){-\al+c}&&
\bullet\ar@/^1.25pc/@<-6pt>[rr]_{-\al+c}\txt<6pt>{\\ \\ \hspace{-4mm}$0$}\ar@/_1.25pc/@< 6pt>[ll]^{\al+c}\ar@/_2.6pc/[llll]_{\al+2c}
\ar@/^2.6pc/[rrrr]^{-\al+2c}&& \bullet\txt<6pt>{\\ \\ \hspace{-3mm}$\al$}\ar@/^1.25pc/@<-6pt>[rr]_{-\al+3c}
\ar@/^2.6pc/[llll]_{\al}\ar@/^3.45pc/@<1ex>[llllll]_(.61){\al+c}&&
\bullet\txt<6pt>{\\ \\ \hspace{-3mm}$2\al$}\ar@/_4.2pc/@<-1.75ex>[llllllll]^{\al}
}
\end{xymatrix}
\end{displaymath}
\end{expl}

 The second part of this section is devoted to showing that finite intervals of $\MGp$ ``far enough" in $\C^+$ have surprisingly 
a very regular structure.

\subsection{Nice behaviour of finite intervals of $\MGp$} In this paragraph, we will consider only the description of $\MGp$ 
in which the set of vertices coincides with $\A^+$.

\begin{defin}Let $\lambda, \mu\in\C^+$. We say that 
\item[(i)]$\lambda$ is  \emph{strongly linked} to $\mu$ if $\lambda=\mu+x\al$, for some $x\in\RR$ and $\al\in\Delta_{+}$,
\item[(ii)]$\lambda$ is \emph{linked} to $\mu$ if $\lambda=w(\mu+n \al)$,  for some $n\in\mathbb{R}$, $\al\in\Delta_{+}$ and
 $w\in\W$.
\end{defin}
 Remark that the fundamental chamber $\C^+$ is a fundamental domain with respect to the left  action of the finite Weyl group 
(cf. \cite[\S 1.12]{Humph}), so the element in point \textit{(ii)} is unique.

\begin{prop}\label{prop_strlinked} There exists a $K>0$, depending only on the root system $\Delta$, such that if 
$\lambda\in\C^+$ and $d_{\lambda}$ is  the minimum of distances from $\lambda$  to the borders of $\C^+$, then all
  $\mu\in\C^+$ linked to $\lambda$ and such that $|\lambda-\mu|< K\cdot d_{\lambda}$ are strongly linked to $\lambda$.
\end{prop}
\proof For any  $\lambda\in \C^+$ and any finite positive root $\alpha\in\Delta_{+}$ we denote by $r_{\lambda,\al}$ the line 
$\{\lambda+\al x\,|\,x\in\RR\}\subseteq \h^{+}_{\RR}$. It is clear that the set of finite dominant weights strongly 
linked to $\lambda$ corresponds to $(\bigcup_{\al\in\Delta_{+}} r_{\lambda,\al})\bigcap \C^+$.
On the other hand, we may describe the set of $\mu\in\C^+$ linked to $\lambda$ as follows. 
Fix $\al \in \Delta_{+}$ and consider the line $r_{\lambda,\al}$. Each time that such a line hits a wall of $\C^+$, it reflects 
off the wall and goes on this way. Let us denote by $\widetilde{r}_{\lambda,\al}$ the piecewise linear path inside of $\C^+$ so obtained. 
Now $\bigcup_{\al\in \Delta_{+}} \widetilde{r}_{\lambda, \al}$ is the set of finite dominant weights linked to $\lambda$.

Thus  it is enough to show that there exists a  $K>0$ such that if $\mu\in\widetilde{r}_{\lambda,\al}$ and
 $|\lambda-\mu|< K\cdot d_{\lambda}$, then $\mu\in r_{\lambda,\al}$. Notice that the finite Weyl group
 acts on $\h^{*}_{\RR}$ as a group of orthogonal transformations, hence we may reduce to show that for
 all $w\in\W\setminus \{e,s_{\alpha}\}$, the distance of the weight $w(\lambda)$ from the line $r_{\lambda,\al}$
is not less than $K\cdot d_{\lambda}$. Moreover, one may think of this reduction as an ``unfolding"  back
  $\widetilde{r}_{\lambda,\al}$ to $ r_{\lambda,\al}$ and considering the conjugates of $\lambda$ instead of $\lambda$.

Since the distance of $w(\lambda)$ from the line  $r_{\lambda, \al}$ is the minimum of the distances of
 $w(\lambda)$ from $\lambda+x\al$ for $x\in\mathbb{R}$, we have to show that $|\lambda-x\al-w(\lambda)|^2\geq K^2 d_{\lambda}^2 $ for all $x\in \mathbb{R}$.  Computing the square norm, and denoting $\lambda^w:=\lambda-w(\lambda)$, we have:
\begin{equation*}
 |\al|^2x^2+2(\lambda^w,\al)x+|\lambda^w|^2-K^2d_{\lambda}^2\geq 0\,\,\,\,\,\,\,\forall x\in\mathbb{R}
\end{equation*}
Hence this is equivalent to showing that the discriminant $D^w=(\lambda^w,\al)^2-|\al|^2|\lambda^w|^2+|\al|^2K^2d_{\lambda}^2\leq 0$.

To start with, let us notice that $D^{s_{\alpha}w}=D^{w}$, since $\lambda^{s_{\alpha}w}=\lambda-w(\lambda)+\langle w(\lambda),\check{\al}\rangle \al=\lambda^w+\langle w(\lambda),\check{\al}\rangle \al$, hence:
\begin{equation*}
\begin{array}{ccl}
D^{s_{\alpha}w}&=&(\lambda^{s_{\alpha}w},\al)^2-|\al|^2|\lambda^{s_{\alpha}w}|^2+|\al|^2K^2d_{\lambda}^2=\\
&=&(\lambda^w+\langle w(\lambda),\check{\al}\rangle\al,\al)^2-|\al|^2(\lambda^w+\langle w(\lambda),\check{\al}\rangle \al,\lambda^w+\langle w(\lambda), \check{\al}\rangle\al)+| \al|^2K^2d_{\lambda}^2=\\
&=&(\lambda^w,\al)^2+2\langle w(\lambda),\check{\al}\rangle|\al|^2(\lambda^w,\al)+\langle w(\lambda),\check{\al}\rangle^2|\al|^4+\\
&-&|\al|^2|\lambda^w|^2-2|\al|^2(\langle w(\lambda), \check{\al}\rangle \al,\lambda^w)-\langle w(\lambda),\check{\al} \rangle^2|\al|^4+|\al|^2K^2d_{\lambda}^2=\\
&=&(\lambda^w, \al)^2-|\al|^2|\lambda^w|^2+|\al|^2K^2d_{\lambda}^2=D^{w}
\end{array}
\end{equation*}
Now if $w^{-1}(\al)$ is a  finite negative root, then clearly $(s_{\al}w)^{-1}(\al)\in\Delta_{+}$, hence, using the invariance
 property just proved, in what follows we may assume that $w\in \W\setminus\{e, s_{\alpha}\}$ is such that
 $w^{-1}(\al)\in\Delta_{+}$.

Denote now by $\Dw$ the set of positive roots sent to negative roots by $w^{-1}$, let $C^w$ be the (closed convex rational) 
cone $\langle \Dw \rangle_{\mathbb{R}^+}$ generated by the elements of $\Dw$ and notice that $\al$ is not in $\pm C^w$. Indeed,
 $\al$ is not in $C^w$ since all elements of this cone are sent to non-negative linear combination of finite negative roots by $w^{-1}$
 and, on the other hand, $\al$ is a finite positive root while all elements in $-C^w$ are non-negative linear combinations of finite negative 
roots.

Let $L^w$ be the set of weights $\lambda^w$, where  $\lambda$ runs in $\C^+$ and fix a reduced expression
 $s_{i_1}\ldots s_{i_r}$, with $s_{i_j}:=s_{\alpha_{i_j}}$, for $\alpha_j\in\stackrel{\cdot}\Pi$. Then we have $w(\lambda)=\lambda-(a_1\beta_{i_1}\ldots a_r \beta_{i_r})$, where $\beta_{j}=s_{i_1}\ldots s_{i_{j-1}}(\al_{i_j})$ for $j=1,\ldots , r$. Notice that $a_{i_j}\geq 0$ for all $j$ since $\lambda\in \C^+$ and, moreover, $\{\beta_{i_1},\ldots,\beta_{i_r}\}=\Dw$. This shows that $L^w\subseteq C^w$.

Let $\pi:\h^{*}_{\mathbb{R}}\setminus \{0\}\rightarrow \mathbb{P}(\h^{*}_{\mathbb{R}})$ be the quotient map to the projective space of $\h^{*}_{\mathbb{R}}$. Given two non-zero vectors $u,v\in\h^{*}_{\mathbb{R}}$ we denote by $[u,v]$ the angle between them; clearly this symbol depends only on  the lines generated by $u$ and $v$ up to sign to change and up to supplementary angles. In particular, the map $\mathbb{P}(\h^{*}_{\mathbb{R}})^2\rightarrow \mathbb{R}$ defined by $(\pi[u], \pi[v])\mapsto \cos^2[u,v]$ is well--defined.

Since $C^w$ is a closed convex rational cone we have that $\pi(C^w\setminus\{0\})$ is closed in $\mathbb{P}(\h^{*}_{\mathbb{R}})$. Hence  the  map $\pi(C^w\setminus\{0\})\rightarrow \mathbb{R}$ sending $\pi(\mu)\mapsto \cos^2[\mu, \al]$ achieves a maximal value $M^{\al,w}$ and this maximal value is less than 1 since $\pi(\al)\not\in\pi(C^w\setminus \{0\})$. In particular we have $\cos^2[\lambda^w,\al]\leq M^{\al,w}<1$ for all $\lambda\in\C^+\setminus\{0\}$ since $L^w\subseteq C^w$.

Finally, since there are only a finite number of pairs $(\al,w)$, we have $M:=\max M^{\al,w}<1$. Now notice that $w(\lambda)\not\in \C^+$, because $w\neq e$, so $|\lambda^w|\geq d_{\lambda}$, as the segment from $\lambda$ to $w(\lambda)$ must cross a wall of $\C^+$.

 We have to show $D^w\leq 0$.  Since 
\begin{equation*}(\lambda^w,\al)=|\lambda^w||\al| \cos[\lambda^w,\al],
\end{equation*}
 our inequality becomes $\cos^2[\lambda^w,\alpha]\leq 1-K^2d_{\lambda}^2/|\lambda^w|^2$. But we have $\cos^2[\lambda^w,\alpha]\leq M<1$ and $1-K^2d^2_{\lambda}/|\lambda^w|^2>1-K^2$. Hence it is enough to choose $K$ such that $M\leq 1-K^2$. This finishes the proof.

\endproof

Let $\rho$ be half the sum of the finite positive coroots, that is $\rho=\frac{1}{2}\sum_{\al\in\Delta_{+}} \check{\al}$. Moreover,  for any alcove $A\in\A$, let us denote by $c_A$ its centroid.

By using Proposition \ref{prop_strlinked}, together with the identification $\check{\al}=2\al/(\al,\al)$ for all $\al\in\Delta$,
 we get the following characterisation of finite intervals of $\MGp$ which are far enough from the walls of the dominant chamber.

\begin{lem}\label{PMGlimit}Let $A,B\in \A^+$, then there exists an integer $m_0=m_0(A,B)$ such that for any $\lambda\in X\cap  
m\rho+\C^+ $ , with $m\geq m_0$, for any pair $D,E\in[A+\lambda,B+\lambda]$ there is an edge $D-\!\!\!-\!\!\!-E$ in $\MGp$ if and only if 
\item[(i)] either $E=Ds_{\alpha}$ for some $\alpha\in\Dre_{+}$ 
\item[(ii)] or $E=D+a\al$ for some $a\in\mathbb{Z}\setminus \{0\}$ and $\al\in \Delta_{+}$.
\end{lem}
\begin{proof} First of all, let us set $m_1=\max_{E,D\in[A,B]} n_0(E,D)$, where $n_0$ is the one in Proposition \ref{gen_bruhat}.
 It follows that for any $m>m_1$ there is an isomorphism of posets $[A+m\rho,B+m\rho]\cong [A+m_1\rho,B+m_1\rho]$.

By the previous proposition, we know that in $\MGp_{|[A+m\rho,B+m\rho]}$ edges adjacent to a vertex corresponding to a given alcove $D$ are of the desired type if for any $E\in [A+m\rho,B+m\rho]$  it holds
\begin{equation*}
|c_{D}-c_{E}|<K\cdot d_{c_{D}}
\end{equation*} 

Observe that for any $n\geq 0$  and for  any $F\in[A+m_1\rho,B+m_1\rho]$  we have  $c_{F+n\rho}=c_{F}+n\rho$ and hence,
for all  $G\in [A+m_1\rho,B+m_1\rho]$, $|c_{F+n\rho}-c_{G+n\rho}|=|c_F-c_G|$. Moreover, if $n>0$, $d_{c_{F+n\rho}}>d_{c_F}$ and therefore it makes sense to consider

\begin{equation*}n_F:=\min \left\{n\mid  K\cdot d_{c_{F+n\rho}}>|c_{F}-c_{G}| \text{ for all }G\in [A+m_1\rho,B+m_1\rho] \right\}
\end{equation*}
 Finally, we may set 
\begin{equation*}
m_0:=m_1+\max \left\{ n_F \mid  F\in[A+m_1\rho,B+m_1\rho]\right\}
\end{equation*}

\end{proof}

We say that the  edges of type (i), that is given by  reflections, are  \textit{stable}, while the ones of type (ii), that is given by translations, are \textit{non-stable}. We denote the corresponding sets $\E_S$, resp.  $\E_{NS}$. 

\begin{expl}Let $\gaff=\widehat{\mathfrak{sl}_3}$ and $A=A^+$, $B=s_0s_1s_2s_1A^+$. Then in the interval $[A,B]$ of $\MGp$ there are
 edges that are neither stable nor non-stable, as the one between $A=A^+$ and $C=s_0s_1A^+$. 
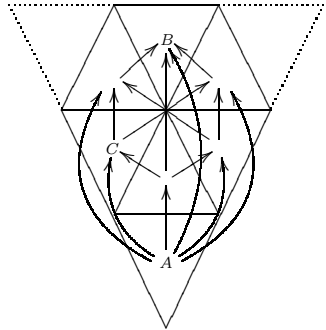
\begin{displaymath}
\begin{xy}0;/r.33pc/:
(-10,20\halfrootthree);(-15,30\halfrootthree)**\dir{.};
(-15,30\halfrootthree);(-5,30\halfrootthree)**\dir{.};
(10,20\halfrootthree);(15,30\halfrootthree)**\dir{.};
(15,30\halfrootthree);(5,30\halfrootthree)**\dir{.};
   (0,0);(-10,20\halfrootthree)**\dir{-};
   (0,0);(10,20\halfrootthree)**\dir{-};
    (-10,20\halfrootthree);(-5,30\halfrootthree)**\dir{-};
  (10,20\halfrootthree);(5,30\halfrootthree)**\dir{-};
   (-10,20\halfrootthree);(10,20\halfrootthree)**\dir{-};
   (-5,10\halfrootthree);(5,10\halfrootthree)**\dir{-};
(-5,30\halfrootthree);(5,30\halfrootthree)**\dir{-};
(-5,30\halfrootthree);(5,10\halfrootthree)**\dir{-};
(5,30\halfrootthree);(-5,10\halfrootthree)**\dir{-};
(0,6\halfrootthree)="A";(0,13\halfrootthree)="A1";
(-5,16\halfrootthree)="A2";(5,16\halfrootthree)="A3";
(-5,22\halfrootthree)="A5";(5,22\halfrootthree)="A4";
(0,26\halfrootthree)="B";
"A"-<0pc,.2pc>*{\text{\tiny{A}}};
"A2"+<-.1pc,.1pc>*{\text{\tiny{C}}};
"B"+<0pc,.25pc>*{\text{\tiny{B}}};
"A"+<0pc,.2pc>,{\ar"A1"+<0pc,-.1pc>};
"A"+<-.5pc,-.15pc>,{\ar@/^17pt/"A5"+<-.4pc,-.1pc>};
"A"+<.5pc,-.15pc>,{\ar@/_17pt/"A4"+<.4pc,-.1pc>};
"A"+<-.4pc,.0pc>,{\ar@/^7pt/"A2"+<-.1pc,-.2pc>};
"A"+<.4pc,.pc>,{\ar@/_7pt/"A3"+<.1pc,-.2pc>};
"A"+<.25pc,.1pc>,{\ar@/_11pt/"B"+<.1pc,-.08pc>};
"A1"+<-.2pc,.2pc>,{\ar"A2"+<.2pc,0pc>};
"A1"+<.2pc,.2pc>,{\ar"A3"+<-.2pc,0pc>};
"A1"+<0pc,.4pc>,{\ar"B"+<0pc,-.2pc>};
"A2"+<0pc,.4pc>,{\ar"A5"+<0pc,.pc>};
"A3"+<0pc,.4pc>,{\ar"A4"+<0pc,.pc>};
"A5"+<.2pc,.35pc>,{\ar"B"+<-.27pc,.1pc>};
"A4"+<-.2pc,.35pc>,{\ar"B"+<.27pc,.1pc>};
"A3"+<-.3pc,.4pc>,{\ar"A5"+<.3pc,.23pc>};
"A2"+<.3pc,.4pc>,{\ar"A4"+<-.3pc,.23pc>};
 \end{xy}
\end{displaymath}
It is enough to translate the interval of $\check{(\al+\beta)}=\al+\beta$ to get the structure described in Lemma \ref{PMGlimit}.
\begin{displaymath}
\begin{xy}0;/r.33pc/:
(-25,50\halfrootthree);(25,50\halfrootthree)**\dir{.};
 (-20,40\halfrootthree);(20,40\halfrootthree)**\dir{.};
 (-15,30\halfrootthree);(15,30\halfrootthree)**\dir{.};
(-10,20\halfrootthree);(10,20\halfrootthree)**\dir{.};
(-5,10\halfrootthree);(5,10\halfrootthree)**\dir{.};
 (0,0);(25,50\halfrootthree)**\dir{.};
  (0,0);(-25,50\halfrootthree)**\dir{.};
    (-5,10\halfrootthree);
(15,50\halfrootthree)**\dir{.};
  (5,10\halfrootthree);
(-15,50\halfrootthree)**\dir{.};
  (-10,20\halfrootthree);
(5,50\halfrootthree)**\dir{.};
  (10,20\halfrootthree);
(-5,50\halfrootthree)**\dir{.};
  (-15,30\halfrootthree);
(-5,50\halfrootthree)**\dir{.};
  (15,30\halfrootthree);
(5,50\halfrootthree)**\dir{.};
  (-20,40\halfrootthree);
(-15,50\halfrootthree)**\dir{.};
  (20,40\halfrootthree);
(15,50\halfrootthree)**\dir{.};
 (-5,50\halfrootthree);(5,50\halfrootthree)**\dir{-};
(-20,40\halfrootthree);(20,40\halfrootthree)**\dir{-};
(-15,30\halfrootthree);(15,30\halfrootthree)**\dir{-};
(-15,30\halfrootthree);(-5,50\halfrootthree)**\dir{-};
(15,30\halfrootthree);(5,50\halfrootthree)**\dir{-};
(-15,30\halfrootthree);(-20,40\halfrootthree)**\dir{-};
(15,30\halfrootthree);(20,40\halfrootthree)**\dir{-};
(0,20\halfrootthree);(-10,40\halfrootthree)**\dir{-};
(0,20\halfrootthree);(10,40\halfrootthree)**\dir{-};
(-5,30\halfrootthree);(5,50\halfrootthree)**\dir{-};
(5,30\halfrootthree);(-5,50\halfrootthree)**\dir{-};
(0,26\halfrootthree)="A";(0,33\halfrootthree)="A1";
(-5,36\halfrootthree)="A2";(5,36\halfrootthree)="A3";
(-5,42\halfrootthree)="A5";(5,42\halfrootthree)="A4";
(-10,33\halfrootthree)="A6";(10,33\halfrootthree)="A7";
(-15,36\halfrootthree)="A8";(15,36\halfrootthree)="A9";
(0,46\halfrootthree)="B";
"A"+<.0pc,-.2pc>*{\text{\tiny{A'}}};
"B"+<.02pc,.35pc>*{\text{\tiny{B'}}};
"A"+<0pc,.2pc>,{\ar"A1"+<0pc,-.1pc>};
"A"+<.15pc,.2pc>,{\ar@/_11pt/"B"+<.1pc,-.08pc>};
"A1"+<-.2pc,.2pc>,{\ar"A2"+<.2pc,0pc>};
"A6"+<-.2pc,.2pc>,{\ar"A8"+<.2pc,0pc>};
"A7"+<-.2pc,.2pc>,{\ar"A3"+<.2pc,0pc>};
"A1"+<.2pc,.2pc>,{\ar"A3"+<-.2pc,0pc>};
"A6"+<.2pc,.2pc>,{\ar"A2"+<-.2pc,0pc>};
"A7"+<.2pc,.2pc>,{\ar"A9"+<-.2pc,0pc>};
"A1"+<0pc,.4pc>,{\ar"B"+<0pc,-.2pc>};
"A2"+<-.1pc,.4pc>,{\ar"A5"+<-.1pc,.pc>};
"A3"+<.1pc,.4pc>,{\ar"A4"+<.1pc,.pc>};
"A5"+<.2pc,.35pc>,{\ar"B"+<-.27pc,.1pc>};
"A4"+<-.2pc,.35pc>,{\ar"B"+<.27pc,.1pc>};
"A3"+<-.3pc,.4pc>,{\ar"A5"+<.3pc,.23pc>};
"A9"+<-.3pc,.4pc>,{\ar"A4"+<.3pc,.23pc>};
"A"+<-.3pc,.22pc>,{\ar"A6"+<.45pc,-.15pc>};
"A2"+<.3pc,.4pc>,{\ar"A4"+<-.3pc,.23pc>};
"A"+<.3pc,.22pc>,{\ar"A7"+<-.5pc,-.05pc>};
"A8"+<.3pc,.4pc>,{\ar"A5"+<-.3pc,.23pc>};
"A"+<-.4pc,-.2pc>,{\ar@/^11pt/"A8"+<-.4pc,-.1pc>};
"A"+<.4pc,-.2pc>,{\ar@/_11pt/"A9"+<.4pc,-.1pc>};
"A8"+<-.15pc,.35pc>,{\ar@/^11pt/"B"+<-.4pc,.4pc>};
"A9"+<.15pc,.35pc>,{\ar@/_11pt/"B"+<.4pc,.4pc>};
"A6"+<.45pc,.1pc>,{\ar@/_11pt/"A4"+<-.18pc,-.18pc>};
"A7"+<-.45pc,.1pc>,{\ar@/^11pt/"A5"+<.18pc,-.18pc>};
\end{xy}
\end{displaymath}

\end{expl}

\begin{lem}\label{iso_Gp} For any pair $A,B\in\A^+$, $B\leq A$ and for any pair $\lambda_1=m_1\rho, \,\lambda_2=m_2\rho\in \check{X}\cap\C^+$ ($m_1,m_2\geq m_0(A,B)$)  then  $\MGp_{|_{[A+\lambda_1,B+\lambda_1]}}$ and $\MGp_{|_{A+\lambda_2,B+\lambda_2]}}$ are isomorphic as oriented graphs.
\end{lem}
\proof Set $\mu:=\lambda_2-\lambda_1$. The isomorphism we are looking for is given by $C\mapsto C+\mu$. Observe that, by Proposition 
\ref{gen_bruhat}, the Bruhat order coincides in the fundamental chamber with the generic one and  so it is invariant by finite coweight translation; then the map we have just defined is an isomorphism of posets. Moreover $C=xA^+$ is connected to $D=yA^{+}$ in $\MGp_{|_{A+\lambda_1,B+\lambda_1]}}$ if and only if $ C+\mu$ is connected to  $D+\mu$ in $\MGp_{|_{A+\lambda_2,B+\lambda_2]}}$, indeed:
\item[(i)] $D=Cs_{\alpha}$ for some $\alpha\Dre_{+}$ if and only if $yx^{-1}=s_{\alpha}$ if and only if $T_{\mu}yx^{-1}T_{-\mu}=T_{\mu}s_{\alpha}T_{-\mu}$, that is $T_{\mu}yx^{-1}T_{-\mu}\in \T$ and this is equivalent to $T_{\mu}yx^{-1}T_{-\mu}=s_{\beta}$ for some $\beta\in \Dre_{+}$, that is $D+\mu=(C+\mu)s_{\beta}$
\item[(ii)] $D=C+a\al$ if and only if $D+\mu=C+a\al+\mu=(C+\mu)+a\al$.
\endproof

\begin{rem} We want to stress the fact  that in Lemma \ref{iso_Gp} we are  \emph{not} proving the existence of an  isomorphism of 
moment graphs, but only between the underlying oriented graphs, that is  we are not considering labels. Our first hope was
 that we could find a collection of $\{f_{l,C}\}_{C\in[A+\lambda_1,B+\lambda_1]}$ satisfying condition \textit{(MORPh2a)} and
 \textit{(MORPH2b)}. In the next two paragraphs, we will see that it is not the case. In particular, it turns out that the labels 
 of stable edges are invariant by finite coroot translation (cf. Lemma \ref{lem_refl_label_root}), while the ones of non-stable edges are
 not (cf. Lemma \ref{lem_transl_label}).
\end{rem}

From now on we will denote by $w\in \WA$ the corresponding alcove $wA^+\in\A$, thanks to the identification (\ref{corr}) of the affine Weyl
 group with its set of alcoves. In particular, if $wA^+$ is contained in the fundamental chamber, we will write $w\in \A^+$.

\subsubsection{Stable edges} Let $|\SRA|=n$ and fix a numbering of the simple reflections. We define the permutation $\sigma_{A,\mu}\in S_{n}$, for $A\in\A$ and $\mu\in \coX$, in the following way: $\sigma_{A, \mu}(i)=j$ if the image under the translation by $\mu$ of the $s_i$--th wall of A is the $s_j$--th wall of $A+\mu$ (cf. \S\ref{ssec_WA_alcoves}). Let $\widetilde{\W}$ the extended affine Weyl group, that is $\widetilde{\W}=\WA\ltimes \Omega$, where $\Omega:=\coX/ \coQ$ (cf. \cite{Lu83}).

\begin{lem}\label{generic_labels}For any $\mu\in \coX$  the permutation defined above is independent on $A\in \A$, i.e. there exists $
\sigma_{\mu}\in S_n$ such that $\sigma_{A, \mu}=\sigma_{\mu}$ for any alcove $A$.
\end{lem}
\proof We know that $T_{k\check{\al}}=s_{-\al+k\delta}s_{\al}$ for $k\in \ZZ_{\geq 0}$ and $\al\in \Delta$. Since we are
 reflecting twice in the same direction (orthogonal to $\al$), the walls of $A+k\check{\al}$ have the same numbering as the ones of $A$.

Thus for any $\mu\in \coX$ there exists an element $\omega\in \Omega$ and roots $\al_1, \ldots \al_r\in \Delta$ such that 
$T_{\mu}=\omega s_{\alpha_1,k_1}s_{\alpha_1}\ldots s_{\alpha_r,k_r}s_{\alpha_r} $ and the numbering of the walls of $A+\mu$ only depends 
on $\omega$.
\endproof

Let us denote (by abuse of notation) also by $\sigma_{\mu}:\widehat{\coQ}\rightarrow\widehat{ \coQ}$ the automorphism of 
the affine  coroot lattice induced by  the map $\al_i\mapsto \al_{\sigma_{\mu}(i)}$ for $\al_i$ corresponding to the simple
 reflection $s_i\in\SRA$. Let us observe that $\sigma_{\mu}$ preserves the positive cone $\langle \widehat{\Pi} \rangle_{\RR_{\geq 0}}$ by 
definition, and the following result is straightforward.

\begin{cor}\label{lem_refl_label_root} Let $x\in \W$, $t\in \T$, $\mu\in\coX$ be such that $x,xt,T_{\mu}x,T_{\mu}xt\in \A^+$. Then, $$l(T_{\mu}x-\!\!\!-\!\!\!-T_{\mu}xt)=\sigma_{\mu}(l(x-\!\!\!-\!\!\!-xt)).$$
\end{cor}

\subsubsection{Non-stable edges} Now we describe how labels of non-stable edges change. In order to do so, we need the following result 

\begin{prop}[\cite{Humph}, Proposition 4.1]\label{conj_reflection_prop} Let $z=T_{z(0)}v$, where $z(0)\!\in\! \coQ$ and $v\!\in\!\W$. Then, for any $\al+n\delta\in \Dre_{+}$ , 
\begin{equation}\label{conj_reflection_eqn}z s_{\al+n\delta}z^{-1}=s_{ v(\al)+ r \delta}\,\,\,\,\,\text{ with }r=n-(v(\al),z(0)).
\end{equation}
\end{prop}

In the proof of  Lemma \ref{lem_transl_label} we will also need  the following equality
\begin{equation}\label{eqn_aff_lables} \check{(\al+n\delta)}=\check{\al}+\frac{2n}{(\al,\al)}\, c.
\end{equation} 

\begin{lem}\label{lem_transl_label} Let $x\in\A^+$ and  $x=T_{x(0)}w$, with $x(0)\in\coQ$ and $w\in \W$.
\begin{itemize}
\item[(i)]If $\alpha\in \Delta$, $n\in\ZZ_{> 0}$ and $T_{n\check{\al}}x\in\A^+ $, then \begin{equation}l(x-\!\!\!-\!\!\!-T_{n\check{\al}}x)=\pm\left ( w^{-1}(\check{\al})+ \frac{2}{(\al,\al)}\big( (\al, x(0))+n\big)\,c\right)\in\Dre_{+}.
\end{equation}
 \item[(ii)] Let $y=T_{a\check{\al}}x$, for some $a\in\ZZ$ and $\al\in\Delta_{+}$. Let moreover $\mu\in\coX$, $\omega\in\Omega$ and $\gamma\in\coQ$ be such that $T_{\mu}=\omega T_{\gamma}$. Then, if $y,T_{\mu}x,T_{\mu}y\in\A^+$, \begin{equation}l(T_{\mu}x-\!\!\!-\!\!\!-T_{\mu}y)=\sigma_{\mu}(l(x-\!\!\!-\!\!\!-y))+ \langle \gamma, \check{\al}\rangle\sigma_{\mu}(c)                                     \end{equation}
\end{itemize}
\end{lem}
\proof
\item[(i)] We have to determine $u\in\W$ and $\gamma\in\Dre_{+}$ such that $x^{-1}u T_{n\check{\al}}x=s_{\gamma}$. Because $T_{n\check{\al}}=s_{-\alpha+n\delta}s_{\alpha}$, we have
\begin{equation*}
x^{-1}uT_{n\check{\al}}x=x^{-1}us_{-\alpha+n\delta}s_{\alpha}x.
\end{equation*}
Clearly, to get a reflection, we have to choose $u=s_{\alpha}$. 

We may now apply twice the proposition above. First consider with $z=s_{\al}$ and $n=0$ and we obtain $s_{\al}s_{-\al+a\delta}s_{\al}=s_{\al+a\delta}$. Next, by applying Proposition \ref{conj_reflection_prop} with $z=x^{-1}$, i.e. $z(0)=-w^{-1}x(0)$ and $v=w^{-1}$, we obtain the following.
\begin{equation*}
x^{-1}uT_{n\check{\al}}x=s_{w^{-1}(\al)+k\delta},
\end{equation*}
with $k=n-(w^{-1}(\alpha),-w^{-1}x(0))=n+(\al,x(0))$, since the bilinear form is $\W$-invariant. The result then follows from (\ref{eqn_aff_lables}) and the fact that $\check{w(\al)}=w(\check{\al})$ for all $\al\in\Delta$ and $w\in\W$.

\item[(ii)]  Let us observe that $T_{\mu}x-\!\!\!-\!\!\!-T_{\mu}y=T_{a\check{\alpha}}(T_{\mu}x)$.
If $x=T_{x(0)}w$, then $T_{\mu}x=T_{\mu+x(0)}w=\omega T_{\gamma+x(0)}w$ (this expression in unique) and we may apply Lemma
 \ref{generic_labels} and point (i) of this lemma with $T_{\gamma}x$ instead of $x$ to get
\begin{align*}
l(T_{\mu}x-\!\!\!-\!\!\!-T_{\mu}y)&=\sigma_{\mu}(l(T_{\gamma}
x-\!\!\!-\!\!\!-T_{\gamma}y)\\
&=\pm \sigma_{\mu}\left(w^{-1}(\check{\al})+\frac{2}{(\al,\al)}\big((\al,\gamma+x(0))+n\big)c\right)\in\Dre_{+}\\
&=\sigma_{\mu}(l(x-\!\!\!-\!\!\!-y))+ \langle \gamma, \check{\al}\rangle\sigma_{\mu}(c).
\end{align*}
\endproof

\subsection{The stable moment graph}\label{ssec_stableMG}
We are now ready to define  the main character of this paper: the \emph{stable moment graph} $\MGst$. This is the moment graph having
 as set of vertices the alcoves in the fundamental chamber (that we identify with the corresponding elements of the Weyl group), equipped
 with the Bruhat order (that here coincides with the generic one);  we connect two vertices if and only if there exists a real
 positive root $\alpha\in \Dre_{+}$ such that $y=xs_{\alpha}$, and in this case we set $l(x-\!\!\!-\!\!\!-xs_{\al}):=\check{\al}$.

Then we have:

\begin{lem}\label{lem_stableMG} For any interval $[y,w]$ and for any $\mu\in \coX$ there  exists an isomorphism of $k$-moment graphs $\MGst_{|_{[y,w]}}\longrightarrow \MGst_{|_{[y+\mu,w+\mu]}}$ for all $k$.
\end{lem}
\proof 
Since the order on the set of vertices of $\MGst$ is invariant by finite coweight translation, we have an isomorphism of posets given 
by the map $f_{\V}:x\mapsto x+\mu$. This  map induces also a bijection between set of edges, as we have already seen in the proof of Lemma \ref{iso_Gp}.

For any $x\in[y,w]$ we set $f_{l,x}=\sigma_{\mu}$ (see Lemma \ref{generic_labels}) and the data $(f_{\V}, \{f_{l,x}\})$ gives us an isomorphism of $k$-moment graphs for any $k$.
\endproof

\section{Sheaves on moment graphs}

The notion of sheaf on a moment graph is due to Braden and MacPherson (cf. \cite{BM01}) and it has been used  by Fiebig in
 several papers (cf.  \cite{Fie08b}, \cite{Fie08a}, \cite{Fie07a}, \cite{Fie07b}, \cite{Fie06}). In the first part of this 
section, we recall the definition of category of sheaves on a $k$-moment graph and we present two important examples, namely, 
the structure sheaf and the canonical --or BMP-- sheaf. In the second part, for any homomorphism of $k$-moment graphs $f$, we 
define the \emph{pullback} functor  $f^*$ and the \emph{push-forward}  functor $f_*$. These two functors turn out to have the
 same adjointness property as in classical sheaf theory (see Proposition \ref{Prop_Adj}).

\subsection{The category of sheaves on a $k$-moment graph}

As in the previous sections, let $Y$ be a lattice of  finite rank and  $k$ a local ring (with $2\in k^*$). 
Let us denote the symmetric algebra of $Y$ by $S=\text{Sym}(Y)$ and  set $S_k:=S\otimes_{\ZZ}k $ its extension. 
As a polynomial ring, $S_k$ has a natural $\ZZ$-grading, but we keep the convention (coming from geometry) of doubling it, 
that is we set  $(S_k)_{\{2\}}=Y_k$. From now on, all the $S_k$-modules will be finitely generated and equipped with 
this $\ZZ$-grading. Moreover, we will consider only degree zero morphisms between them. Finally, for $j\in \mathbb{Z}$ and $M$ a
 graded $S_k$-module we denote by $M\{j\}$ the graded $S_k$-module obtained from $M$ by shifting the grading by $j$, that is $M\{j\}_{\{i\}}=M_{\{j+i\}}$.

\begin{defin}[cf. \cite{BM01}]\label{Def_ShMG} Let $\MG=(\V, \E, \tle, l)\in\MGYk$, then a \emph{sheaf} $\F$ \emph{on }$\MG$ is given by the following data
$(\{\F^x\}, \{\F^E\}, \{\rho_{x,E}\})$ 
\item[(SH1)] for all $x\in \V$, $\F^x$  is an $S_k$-module;
\item[(SH2)] for all $E\in\E$, $\mathcal{F}^{E}$ is an $S_k$-module such that $l(E)\cdot \mathcal{F}^E=\{0\}$;
\item[(SH3)] for $x\in \V$, $E\in \E$, $\rho_{x,E}:\mathcal{F}^x\rightarrow \mathcal{F}^E$ is  a homomorphism of $S_k$-modules defined if $x$ lies on the border of the edge $E$.
\end{defin}

\begin{defin}{\cite{FieNotes}}\label{Def_morphShMG} Let $\MG=(\V, \E,\tle, l)\in \MGYk$ and let
 $\F=(\{\F^x\}, \{\F^E\}, \{\rho_{x,E}\})$, $\F'=(\{\F'^x\}, \{\F'^E\}, \{\rho'_{x,E}\})$ be two sheaves on it.  
A \emph{morphism} $\varphi:\F\longrightarrow \F'$  is given by the following data
\begin{itemize}
\item[(i)] for all $x\in\V$, $\varphi^x:\F^{x}\rightarrow {\F'}^x$ is a homomorphism of $S_k$-modules
\item[(ii)] for all $E\in\E$,  $\varphi^E:\F^{E}\rightarrow {\F'}^E$ is a homomorphism of $S_k$-modules
\end{itemize}
such that,  for any $x\in\V$ on the border of $E\in\E$, the following diagram commutes
\begin{equation*}
\begin{xymatrix}{
\F^x \ar[d]^{\varphi^x}\ar[r]^{\rho_{x,E}}&\F^E \ar[d]^{\varphi^E}\\
{\F'}^x \ar[r]^{\rho'_{x,E}}&{\F'}^E\\
}
\end{xymatrix}
\end{equation*}
\end{defin}

\begin{defin}\label{Def_catShMG} Let $\MG\in\MGYk$. We denote by $\ShMGk$ the category, whose objects are the sheaves on $\MG$ 
and whose morphisms are as in Definition \ref{Def_morphShMG}. 
\end{defin}

\subsection{Combinatorial sheaf theory}
In this paragraph we recall or introduce notions as space of sections, flabbiness, pullback and pushfoward functors,
 which mimic the corresponding ones in classical sheaf theory.

\subsubsection{Sections of sheaves} Let us consider a sheaf $\F\in \ShMGk$  and let $\I$ be a subset of the set of vertices $\V$ of
 $\MG$. The the set of \emph{local sections} of $\F$ over $\I$ is defined as follows.
\begin{equation*}
 \Gamma(\I, \F)=\left\{(m_x)\in \prod_{x\in \mathcal{I}} \mathcal{F}^{x} \mid \begin{array}{c} 
\rho_{x,E}(m_x)=\rho_{y,E}(m_y)\\ \forall E:x-\!\!\!-\!\!\!- y\in\E ,\, x,y\in\mathcal{I}                                                                                                                  
                                                                                                                 \end{array}\right\}.
\end{equation*}

We write $\Gamma(\F)$ for $\Gamma(\V, \F)$ and we call it the set of \emph{global sections} of $\F$.

\subsubsection{The structure sheaf and the structure algebra}\label{sssec_ShZ_Z}. With any $k$-moment graph, it is possible to associate its structure sheaf $\ShZ$, that
 is the sheaf on $\MG$ given by $\ShZ^x=S_k$ for all $x\in \V$, $\ShZ^E=S_k/l(E)\cdot S_k$ for all $E\in\E$ and 
$\rho_{x,E}:\ShZ^x=S_k\rightarrow \ShZ^E=S_k/l(E)\cdot S_k$ the canonical quotient map, for any vertex $x\in \V$ on 
the edge $E\in \E$. Then the \emph{structure algebra} $\Z$ of $\MG$ is the set of global sections of $\ShZ$, namely
\begin{equation*}\Z:=\Gamma(\ShZ)=\left\{(z_x)_{x\in\V}\in\bigoplus_{x\in\V}S_k \,\,\Big\vert\,\,  \forall \, E:x-\!\!\!-\!\!\!-y\in\E  \,\,\, z_x-z_y\in l(E)\cdot S \right\}	
\end{equation*}
The symmetric algebra $S_k$ acts on the structure algebra via diagonal action and it is easy to check that $\Z$ is
 actually an algebra under componentwise addiction and multiplication.

\subsubsection{Flabby sheaves} Once obtained the analogue of the spaces of sections, we would like to define the concept of flabby sheaves. Clearly, in order
 to do so, the notion of open set is needed.   We declare $\I\subseteq\V$ to be \emph{open} 
if  it is upwardly closed,  that is if and only if whenever $x\in \I$ and  $y\geq x$, then also $y\in \I$.
\begin{defin}[cf. \cite{Fie08a}]
A sheaf $\F$ on $\MG$ is \emph{flabby} if the map $\Gamma(\F)\rightarrow \Gamma(\I,\F)$, given by the projection on the 
$\I$-components, is surjective for any $\I\subseteq\V$ open.		
\end{defin}

\subsection{Pullback and pushforward of sheaves}

Let $f=(f_{\V}, \{f_{l,x}\}):\MG=(\V, \E, \tle, l)\rightarrow \MG'=(\V, \E, \tle, l)$  be a morphism of $k$-moment graphs.
 We want to define two functors
\begin{displaymath}\xymatrix{  \ShMGk \ar@/_1.4pc/[rr]_{f_*} && \textbf{Sh}_k(\MG')\ar@/_1.4pc/[ll]_{f^*}
}
 \end{displaymath}

From now on, for any $\varphi\in\text{Aut}_k(Y_k)$, we will denote by $\varphi$ also the automorphism of $S_k$ that it induces. 

We need a lemma, in order to make consistent the definitions we are going to give.

\begin{lem} Let $s\in S_k$,  $f\in\Hom_{\MGYk}(\MG, \MG')$, $\F\in\ShMGk$ and $\H\in\textbf{Sh}_k(\MG')$. Let
 $E:x-\!\!\!-\!\!\!-y\in\E$ and $F:f_{\V}(x)-\!\!\!-\!\!\!-f_{\V}(y)\in\E'$, then
 \item[(i)] the twisted actions of $S_k$ on $\F^E$ defined via $s\centerdot m_E:=f_{l,x}^{-1}(s)\cdot m_E$ and 
$s\centerdot m_E:=f_{l,y}^{-1}(s)\cdot m_E$ coincide on $\F^E/l'(F)\centerdot \F^E$ ($\cdot$ denotes the action of
 $S_k$ on $\F^E$ before the twist). Moreover, $l'(F)\centerdot \F^E=\{0\}$ in both cases.
\item[(ii)] the twisted actions of $S_k$ on $\H^F$ defined via $s\centerdot n_F:=f_{l,x}(s)\cdot n_F$ and
 $s\centerdot n_F:=f_{l,y}(s)\cdot n_F$ coincide on $\H^F/l(E)\centerdot\H^F$ ($\cdot$ denotes the action of $S_k$ on $\F^E$ before 
the twist). Moreover, $l(E)\centerdot \H^F=\{0\}$ in both cases.
\end{lem}
\proof  It is enough to prove the claim for $s\in (S_k)_{\{2\}}=Y_k$, since $S_k$ is a $k$-algebra generated by $Y_k$. 
 \item[(i)]  The statement follows from \textit{(MORPH2a)}, \textit{(MORPH2b)} and the computations we made in the proof
 of Lemma
 \ref{Lem_IsoMG}.
 \item[(ii)] It is an immediate consequence of conditions \textit{(MORPH2a)}, \textit{(MORPH2b)}.
\endproof

If $\varphi$ is an automorphism of $S_k$, for any $S_k$-module $M$, we will denote $\text{Tw}_{\varphi}:M\rightarrow M$ the map
 sending $M$ to $M$ and twisting the action of $S_{k}$ on $M$ by $\varphi$.

\subsubsection{Definitions}

\begin{defin}\label{Def_push} Let $\F\in\ShMGk$, then $f_*\F\in\textbf{Sh}_k(\MG')$ is defined as follows
\item[(PUSH1)] for any $u \in \V'$, $$(f_*\F)^u:=\Gamma(f_{\V}^{-1}(u), \F)$$ and the structure of $S_k$-module is
 given by $s\centerdot (m_x)_{x\in f_{\V}^{-1}(u)}:=(s\cdot m_x)_{x\in f_{\V}^{-1}(u)}$
\item[(PUSH2)] for any $u \in \V'$, $$(f_*\F)^F:=\bigoplus_{E:f_{\E}(E)=F}\F^E$$ and the action of $S_k$ is twisted 
in the following way: $s\centerdot (m_E)_{E: f_{\E}(E)=F}:=(f_{l,x}^{-1}(s)\cdot m_E)_{E: f_{\E}(E)=F}$, where $x$ is 
on the border of $E$
\item[(PUSH3)]  for all $u\in \V'$ and $F\in \E'$, such that $u$ is in the border of the edge $F$,$(f_*\rho)_{u,F}$ 
is defined as the composition of the following maps:
\begin{displaymath}\begin{xymatrix}{
\Gamma(f_{\V}^{-1}(u), \F)  \ar@{^{(}->}[r]& \bigoplus_{x :  f_{\V}(x)=u} \F^x \ar[r]^{\oplus \rho_{x,E}}&\bigoplus_{E :  f_{\V}(E)=F} \F^E\ar[r]^{\text{Tw}}&\bigoplus_{E :  f_{\V}(E)=F} \F^E,
}
\end{xymatrix}
\end{displaymath}
where $\text{Tw}=\oplus \text{Tw}_{f_{l,x}^{-1}}$.
We call $f_*$ \emph{direct image} or \emph{push-forward} functor.
\end{defin}

\begin{defin}\label{Def_pull} Let $\H\in\textbf{Sh}_k(\MG')$, then $f^*\H\in \ShMGk$ is defined as follows
\item[(PULL1)] for all $x\in\V$, $(f^*\H)^x:=\H^{f_{\V}(x)}$ an the action of $S_k$ is twisted by $f_{l,x}$
\item[(PULL2)] for all $E:x-\!\!\!-\!\!\!-y\in\E$\begin{equation*}(f^*\H)^{E}=\left\{ \begin{array}{ll}
\H^{f_{\V}(x)}/l(E) \H^{f_{\V}(x)} & \text{if } f_{\V}(x)=f_{\V}(y)\\
\H^{f_{\E}(E)} &\text{otherwise }
\end{array}
\right.
\end{equation*}
and of $s\in S_k$ acts on $(f^*\H)^{E}$ via $f_{l,x}(s)$.
\item[(PULL3)] for all $x\in \V$ and $E\in \E$, such that $x$ is in the border of the edge $E$,\begin{equation*}(f^*\rho)_{x,E}=\left\{ \begin{array}{ll}
\text{canonical quotient map} & \text{if } f_{\V}(x)=f_{\V}(y)\\
\text{Tw}_{f_{l,x}}\circ \rho_{f_{\V}(x),f_{\E}(E)}\circ \text{Tw}_{f_{l,x}^{-1}} &\text{otherwise }
\end{array}
\right.
\end{equation*}
We call $f^*$ \emph{inverse image} or \emph{pullback} functor.
\end{defin}

\begin{expl}\label{Ex_pullpush_pt} Let $\MG\in\MGYk$ and let $p:\MG\rightarrow \{\text{pt}\}$ be  the homomorphism of $k$-moment
 graphs having $p_{l,x}=\id_{Y_k}$ for all $x$, vertex of $\MG$. Then, for any $\F\in\ShMGk$ $p_*(\F)=\Gamma(\F)$. 
Moreover $p^{*}(S_k)=\mathscr{Z}$, the structure sheaf of $\MG$.
\end{expl}

\subsubsection{Adjuction formula}\label{ssec_adjformula} Although the following result will not be used in the rest of the paper, we include it
for completeness. 

\begin{prop}\label{Prop_Adj}
Let $f\in\Hom_{\MGYk}(\MG, \MG')$, then $f^*$ is left adjoint to $f_*$, that is for all pair of sheaves $\F\in\ShMGk$ and $\mathcal{H}\in\textbf{Sh}_k(\MG')$ the following equality holds 
\begin{equation}\label{Eqn_Adjunction} \Hom_{\ShMGk}(f^*\mathcal{H}, \F)=\Hom_{\textbf{Sh}_k(\MG')}(\mathcal{H}, f_*\F)
\end{equation}
\end{prop}
\proof
Let us take $\varphi\in\Hom_{\ShMGk}(f^*\mathcal{H}, \F)$, that is $\varphi=(\{\varphi^x\}_{x\in\V}, \{\varphi^E\}_{E\in\E})$ 
such that for all $x\in\V$ and $E\in\E$, with $x$ is on the border of $E$, the following diagram commutes

\begin{equation}\label{Diagr_Adj1}\begin{xymatrix}{
(f^*\H)^x\ar[d]|{\,\,\,\,\,(\!f^*\rho'\!)_{x,E}}\ar[r]^{\,\,\,\,\varphi^x}&\F^x\ar[d]|{\rho_{x,E}}\\
(f^*\H)^E\ar[r]_{\,\,\,\,\varphi^E}&\F^E
}
\end{xymatrix}
\end{equation}

We want to show that there is a bijective map 
$ \gamma:\Hom_{\ShMGk}(f^*\mathcal{H}, \F)\rightarrow\Hom_{\textbf{Sh}_k(\MG')}(\mathcal{H}, f_*\F)$ and it is  given by 
$\varphi=(\{\varphi^x\}_{x\in\V}, \{\varphi^E\}_{E\in\E})\mapsto \psi=(\{\psi^u\}_{u\in\V'}, \{\psi^F\}_{F\in\E'})$, where
\begin{equation*}\psi^u:=\oplus_{x\in f_{\V}^{-1}(u)}\, \varphi^x\, ,\hspace{10mm}\psi^F:=\oplus_{E\in f_{\E}^{-1}(F)}\, \varphi^E
\end{equation*}

We start by verifying that this map is well-defined. We have to show that for any 
$h\in \H^u$, $\psi^u(h)\in(f_*\F)^u=\Gamma(f_{\V}^{-1}(u),\F)$, that is, for any $x, y\in f_{\V}^{-1}(u)$ such 
that $E:x-\!\!\!-\!\!\!-y\in\E$, $\rho_{x,E}(\varphi^{x}(h))=\rho_{y,E}(\varphi^{y}(h))$. 

From Diagram (\ref{Diagr_Adj1}), we get the following commutative diagram
\begin{equation}\begin{xymatrix}{
(f^*\H)^x=\H^{f_{\V}(x)}=\H^u\ar[d]|{\,\,\,\,\,(\!f^*\rho'\!)_{x,E}}\ar[rr]^{\hspace{19mm}\varphi^x}&&\F^x\ar[d]|{\rho_{x,E}}\\
(f^*\H)^E=\H^{u}/l(E) \H^{u}\ar[rr]^{\hspace{19mm}\varphi^E}&&\F^E\\
(f^*\H)^y=\H^{f_{\V}(y)}=\H^u\ar[u]|{\,\,\,\,\,(\!f^*\rho'\!)_{y,E}}\ar[rr]^{\hspace{19mm}\varphi^y}&&\F^y\ar[u]|{\rho_{x,E}}\\
}
\end{xymatrix}
\end{equation}

But $(\!f^*\rho'\!)_{y,E}=(\!f^*\rho'\!)_{x,E}$ by definition (they are both the canonical projection) and we obtain 
\begin{equation*}\rho_{x,E}\circ \varphi^x=\varphi^E\circ (\!f^*\rho'\!)_{x,E}=\varphi^E\circ (\!f^*\rho'\!)_{y,E}=\rho_{y,E}\circ \varphi^y
\end{equation*}

It is clear that the map $\gamma: \Hom_{\ShMGk}(f^*\mathcal{H}, \F)\rightarrow\Hom_{\textbf{Sh}_k(\MG')}(\mathcal{H}, f_*\F)$ we defined is injective. To conclude our proof, we have to show the  surjectivity of $\gamma$.

Suppose $\psi=(\{\psi^u\}_{u\in\V'}, \{\psi^F\}_{F\in\E'})\in \Hom_{\textbf{Sh}_k(\MG')}(\mathcal{H}, f_*\F)$, where, for all $u\in\V'$ and $F\in\E'$ such that $u$ is on the border of $F$, the following diagram commutes

\begin{equation}\label{Diagr_Adj2}\begin{xymatrix}{
\H^u\ar[d]|{\,\,\,\,\,\rho'_{u,F}}\ar[r]^{\psi^x}&\Gamma(f_{\V}^{-1}(u),\F)\ar[d]|{\oplus (\text{Tw}_{f_{l,x}}\circ\rho_{x,E})}\\
(f^*\H)^F\ar[r]_{\psi^E}&\bigoplus_{E\in f_{\E}^{-1}(F)}\F^E
}
\end{xymatrix}
\end{equation}

We claim  that there exist $\varphi=(\{\varphi^{x}\})_{}\in \Hom_{\ShMGk}(f^*\mathcal{H}, \F)$ such that $\gamma(\varphi)=\psi$.

For any $x\in\V$, let us consider $u:=f_{\V}(x)$ and define $\varphi^x$ as the composition of the following maps

\begin{displaymath}\begin{xymatrix}{
\H^u\ar[r]^(0.33){\psi^y}\ar@/_1.6pc/[rrr]_{{\varphi^x}}&\Gamma(f_{\V}^{-1}(u), \F)\ar@{^{(}->}[r]&\bigoplus_{y\in f_{\V}^{-1}(u)}\F^{y}\ar@{->>}[r]&\F^x
}
\end{xymatrix}
\end{displaymath}

For any $E:x-\!\!\!-\!\!\!-y\in\E$ such that $f_{\V}(x)\neq f_{\V}(y)$, that is there exists an edge $F\in\E'$ such that $f_{\E}(E)=F$, we define $\varphi^E$ as the composition of the following  maps

\begin{displaymath}\begin{xymatrix}{
\H^F\ar[r]^(0.33){\psi^F}\ar@/_1.6pc/[rrr]_{{\varphi^E}}&\bigoplus_{L\in f_{\E}^{-1}(F)} \F^L\ar[r]^{\text{Tw}_{f_{l,y}}}&\bigoplus_{L\in f_{\E}^{-1}(F)} \F^L\ar@{->>}[r]&\F^E
}
\end{xymatrix}
\end{displaymath}

Now, it is clear that $\gamma(\varphi)=\psi$. Indeed, if $u\not\in f_{\V}(\V)$, then $\psi^u=0$ and the claim is trivial. Otherwise, $u\in f_{\V}(\V)$ and we get the following diagram, with Cartesian squares

\begin{displaymath}\begin{xymatrix}{
\H^u\ar[d]|{\rho'_{u,F}}\ar[r]^(0.33){\psi^y}\ar@/^1.6pc/[rrr]<1ex>^{{\varphi^x}}&\Gamma(f_{\V}^{-1}(u), \F)\ar[d]|{\,\,\,\,\,\,\,(f_*\rho)_{y,F}}\ar@{^{(}->}[r]&\bigoplus_{y\in f_{\V}^{-1}(u)}\F^{y}\ar[d]|{\oplus \rho_{z,L}}\ar@{->>}[r]&\F^x\ar[d]|{\rho_{x,E}}\\
\H^F\ar[r]^(0.33){\psi^F}\ar@/_1.6pc/[rrr]<-1ex>_{{\varphi^E}}&\bigoplus_{L\in f_{\E}^{-1}(F)} \F^L\ar[r]^{\text{Tw}_{f_{l,y}}}&\bigoplus_{L\in f_{\E}^{-1}(F)} \F^L\ar@{->>}[r]&\F^E
}
\end{xymatrix}
\end{displaymath}

\endproof

As application of the previous proposition, we get the following corollary.

\begin{cor} Let $\MG\in\MGYk$ and let $\ShZ$, resp. $\Z$, be its structure sheaf, resp. its structure algebra. Then the functors $\Gamma(-),\Hom_{\ShMGk}( \ShZ, -): \ShMGk\rightarrow \Z-\text{modules}$ are naturally equivalent.  In particular, we get the following isomorphism of $S_k$-modules
\begin{equation*}
\Z\cong \text{End}_{\ShMGk}(\ShZ).
\end{equation*}
\end{cor}
\proof Consider the homomorphism $p:\MG\rightarrow \{\text{pt}\}$, where we set $p_{l,x}=\id_{Y_k}$ for all $x$, vertex of $\MG$. The structure sheaf of $\{\text{pt}\}$ is just a copy of $S_k$ and, for all $\F\in\ShMGk$, by Prop. \ref{Prop_Adj}, we get
\begin{equation*}\Hom_{\ShMGk}(p^*S_k, \F)=\Hom_{\textbf{Sh}_k(\{\text{pt}\})}(S_k, p_*\F)
\end{equation*}

Bu we have already noticed in  Example  \ref{Ex_pullpush_pt} that  $p^*S_k\cong \ShZ$ and $p_*\F=\Gamma(\F)$. Moreover,  that $ \Hom_{S_k}(S_k, \Z)\cong \Z$ and we get the claim.

\endproof

\subsection{BMP-sheaves}
The following definition, due to Fiebig and Williamson, generalises the one of canonical sheaves given by Braden and MacPherson 
in \cite{BM01}. These sheaves will play a fundamental role in the rest of this paper.

\begin{defin}[cf. \cite{FW}]\label{defin_BMP} Let $\MG\in\MGYk$ and let $\BMP$ be a sheaf on it.  We say that $\BMP$ is a
\emph{Braden-MacPherson sheaf} if it satisfies the following
properties:
\begin{itemize}
 \item[(BMP1)]  for any $x\in\V$, $\BMP^x\in S\text{-mod}$ is free
\item[(BMP2)] for any $E:x\rightarrow y\in\E$, $\rho_{y,E}:\BMP^y\rightarrow \BMP^E$ is surjective with kernel
$l(E)\cdot\BMP^y$  
\item[(BMP3)] $\BMP$ is flabby 
\item[(BMP4)]  for any $x\in\V$, the map
$\Gamma(\BMP)\rightarrow \BMP^x$ is surjective.
\end{itemize}

\end{defin}

 Hereafter, Braden-MacPherson sheaves will be referred to also as  \emph{BMP}-sheaves
 or \emph{canonical} sheaves. 

An important theorem, characterizing Braden-MacPherson sheaves, is the following one.
\begin{theor}[cf. \cite{FW}, Theorem 6.3] \label{theor_FW_BMPindc}Let $\MG\in\MGYk$
\item[(i)] For any $w\in\V$, there is up to isomorphism unique Braden-MacPherson sheaf $\BMP(w)\in\ShMGk$ with the following properties:

\vspace{2.3mm}\hspace{3mm}\textit{(BMP0)} $\BMP(w)$ is indecomposable in $\ShMGk$

\hspace{3mm}\textit{(BMP1a)} $\BMP(w)^w\cong S_k$ and $\BMP(w)^x=0$, unless $x\leq w$
\vspace{2.3mm}

\item[(ii)]Let $\BMP$ be a Braden-MacPherson sheaf. Then, there are $w_1, \ldots, w_r\in\V$ and $l_1\ldots l_r\in\mathbb{Z}$ such that 
\begin{equation*} \BMP\cong \BMP(w_1)[l_r]\oplus \ldots \oplus \BMP(w_r)[l_r]
\end{equation*}h
\end{theor}

We want to quote also a result by Fiebig, that will be used later, which tells us that the structure sheaf is not in general 
flabby, but that if it is the case, then it is isomorphic to an indecomposable BMP-sheaf.

\begin{prop}[\cite{Fie06}, Proposition 4.2]\label{prop_StrucShflabby} Let $\MG\in\MGYk$ be such that it has a highest  vertex $w$.
 Then $\BMP(w)\cong \ShZ$ if and only if $\ShZ$ is flabby.
\end{prop}

\subsubsection{Pullback of BMP-sheaves} We conclude this section by recalling a result from our paper \cite{L11}.

 \begin{lem}[cf. \cite{L11}] \label{pullbackBMP}Let $\MG$ and $\MG'$ be two $k$-moment graphs on $Y$, both with a unique
 maximal vertex, w resp. w', and let  $f:\MG\longrightarrow \MG'$ be an isomorphism. If  $\BMP_w$ and $\BMP'_{w'}$ are the
 corresponding canonical sheaves, then $\BMP_w\cong f^*\BMP'_{w'}$ as $k$-sheaves on $\MG$.
\end{lem}

\section{Statement of the main result and proof of the subgeneric case}

 In order to provide a new approach to Kazhdan-Lusztig amd Lusztig conjectures, Fiebig applied the theory of sheaves on moment 
graphs to the representation theory of complex symmetrisable Kac-Moody algebras and of semisimple reductive algebraic groups over
 a field of positive characteristic and formulated a conjectural formula relating Kazhdan-Lusztig polynomials to the stalks of indecomposable 
Braden-MacPherson sheaves (cf. \cite{Fie08a}, \cite{Fie07b}, \cite{Fie07a}). More precisely, if $\MGJ$ is the Bruhat graph we
 defined  in \S\ref{bruhatMG}, thus for any $w\in \W^J$ we can consider the subgraph $\MG^J_{w}:=\MGJ_{|\{\leq w\}}$. It is a 
finite $k$-moment graph (for any $k$) with highest vertex $w$, hence we may build the corresponding indecomposable
 Braden-MacPherson sheaf  $\BMP^J(w)\in \ShMGJk$. For any free and finitely generated $S_k$-module $M=\bigoplus_{i=1}^n S\{l_i\}$, its graded 
rank is
\begin{equation*}
 \rk M=\sum_{i=1}^n  q^{-l_i}\in \ZZ[q^{\pm 1}]
\end{equation*}

Finally, let us denote by  $P^{J,-1}_{y,w}$  Deodhar's analogue of Kazhdan-Lusztig polynomials at the parameter $u=-1$
 (cf. \cite{Deo87}).

\begin{quest}[Cf. \cite{Fie07b}, Conjecture 4.4.]\label{mult_conj} Under which assumptions on the characteristic of the base field, if $y\leq w$ 
then $\rk \,(\BMP^J(w))^y= P^{J,-1}_{y,w}(q)$?
\end{quest}

This question motivated our paper \cite{L11}, where we were able to give the moment graph analogue of several properties of
 Kazhdan-Lusztig polynomials. The interpretation of certain equalities of these polynomials at a categorical level has the 
advantage of furnishing a deeper understanding of several phenomena, since they are lifted to a setting where there is extra 
structure.
The problem we address in this paper concerns a stabilisation property of the affine Kazhdan-Lusztig polynomials, proved by Lusztig in \cite{Lu80},
 while he was trying to find a support for his conjecture on modular representations.

Let $\W$ denote a finite Weyl group, $\WA$ its affinisation and $\SR$ the set of \emph{finite} simple reflections.
 Then, as we  have already noticed, the set of minimal representatives of $\W\setminus \WA$ is in bijection with the set of 
alcoves in the fundamental chamber and the corresponding  parabolic Kazhdan-Lusztig polynomials may be indexed by pairs of
 alcoves $A, B\in \A^+$. The result that Lusztig proved can be reformulated in terms of these polynomials. In particular,  
 quoting Soergel's reformulation (cf. \cite[Theorem 6.1]{Soe97}), the parabolic Kazhdan-Lusztig polynomials 
$P_{A,B}^{J, -1}$ indexed by pairs of alcoves far enough in the fundamental chamber stabilize, in the sense that, for 
any pair of alcoves $A,B$, there exists a polynomial $Q_{A,B}$ with integer coefficients such that
 \begin{equation*}\lim_{\overrightarrow{\mu\in \C^+}}P_{A+\mu,B+\mu}^{J, -1}=Q_{A,B}
\end{equation*}
The $Q_{A,B}$'s are called \emph{generic polynomials} and turn out to have a realization very similar to the one of the regular 
Kazhdan-Lusztig polynomials. Indeed, Lusztig in \cite{Lu80} associated to every affine Weyl group $\WA$ its periodic module 
\textbf{M}, that is the free $\mathbb{Z}[q^{\pm\frac{1}{2}}]$-module with set of generators -or standard basis- indexed by the set
 of all alcoves $\A$. It is possible to define an involution and to prove that there exists  a self-dual basis 
of $\textbf{M}$: the \emph{canonical basis}. In this setting, the generic polynomials are the coefficients of the change basis
 matrix. Our interest in the periodic module is motivated by the fact  that $\textbf{M}$ governs the representation theory of 
the affine Kac-Moody algebra, whose Weyl group is $\WA$, at the critical level.

 The aim of this section is to study the behaviour of indecomposable Braden-MacPherson sheaves on finite intervals of the
 parabolic Bruhat graph $\MGp$ far enough in $\C^+$ (see \S\ref{ssec_stabMG}). More precisely, let $\I=[A,B]$ be an interval 
far enough in the fundamental chamber. Inspired by \cite[ Proposition 11.15]{Lu80}, we claim that  for any $C\in[A,B]$
 and for all $\mu\in\coX\cap \C^+$,
\begin{equation}\label{eqn_stab}\BMP(B)^{C}\cong \BMP(B+\mu)^{C+\mu}.
\end{equation}

We showed in \S\ref{ssec_stabMG} that $\MGp_{|_{[A,B]}}$ is in general not isomorphic 
to $\MGp_{|_{[A+\mu,B+\mu]}}$ as a $k$-moment graph, so unluckily  we cannot use Lemma \ref{pullbackBMP} to get
 the isomorphism of $S_k$-modules above.  On the other hand, we proved in Lemma \ref{lem_stableMG}  that, 
for all $\mu\in\check{X}$,  there is an isomorphism of $k$-moment graphs
\begin{equation*} \tau_{\mu}:\MGst_{|_{[A,B]}}\rightarrow \MGst_{|_{[A+\mu,B+\mu]}}
\end{equation*}
Thus, by Lemma \ref{pullbackBMP}, we get an isomorphism between the indecomposable canonical sheaf $\BMP(B)$ on 
$\MGst_{|_{[A,B]}}$ and $ \tau_{\mu}^*\BMP(B+\mu)$, the pullback of the indecomposable Braden-MacPherson sheaf
 $\BMP(B+\mu)$ on  $\MGst_{|_{[A+\mu,B+\mu]}}$.

For any finite interval $\I$ far enough in the fundamental chamber, consider the monomorphism
 $g_{\I}:\MGst_{|_{\I}}\hookrightarrow\MGp_{|_{\I}}$, given by $g_{\I,\V}=\id_{\V}$ and $g_{\I,l,x}=\id$ for all $x\in\I$. We 
obtain the functor 
\begin{equation*}{\cdot}^{\text{stab}}:\textbf{Sh}_{\MGp_{|_{\I}}}\rightarrow \textbf{Sh}_{\MGst_{|_{\I}}},
 \end{equation*}
 defined by the setting $\F\mapsto {\F}^{\text{stab}}:=g_{\I}^*(\F)$. The goal of the rest of this paper is to prove the following result.

\begin{theor}\label{theor_stabfunctor} Let $\I$ be a finite interval far enough in the fundamental chamber and let $k$ be such that
$(\MGp_{|\I}, k)$ is a GKM-pair. Then the functor ${\cdot}^{\text{stab}}:\ShMGI{par}\rightarrow \ShMGI{stab}$ preserves indecomposable Braden-MacPherson sheaves.
\end{theor}

We will prove this theorem via explicit calculations in the $\widehat{\frak{sl}_2}$-case, while for the general case we will
 need  results and methods developed by Fiebig in \cite{Fie07a}.

Once Theorem \ref{theor_stabfunctor} is proven, we will obtain Equality (\ref{eqn_stab}) by applying Lemma \ref{lem_stableMG}.

\subsection{The subgeneric case} Using the same terminology as Fiebig in \cite{Fie12}, we denominate  \emph{subgeneric} the case in which  $\gaff=\mathfrak{sl}_2$.

In the subgeneric case, for any finite interval $\I$ and any $k$ such that $(\MGp_{|\I}, k)$ is a GKM-pair, the corresponding indecomposable BMP-sheaf is isomorphic to the structure
 sheaf $\A$ of $\MGp_{|\I}$. There are several ways to prove this fact. 

The first one comes from the topology of the underlying variety. We try to give a roughly idea of this argument. To start with, let us suppose $k=\mathbb{Q}$. In the case of  a finite
 interval of a Bruhat moment graph, its structure algebra and the space of global 
sections of a BMP-sheaf describe the $T$-equivariant cohomology and intersection cohomology, respectively, of the corresponding
 Richardson variety. Since all the Richardson varieties are in the subgeneric case rationally smooth, the intersection cohomology
  coincides with the usual cohomology and then the structure sheaf with the canonical sheaf. It 
is possible to perform the same argument, using parity sheaves and results from \cite{FW} in the positive characteristic setting.

An alternative proof comes from Fiebig's multiplicity one result. Indeed, using the inductive formula (2.2.c) of \cite{KL}, it is
 easy to show that, $P_{A,B}^{J-,1}=1$ for all $A,B\in \A^+$ with $A\leq B$. We then have to apply the following theorem
\begin{theor}[cf. \cite{Fie06}, Theorem 6.4.] Suppose that $k$ and $w$ are such that $(\MG_w,k)$ is a GKM-pair. For
all $x\leq w$ we have
\begin{equation*}
\BMP(w)^x \cong S_k \qquad \text{if and only if}\qquad P_{x,w} = 1
\end{equation*}
\end{theor}

But $\BMP(w)^x\cong S_k$ if and only if $\BMP(w)^y\cong S_k$ for all $y\in[x,w]$ (see, e.g. \cite[Theorem 5.1]{Fie06}), which, together with (BMP2), 
 implies that
 the indecomposable BMP-sheaf $\BMP(w)$  on $\MG_{[x,w]}$ is isomorphic to the structure sheaf of such a graph.

A third proof is presented in the appendix of this paper and it makes use only of some combinatorial arguments together with 
intrinsic properties of indecomposable BMP-sheaves.

In this paragraph, $\MGst$, resp. $\MGp$,  denotes the parabolic moment graph, resp. the stable moment graph, for the $\widetilde{A_1}$
 root system. 

We have already seen that for any two vertices $x,w$ with $x\leq w$ the stalk of the Braden-MacPherson sheaf on 
 $\MGp_{\leq w}$ is $\BMP(w)^x\cong S_k$. This fact, by  Proposition \ref{prop_StrucShflabby},  is equivalent to the flabbiness of the
 structure sheaf on $\MGp_{\leq w}$. Therefore in order to show that the functor ${}^{\text{stab}}$ preserves indecomposable canonical 
sheaves, again by Proposition \ref{prop_StrucShflabby}, it is enough to verify that, for any vertex $w$, the structure sheaf 
 $\A$ on $\MGst_{\leq w}$ is still flabby.

Recall that the set of vertices of  $\MGp$ (and so of  $\MGst$) can be identified with the finite (co)root lattice, that is $\mathbb{Z}\al$, where $\al=\check{\al}$ is the positive (co)root  of $A_1$. Moreover,  $\MGp$ is a complete graph and the label function is given, up to a sign, by $l(h\al-\!\!\!-\!\!\!-k\al)=-\al+(h+k)c$. By definition, we get $\MGst$ from $\MGp$ by deleting the non-stable edges, then $h\al-\!\!\!-\!\!\!-k\al\in \E^{\text{stab}}$ if and only if $\text{sgn}(h)=-\text{sgn}(k)$ (where, by convention, we set $\text{sgn}(0)=-$).

\begin{lem}\label{even_section} Let $r\in \mathbb{Z}_{> 0}$. If $n\in\mathbb{Z}$,  set, for any $h\in \mathbb{Z}$, with $h\al\leq n\al$, 
\begin{equation*} ^{\text{e}}z^r_{n\al,h\al}:=\left \{ \begin{array}{lcr}
0&\text{if}& |h|\in [|n|-r+1, |n|]\\
\prod_{i=0}^{r-1}\big[\big(\al+(|n|-h-i)c\big)(|n|-h-i)\big] &\text{if}& h\in(0,|n|-r]\\
\prod_{i=0}^{r-1}\big[\big(-\al+(|n|+h-i)c\big)(|n|+h-i)\big]  &\text{if}& h\in[r-|n| ,0]\\
\end{array}\right.
\end{equation*}
If $(\MGst_{\leq n\al},k)$ is a GKM-pair, then $^{\text{e}}z^r_{n\al}=(^{\text{e}}z^r_{n\al,h\al})\in\Gamma(\{\leq n\al\}, \A)_{\{r\}}$.
\end{lem} 
\proof  We verify that, for any $h,k\in \mathbb{Z}$ such that $h\al, k\al \leq n\al$, if $h\al-\!\!\!-\!\!\!-k\al$ is an edge, 
 then 
\begin{equation}  ^{\text{e}}z^r_{n\al,h\al}- ^{\text{e}}z^r_{n\al,k\al}\equiv 0   \qquad (\text{mod } -\al+(h+k)c)
\end{equation}

We may clearly suppose $h>0$ and $k\leq 0$.

Let at first consider $h\in [|n|-r+1, n]$. If  $-k\in [|n|-r+1, n]$, then $^{\text{e}}z^r_{n\al,h\al}= {}^{\text{e}}z^r_{n\al,k\al}=0$ and there is nothing to prove. Otherwise, $k\in [r-|n|,0]$ and
\begin{equation}^{\text{e}}z^r_{n\al,h\al}- ^{\text{e}}z^r_{n\al,k\al}=0-\prod_{i=0}^{r-1}\big[\big(-\al+(|n|+k-i)c\big)(|n|+k-i)\big].
\end{equation}
Now, $-\al+(h+k)c$ divides $\prod_{i=0}^{r-1}\big[\big(-\al+(|n|+k-i)c\big)(|n|+k-i)\big]$ if and only if there exists an $i\in [0,r-1]$ such that $|n|-i=h$, i.e. $h-|n|=-i$. But we supposed $h\in[|n|-r+1, n]$ that is, precisely, $h-|n|\in[-r+1, 0]$.

Let consider the case $h\in (0, |n|-r]$. If $-k\in [|n|-r+1, n]$, then 
\begin{equation}^{\text{e}}z^r_{n\al,h\al}- ^{\text{e}}z^r_{n\al,k\al}=\prod_{i=0}^{r-1}\big[\big(\al+(|n|-h-i)c\big)(|n|-h-i)\big]-0.
\end{equation}
Now, $-\al+(h+k)c$ divides $\prod_{i=0}^{r-1}\big[\big(\al+(|n|-h-i)c\big)(|n|-h-i)\big]$ if and only if there exists an $i\in [0,r-1]$ such that $|n|-i=-k$, i.e. $-k-|n|=-i$. But we supposed $-k\in[|n|-r+1, |n|]$ that is, precisely, $-k-|n|\in[-r+1, 0]$.

Otherwise,   $k\in [r-|n|,0]$ and
\begin{align*} {}^{\text{e}}z^r_{n\al,h\al}- ^{\text{e}}z^r_{n\al,k\al}&=\prod_{i=0}^{r-1}\!\!\big[\!\big(\al\!+\!(|n|-h-i)c\big)(|n|-h-i)\!\big]\!-\!\prod_{i=0}^{r-1}\!\!\big[\!\big(\!-\!\al\!+\!(k+|n|-i)c\big)(k+|n|-i)\!\big]\\
&\equiv \prod_{i=0}^{r-1}[(k+h+|n|-h-i)(|n|-h-i)c]\\
&-\prod_{i=0}^{r-1}[(-k-h+k+|n|-i)(|n|+k-i)c] \qquad (\text{mod }-\al+(h+k)c)\\
&=c^r\prod_{i=0}^{r-1}[(k+|n|-i)(|n|-h-i)-(-h+|n|-i)(|n|+k-i)]\\
&=0
\end{align*}
\endproof

\begin{lem}\label{odd_section} Let $r\in \mathbb{Z}_{> 0}$. If $n\in\mathbb{Z}$, for any $h\in \mathbb{Z}$, such that $h\al\leq  n\al$, we set
\begin{equation*} ^{\text{o}}z^r_{n\al,h\al}:=\left \{ \begin{array}{lcr}
0&\text{if}& |h|\in [|n|-r+2, |n|]\\
\prod_{i=0}^{r-1}\big[\big(\al+(|n|-h-i)c\big)(|n|-h-i+1)\big] &\text{if}& h\in(0, |n|-r+1]\\
\prod_{i=0}^{r-1}\big[\big(-\al+(|n|+h-i+1)c\big)(|n|+h-i)\big]  &\text{if}&  h\in [r+n-1, 0]\\
\end{array}\right.
\end{equation*}
If $(\MGst_{\leq n\al},k)$ is a GKM-pair, then  $^{\text{o}}z^r_{n\al}=(^{\text{o}}z^r_{n\al,h\al})\in\Gamma(\{\leq n\al\}, \A)_{\{r\}}$.
\end{lem}
\proof The proof is very similar to the one of the previous lemma and therefore we omit it.
\endproof

Define $^{\text{e}}z^0_{n\al}:=(1)_{h\al\leq n\al}$.

\begin{lem}\label{lemma_sections1} Let $r\in\mathbb{Z}_{\geq 0}$, $n\in\mathbb{Z}$ and 
$m\in \mathbb{Z}$ be such that $m\al\leq n\al$. If $(\MGst_{\leq n\al},k)$ is a GKM-pair, for all $z\in \Gamma([ m\al, n\al], \A)_{\{r\}}$,
 there exist ${}^{o}s^i_{k}, {}^{e}s^j_{k}\in  (S_k)_{\{i\}}$, with $i\in[0, r]$, $j\in(0,r]$ and $p$ such that
 $p\al\in [ m\al, n\al]$,  such that 
\begin{equation}
z=\sum_{j=1}^{r}{}^{\text{e}}s^j_{p} (^{\text{e}}z^{r-j}_{p\al})_{p\al\in [ m\al, n\al]} +\sum_{i=0}^{r}{}^{\text{o}}s^i_{p} 
(^{\text{o}}z^r_{p\al})_{p\al\in [ m\al, n\al]}.
\end{equation}
\end{lem}
\proof Let $h\al$ be the maximal vertex in $[m\al,n\al]$ such that $z_{h\al}\neq 0$. We prove the statement by induction on
 $l= \sharp [m\al ,h\al]$.

 If such a vertex does not exists, that is $l=0$, then $z=(0)$ and there is nothing to prove.

We should consider four cases: $n>0$ and $l>0$; $n>0$ and $l\leq 0$; $n\leq 0$ and $l>0$; $n\leq 0$ and  $l\leq 0$. 
Actually, we will verify only the first case, since the others can be proven in a very  similar way.

Let $n>0$ and $h>0$. If $h=n$, then we set $z'=z-z_{n\alpha}^{\text{e}}z^0_{n\al}$ and the result follows from the inductive hypothesis.
 Otherwise, $h<n$ and then  $\prod_{i=0}^{n-h-1}(\al+(n-h-i)c)$ divides $z_{h\al}$ in $S_k$ and we may set 
 \begin{equation}{}^{\text{e}}s^{r-n+h}_h:= \prod_{i=0}^{n-h-1}[(\al+(n-h+i)c)(n-h-i)]^{-1}\cdot z_{h\al}\in S_{\{r-n+h\}}.
\end{equation}
Now $z':=z-{}^{\text{e}}s^{r-n+h+1}_h \cdot (^{\text{e}}z^{n-h}_{p\al})_{p\al\in [ m\al, n\al]}\in  \Gamma([ m\al, n\al], \A)_{\{r\}} $ 
has the property that $z'_{p\al}=0$ for all $p\in[h\al,n\al]$ and we obtain the  statement from the inductive hypothesis.

\endproof

\begin{cor} For any $n\in \mathbb{Z}$ and any $k$ such that $(\MGst_{\leq n\al},k)$ is a GKM-pair, the structure sheaf $\A$ on $\MGst_{\leq n\alpha}$ is flabby. 
\end{cor}
\proof We have to show that every local section $z\in\Gamma( \I, \A)$, with $\I$ open can  be extended to a global section $\widetilde{z}\in\Gamma(\MGst_{\leq n\alpha}, \A)$. Since the set of vertices of $\MGst$ is totally ordered, then any open set of $\MGst_{\leq n\alpha}$ is actually an interval, that is there exists an $m\in \mathbb{Z}$ such that $\I=[m\al, n\al]$. 

Suppose $z\in\Gamma(\I, \A)_{\{r\}}$, then by Lemma \ref{lemma_sections1}, we can write
\begin{equation}z=\sum_{j=1}^{r}{}^{\text{e}}s^j_{k} (^{\text{e}}z^{r-j}_{k\al})_{k\al\in [ m\al, n\al]} +\sum_{i=0}^{r}{}^{\text{o}}s^i_{k} (^{\text{o}}z^r_{k\al})_{k\al\in [ m\al, n\al]}.
\end{equation}

By Lemma \ref{even_section} and Lemma \ref{odd_section}, $z$ is a sum of extensible sections, and so it is extensible as well. 

\endproof

Finally, we get the following theorem.

\begin{theor} Let $\gaff=\widehat{\mathfrak{sl}_2}$,  $\I$ be a finite interval far enough in the fundamental chamber and let $k$ be such that
$(\MGp_{|\I}, k)$ is a GKM-pair. Then for every finite interval  $\I$, the functor ${\cdot}^{\text{stab}}$ preserves indecomposable canonical sheaves.
\end{theor}

\section{General case}

In order to prove Theorem \ref{theor_stabfunctor}, we have to show that, for any interval $\I$ far enough in the fundamental chamber,
  if $\BMP$ is an indecomposable Braden-MacPherson sheaf on $\MGp_{|\I}$, then $\st\BMP$ is indecomposable and satisfies properties \emph{(BMP1)}, \emph{(BMP2)}, \emph{(BMP3)}, \emph{(BMP4)}.

\begin{lem}\label{lem_reduction} If $\F\in \textbf{Sh}_{\MGp_{|_{\I}}}$ satisfies \emph{(BMP1)}, \emph{(BMP2)}, \emph{(BMP4)}, 
 then also $\st\F$ has the same properties.
\end{lem}
\proof 
To begin with, let us observe that \emph{(BMP1)} and \emph{(BMP2)} follow immediately from the definition  of the pullback functor, as ${\cdot}^{\text{stab}}=i_{\I}^{*}$.

 Moreover $\MGst_{|\I}$ is obtained from $\MGp_{|\I}$ by deleting the non-stable edges and hence there is an inclusion
\begin{align*}
\Gamma(\I,\F)&=\left\{(m_x)\in \bigoplus_{x\in \mathcal{I}} \mathcal{F}^{x} \mid \begin{array}{c} 
\rho_{x,E}(m_x)=\rho_{y,E}(m_y)\\ \forall E:x-\!\!\!-\!\!\!- y\in\E ,\, x,y\in\mathcal{I}                                                                                                                  
                                                                                                                 \end{array}\right\}\\
&\subseteq\left\{(m_x)\in \bigoplus_{x\in \mathcal{I}} \mathcal{F}^{x}=\bigoplus_{x\in\I}(\st\F)^x \mid \begin{array}{c} 
\rho_{x,E}(m_x)=\rho_{y,E}(m_y)\\ \forall E:x-\!\!\!-\!\!\!- y\in\E_S ,\, x,y\in\mathcal{I}
                                                                                                                 \end{array}\right\}\\
&= \Gamma(\I,\st\F)
\end{align*}
By assumption there is a surjective map $\Gamma(\I,\F)\rightarrow \F^x=(\st\F)^x$ and, from the inclusion above, it follows that the map $\Gamma(\I,\st\F)\rightarrow (\st\F)^x$ is surjective too.

\endproof

Therefore it is only left to show that $\st\BMP$ is a flabby indecomposable sheaf on $\MGst_{|_{\I}}$.

\subsection{Flabbiness} Let $\MGa$ denote the regular Bruhat graph of $\gaff$, while let $\MGper$ be the periodic graph of 
 \S\ref{ssec_perMG}.

It is possible to define a functor ${\cdot}^{\text{per}}:\textbf{Sh}_{\MGa}\rightarrow \textbf{Sh}_{\MGper}$ in a very easy way.
 Let $\F=(\{\F^x\}, \{\F^E\}\{\rho_{x,E}\})$, then  we set $(\F^{\text{per}})^x=\F^x$ for any $x\in\V$, $(\F^{\text{per}})^E=\F^E$for 
any $E\in\E$ and  $\rho^{\text{per}}_{x,E}=\rho_{x,E}$.

 A fundamental step in the proof of  the flabbiness of $\st \BMP$ consists in showing that ${\cdot}^{\text{per}}$ maps canonical sheaves 
to flabby sheaves.

\subsubsection{Translation functors}
 
Here we recall the definition of  (left and right) translation functors and of the corresponding categories of special modules.

Let us denote by $\Z$ the structure algebra of the (affine regular) Bruhat graph $\MGa$. To any simple reflection $s\in\SRA$, 
we associate the two following automorphisms of the structure algebra.

First, let $\sigma_s$ be defined as follows.
\begin{equation*}
 \begin{array}{rccc}
\sigma_s:&\Z&\longrightarrow &\Z\\
&(z_x)_{x\in\WA}&\mapsto& (z'_x)_{x\in\WA}
   \end{array}
\qquad \text{with } z'_x=z_{xs}
\end{equation*}
Secondly, if $\tau_s$ is the automorphism of the symmetric algebra $S_k$ induced by the mapping $\lambda\mapsto s(\lambda)$ 
for all $\lambda\in \haff^*$, we set:
\begin{equation*}
 \begin{array}{rccc}
{}_s\sigma:&\Z&\longrightarrow &\Z\\
&(z_x)_{x\in\WA}&\mapsto& (z'_x)_{x\in\WA}
   \end{array}
\qquad \text{with } z'_x=\tau_s(z_{sx})
\end{equation*}

Let $\Zmod$ be the category of $\Z$-modules which are finitely generated and free as $S_k$-modules. Moreover, let us denote by
 $\Zs$ and $\sZ$
 the $S_k$-sub-modules of $\Z$ consisting of $\sigma_s$- and ${}_s\sigma$-invariants, respectively.  
To start with, let us observe that $\Z$ is a $\Zs$- and $\sZ$-module
under pointwise multiplication.. Let  $c^s=(c^s_x)_{x\in\WA}$ be the element of $\Z$ given by $c^s_x=x(\alpha_s)$, while 
 $\overline{\alpha_s}\in\Z$ denotes the constant section whose  components are all equal to $\alpha_s$. 
Thus the following lemma tells us that $\Z$ decomposes in a nice way as $\Zs$-module, 
resp. $\sZ$-module

\begin{lem}\label{lem_left_invariants}
\begin{itemize}
 \item [(i)] There is a decomposition $\Z=\Zs\oplus c^s\Zs$ of $\Zs$-modules (cf. \cite{Fie08b}, Lemma 5.1)
\item [(ii)] There is a decomposition $\Z=\sZ\oplus \overline{\alpha_s}\Zs$ of $\Zs$-modules (cf. \cite{L12}, Lemma 4.1)
\end{itemize}
\end{lem}

\begin{defin}
\begin{itemize}
 \item[(i)]
\begin{itemize}
 \item The \emph{(right) translation on the wall} is the functor $\theta^{s,on}:\Zmod \rightarrow\Zsmod$ given by 
$M\mapsto \text{Res}_{\Z}^{\Zs} M$.
\item The \emph{(right) translation out of the wall}  is the functor $\theta^{s,out}:\sZmod \rightarrow\Zmod$ defined by 
the $N\mapsto \text{Ind}_{\Z}^{\Zs} N$. Observe that this functor is well--defined due to Lemma \ref{lem_left_invariants}.
\item 
By composition, we get a functor $\theta^s:=\theta^{s,out}\circ\theta^{s,on}: \Zmod\rightarrow \Zmod$ that we call \emph{(right) translation functor}.
(cf. \cite{Fie08b}).
\end{itemize}
 \item[(ii)] \begin{itemize}
 \item The \emph{(left) translation on the wall} is the functor ${}^{s,on}\theta:\Zmod \rightarrow\sZmod$ given by 
$M\mapsto \text{Res}_{\Z}^{\sZ} M$.
\item The \emph{(left) translation out of the wall}  is the functor ${}^{s,out}\theta:\sZmod \rightarrow\Zmod$ defined by 
the $N\mapsto \text{Ind}_{\Z}^{\sZ} N$. Again, this functor is well--defined thanks to Lemma \ref{lem_left_invariants}.
\item 
By composition, we get a functor ${}^{s}\theta:={}^{s,out}\theta^{s,out}\circ{}^{s,on}\theta: \Zmod\rightarrow \Zmod$ that we call
 \emph{(left) translation functor}.
(cf. \cite{L12}). 
\end{itemize} 
\end{itemize}
\end{defin}

\begin{rem}
 The advantage of defining left translation functors is that they can be applied  also in  the parabolic setting. Indeed, if $\Z^J$ is the structure algebra of the Bruhat graph $\MG^J=\MG^J(\g\supseteq \b\supseteq \h, J)$, then the
automorphism  ${}_s\sigma^J:\Z^J\rightarrow \Z^J$ given by $(z_x)\mapsto (z'_x)$ with  $z'_x=\tau_s(z_{sx})$
 is still well-defined for any finitary $J\subseteq \SRA$ (cf. \cite{L12}, where the definition of left translation functors
 lead to a categorification of a parabolic Hecke module). In the case of $J=\SR$, the set of finite simple reflections, 
 we will denote the corresponding translation functor by ${}^s\theta^{\text{par}}$.
\end{rem}

\subsubsection{Flabby $\Z$-modules} 
 For all $k$-moment graphs $\mathcal{K}$, Fiebig defined the localization functor 
\begin{equation*}\mathscr{L}:\Z(\mathcal{K})-\text{mod}\rightarrow\textbf{Sh}_k(\mathcal{K}),
\end{equation*}
whose definition we are going to recall.

First of all we need to set some notation.  Let $Q$ be the quotient field of $S_k$ and,  for any  $M\in\Zmod$, let us
 write $M_{Q}:=M\otimes_{S_k}Q$. If $\mathcal{K}$ is a Bruhat graph, then  (see \S3.1 of \cite{Fie08a})  there is a decomposition $M_{Q}:=M\cap \bigoplus_{x\in \V} M_{Q}^{x}$ and so a canonical inclusion $M\subseteq \bigoplus_{x\in \V} M_{Q}^{x}$. For all subsets  of the set of vertices $\Omega\subseteq \V$, we may define:
\begin{equation*}M_{\Omega}:=M\cap \bigoplus_{x\in \Omega} M_{Q}^{x},\end{equation*}
\begin{equation*}M^{\Omega}:=M/M_{\V\setminus \Omega}=\im\left( M\rightarrow M_{Q}=\bigoplus_{x\in \Omega} M_{Q}^{x}\right).
\end{equation*}

For any vertex $x\in\V$, we set
\begin{equation*}\mathscr{L}(M)^x=M^x
\end{equation*}
For any edge $E:x-\!\!\!-\!\!\!-y$, let us consider  $\Z(E)=\{(z_x,z_x)\in S_k\oplus S_k\,\vert\, z_x-z_y\in l(E)\cdot S_k\}$ and $M(E):=\Z(E)\cdot M^{x,y}$.  For $m=(m_x,m_y)\in M(E)$, let us set $\pi_x((m))=m_x$, $\pi_y((m))=m_y$ and define $\mathscr{L}(M)^E$ as the push--out  in the following diagram of $S_k$-modules:
\begin{displaymath}
\begin{xymatrix}
{M(E) \ar[r]^{\pi_x}\ar[d]_{\pi_y}&M^x \ar[d]_{\rho_{x,E}}\\
M^y\ar[r]^{\rho_{y,E}}&\mathscr{L}(M)^E
}
\end{xymatrix}
\end{displaymath}

In this way, we also get the  maps $\rho_{x,E}$ and $\rho_{y,E}$.

Finally, from \cite[Theorem 3.5]{Fie08a}, the functor $\mathscr{L}$ is left adjoint to the global section functor 
\begin{equation*}\Gamma:\textbf{Sh}_k(\mathcal{K})\rightarrow \Z(\mathcal{K})-\text{mod}.
\end{equation*}


From now on, we will deal again with Bruhat moment graphs and we will denote once again by $\Z$  the structure algebra of
 the regularaffine  Bruhat graph $\MGa$.
 First of all, let us notice that, being the structure algebra of a moment graph independent from the partial order on the
set of vertices, $\Z$ coincides with the structure algebra of the periodic moment graph $\MGper$. Therefore we may summarise this setting as
 follows.

 \begin{displaymath}\xymatrix{ \Zmod \ar@/_1.4pc/[rr]_{\mathscr{L}} && \textbf{Sh}_k(\MGa)\ar@/_1.4pc/[ll]_{\Gamma}
}\qquad
\xymatrix{ \Zmod \ar@/_1.4pc/[rr]_{\mathscr{L}} && \textbf{Sh}_k(\MGper)\ar@/_1.4pc/[ll]_{\Gamma}
}
\end{displaymath}



Since for a BMP-sheaf $\F$ on $\MGa$ it holds $\mathscr{L}\circ \Gamma(\F)\cong \F$ (cf. \cite{Fie08a}), our claim is equivalent to the fact that 
$\mathscr{L}(\Gamma(\F^{\text{per}}))=\mathscr{L}(\Gamma(\F))^{\text{per}}$  is a flabby sheaf on $\MGper$ if $\F$ is a Braden 
MacPherson sheaf.
 We will prove it using translation functors.

\subsubsection{Three properties of the Bruhat order}\label{ssec_propertiesBruhat}
Let us consider the following properties of the Bruhat order.
\begin{itemize}
\item[(PO1)] The elements $w$ and $tw$ are comparable for all $w\in\WA$ and $t\in\T$. The relations between all such pairs 
$w$, $tw$ generate the partial order.
\item[(PO2)]We have $[w,ws] = \{w,ws\}$ for all $w\in\WA$ and $s\in\SRA$ such that $w<ws$.
\item[(PO3)]
\begin{itemize}
\item[(i)]
For $x, y\in\WA$ such that $x<xs$ and $y\leq xs$ we have $ys\leq xs$.
\item[(ii)] 
For $x, y\in\WA$ such that $xs <x$ and $xs \leq y$ we have $xs \leq ys$.
\end{itemize}
\end{itemize}

From now on,  $\trianglelefteq$ will denote a partial order on $\WA$ having the same properties as above. And let $\MG_{\trianglelefteq}$ be the
 moment graph on $\widehat{\coQ}$ obtained from the regular Bruhat graph $\MGa$ by changing the orientation of the edges according to 
this new partial order. This graph still satisfies the property of not having oriented cycles, thanks to (PO1) above.

\begin{defin}\label{def_tle_flabby}An object $M\in \Z-\text{mod}$ is \emph{$\trianglelefteq$-flabby}  if  the  corresponding sheaf $\mathscr{L}(M)$
 is a flabby sheaf on $\MG_{\trianglelefteq}$.
\end{defin}

In \cite{Fie08b} Fiebig proved that translation functors preserve $\leq$-flabby object, while he was proving them to 
preserve a stronger property of objects in $\Zmod$. In what follows, we are going to recall the main steps of his argument
and to point out that, whenever a proof depends on the Bruhat order on the set of  vertices, it actually depends only on the
 properties we listed before

\subsubsection{The moment graph $\MGa^{s}_{\trianglelefteq}$} A first fundamental observation made by Fiebig is that it is
 possible to identify the
 set of invariants $\Zs$ with the structure algebra of the Bruhat moment graph $\MGa^s=\MG(\gaff\supseteq \baff\supseteq \haff, \{s\})$ 
(cf. \cite[\S 5.1]{Fie08b} ). 
We were able to define Bruhat
moment  graphs, since the Bruhat order is preserved by the quotient map (see \cite[Proposition 2.5.1]{BB}). We want now to show
 that it makes sense to define a quotient graph of $\MG_{\trianglelefteq}$ too, if $J$ consists of only one simple reflection. 
Let $\trianglelefteq$ be a partial order having the above properties, then
for  any $s\in\SRA$, we will denote $\WA^s_{\trianglelefteq}$ the set of minimal
 representatives wrt $\trianglelefteq$ of the equivalence classes of $\WA/\{e,s\}$. Moreover, for any $x\in \WA$
let us set $\overline{x}_{\trianglelefteq}:=\min\{x,xs\}$, where the minimum is taken with respect to $\trianglelefteq$. 

\begin{lem}\label{lem_map_OP}
 The map $\WA\rightarrow \WA^{s}_{\trianglelefteq}$, given by $u\mapsto \overline{u}_{\trianglelefteq}$,
is order preserving.
\end{lem}
\proof
Let us suppose that $u\trianglelefteq v$ in $\WA$. We want to show that $\overline{u}_{\trianglelefteq} \trianglelefteq \overline{v}_{\trianglelefteq}$.

First of all, since $\overline{u}_{\trianglelefteq}\trianglelefteq u \trianglelefteq v$, if $v=\overline{v}_{\trianglelefteq}$, there is nothing to 
prove. Otherwise,  $\overline{v}_{\trianglelefteq}=vs\triangleleft v$ and again either
 $\overline{u}_{\trianglelefteq}=u\triangleleft us$ or $\overline{u}_{\trianglelefteq}=us\triangleleft u$. In the first case,
 we may apply (PO3) (ii) with $x=us$ and $y=v$ and obtain 
$\overline{u}_{\trianglelefteq}\trianglelefteq \overline{v}_{\trianglelefteq}$.
 On the other hand, if $\overline{u}_{\trianglelefteq}=us\triangleleft u$, by (PO3) (ii) with $x=u$ and $y=v$, it follows
$\overline{u}_{\trianglelefteq}\trianglelefteq \overline{v}_{\trianglelefteq}$.
\endproof

This leads us to the definition of the moment graph $\MGa^{s}_{\trianglelefteq}$ on $\widehat{\coQ}$, having as set of vertices 
$\WA^s_{\trianglelefteq}$, equipped with the partial order $\trianglelefteq$,  edges $x\rightarrow y$ if either 
$x-\!\!\!-\!\!\!- y$ or $x-\!\!\!-\!\!\!-ys$ was
 an edge of $\MG$, in which case the label is the same. It is clear that $\Zs$ coincides with the structure algebra of this 
graph  $\Z'$. Again, we say that $N\in \Z'\!-\!\text{mod}^{\text{f}}$ is $\trianglelefteq$-flabby if $\mathscr{L}(N)$ is a flabby sheaf on 
$\MG^s_{\trianglelefteq}$.

 For all  $\I\subseteq  \WA$, let us denote by $\overline{\I}_{\trianglelefteq}$ the image of $\I$ in 
$\WA^s_{\trianglelefteq}$ under the map $u\mapsto \overline{u}_{\trianglelefteq}$. The following two results do not depend 
on the partial order.

\begin{prop}\label{prop_ths}(cf. \cite{Fie08b}, Proposition 5.3) Let $M\in \Zmod$ and $N\in\Z'\!-\!\text{mod}^{\text{f}}$. Then
for $\I\subseteq \WA$ such that $s\cdot \I=\I$ we have
\begin{equation*}
 (\theta^{s, on}M)^{\overline{\I}_{\trianglelefteq}}= \theta^{s, on}(M^{\I}),\qquad
  (\theta^{s, out}N)^{\I}= \theta^{s, out}(N^{\overline{\I}_{\trianglelefteq}})
\end{equation*}
and
\begin{equation*}
 (\theta^{s, on}M)_{\overline{\I}_{\trianglelefteq}}= \theta^{s, on}(M_{\I}),\qquad
  (\theta^{s, out}N)_{\I}= \theta^{s, out}(N_{\overline{\I}_{\trianglelefteq}})
\end{equation*}
\end{prop}

\begin{lem}\label{lem_GammaL}(cf. \cite{Fie08a}, Lemma 3.3)
For any $M\in \Zmod$, resp. $N\in\Z'\!-\!\text{mod}^{\text{f}}$, and any $\I\in\WA$, resp 
$\mathcal{J}\subseteq \WA^s_{\trianglelefteq}$, it holds $\Gamma(\I, \mathscr{L}(M))=\Gamma(\mathscr{L}(M^{\I}))$,
resp $\Gamma(\mathcal{J}, \mathscr{L}(N))=\Gamma(\mathscr{L}(N^{\mathcal{J}}))$
\end{lem}

Therefore we get 

\begin{lem}\label{lem_thson_flabby}
 Let $M\in \Zmod$ be such that the adjunction morphism $M\mapsto \Gamma(\mathscr{L}(M))$ is an
isomorphism. If $M$ is $\trianglelefteq$-flabby, then also $\theta^{s,on}(M)\in \Z'\!-\!\text{mod}^{\text{f}}$ is
 $\trianglelefteq$-flabby. 
\end{lem}
\proof
By Lemma \ref{lem_map_OP}, if $\overline{I}_{\trianglelefteq}\subseteq \W^s_{\trianglelefteq}$ be an upwardly closed subset (wrt to $\trianglelefteq$), then 
also its preimage in $\WA$ is upwardly closed wrt $\trianglelefteq$. 

We want to show that the canonical map 
\begin{equation*} \Gamma(\mathscr{L}(\theta^{s,on}(M)))\rightarrow
 \Gamma( \overline{\I}_{\trianglelefteq},\mathscr{L}(\theta^{s,on}(M)))
\end{equation*}
is surjective.

By combining Lemma \ref{lem_GammaL} and Proposition \ref{prop_ths}, we get
\begin{equation*}\Gamma(\overline{\I}_{\trianglelefteq},\mathscr{L}(\theta^{s,on}(M))=\Gamma(\mathscr{L}(\theta^{s,on}(M^{\I})).
\end{equation*}

Because $(\Gamma\circ\mathscr{L})\circ\theta^{s,on}=\theta^{s,on}\circ (\Gamma\circ\mathscr{L})$,  our claim is equivalent to the 
surjectivity of the following map 
\begin{equation*}
 \theta^{s,on}(\Gamma\circ \mathscr{L}(M))\rightarrow \theta^{s, on}(\Gamma\circ \mathscr{L}(M^{\I})).
\end{equation*}

By applying once again Lemma \ref{lem_GammaL}, 
$\theta^{s, on}(\Gamma\circ \mathscr{L}(M^{\I}))=\theta^{s,on}(\Gamma(\I,\mathscr{L}(M)))$
 and hence the statement follows from the assumption of $M$ being $\trianglelefteq$-flabby, 
since $\theta^{s,on}$ is right-exact.
\endproof

\subsubsection{Local definition of translation out of the wall} In order to prove that also $\theta^{s,out}$
 preserves $\trianglelefteq$-flabby objects, we have to give a realization of it in terms of sheaves. Let us consider $N\in\Z'\!-\!\text{mod}^{\text{f}}$ such that it
$\theta^{s,out}(N)$ is isomorphic to  $\Gamma(\mathscr{L}(\theta^{s,out}(N)))$. Let us set
 \begin{equation*}\mathscr{M}:=\mathscr{L}(\theta^{s,out}(N))\in\textbf{Sh}_k(\MG_{\trianglelefteq})
 \qquad
\mathscr{N}:=\mathscr{L}(N)\in\textbf{Sh}_k(\MG^s_{\trianglelefteq})
 \end{equation*}

Let us denote by $p_s:\MGa_{\trianglelefteq}\rightarrow \MGa^s_{\trianglelefteq}$ the morphism of $k$-moment graphs defined as 
$p_{s,\V}:x\mapsto \overline{x}_{\trianglelefteq}$ and $p_{s,l,x}=\id$ for all $x\in\WA$. 

We may now reformulate Lemma 5.4 of \cite{Fie08b} in this notation. The proof only needs (PO1), (PO2), (PO3) of
 $\trianglelefteq$.
\begin{lem}\label{lem_pullths}(cf. \cite{Fie08b}, Lemma 5.4)
 It holds $\mathscr{M}=p_{s}^{*}(\mathscr{N})$.
\end{lem}

Recall that a poset $P$ is \emph{directed} if for any pair of elements $u,v\in P$ there exists a $z\in P$ such that  
$u,v\trianglelefteq z$. The following result is a generalization of \cite[Proposition 2.2.9]{BB}.

\begin{lem}\label{lem_tle_directed} The affine Weyl group $\WA$ equipped with the partial order $\tle$ is a directed poset.
\end{lem}
\proof
By induction on $\ell(u)+\ell(v)$. The base step is $u=v=e$ and in this case we can take $z=e$. 
Now let $u=s_1s_2\cdots s_r$ with $s_1,s_2,\ldots,s_r$ simple reflections and $r\geq 1$. Let us
 set $u':=us_r$; since $\ell(u')<\ell(u)$ there exists $z'$ be such that $u',v\trianglelefteq z'$. 
By (PO2), we know that $u$ and $u'$ are comparable, so we have two cases.

If $u\triangleleft u'$ then we take $z=z'$ and our claim follows from the inductive hypothesis.
Otherwise,  $u'\triangleleft u$ and we have again to consider two cases for $z'$, $z's_r$.
Suppose first that $z'\triangleleft z's_r$. Let $x=z'$ and $y=xs_r$ and notice that
$x=z'<z's_r=xs_r$
and $y=us_r\trianglelefteq z'=zs_r$, so we may apply  (PO3) (i) and find $u=ys_r\trianglelefteq
xs_r=z's_r$. So in this case we take $z=z's_r$ since we have also $v\trianglelefteq
z'<z's_r=z$.

Finally suppose $z's_r\triangleleft z'$. Let $x=z's_r$ and $y=us_r$ and notice that
$x\triangleleft xs_r$
and $y=us_r\trianglelefteq z'=xs_r$, so we may again apply (PO3) (i) and find $u=ys_r\trianglelefteq
xs_r=z'$.
So we take $z=z'$ and our claim is proved.
\endproof

For $u,v\in \WA$ we will write $u \dottl v$ if $u \tle v$ and $[u,v]=\{u,v\}$. The following lemma generalises 
\cite[Corollary 2.2.8]{BB}.

\begin{lem}\label{lem_tle_int}
 Let $s\in \SRA$ and $t\in\T$ be such that $s\neq t$. Let moreover $w\in\WA$ be such that $w\dottl ws,wt$. Then $ws,wt\dottl wts$.
\end{lem}
\proof The claim follows by applying repeatedly (PO3). 

To start with, by  (PO3) (ii) with $x=ws$ and $y=wt$, we get $w\trianglerighteq wts$. Since $s\neq t$, it holds $wts\neq w$
 and hence $wts\triangleleft wt$. Indeed, if it were not the case, $wts\in [w,wt]\setminus \{w,wt\}$ but by assumption
 $[w,wt,]=\{w,wt\}$. From (PO2) we deduce that $wts \dottl wt$. 

Now we apply (PO3) (i) with $x=wt$ and $y=w$ and it follows $ws\trianglerighteq wts$. It is only left to show that
 $[ws,wts]=\{ws,wts\}$.

Suppose there were a $z\in[ws,wts]\setminus \{ws,wts\}$. By  (PO1), either $z\triangleleft zs$ or $zs \triangleleft z$.  In the first case, 
by applying twice (PO3) (i)
(once with $y=ws$ and $x=z$  and once with $y=z$ and $x=wts$)
we obtain $zs\in [w,wt]=\{w,wt\}$, which is a contradiction, since we supposed $zs\neq ws,wts$. On the other hand, if $zs\triangleleft z$, 
we repeat the same argument as before, once $z$ has been substituted by $zs$, and we get $z\in[w,wt]$. It follows that either 
$z=w$ or $z=wt$, which is not possible because $w\triangleleft ws$ and $wt\triangleleft wts$, while we assumed $zs\triangleleft z$.
\endproof

 For $z\tle v\in \WA$, a \emph{maximal chain} from $z$ to $v$ of length 
$k$ is a sequence  $z=z_0\dottl z_1\dottl \ldots \dottl z_{k-1} \dottl z_k =v$, where $z_i\in\WA$.

\begin{lem}\label{lem_tle_chain}
 Let $u,v\in\WA$ be such that $u\triangleleft v$ and let $s\in\SRA$ be such that $vs\triangleleft v$ and $us\triangleleft u$. Then for
any   $z\in (us,v]\setminus \{\trianglerighteq u\}$ which is minimal, any maximal chain from $z$ to $v$ of length $k$ induces a
 maximal chain from $us$ to $v$ of the same length.
\end{lem}
\proof By the minimality assumption, $us \dottl z=z_0$ and,  (PO1), there exists a $t_0\in\T$
 such that $z_0=ust_0$. From Lemma \ref{lem_tle_int} with $w=us$ and $t=t_0$, we get $z_0,u\dottl z_0s$. If $z_1=z_0s$, we obtain
a maximal chain
\begin{equation*}
 us\dottl z_0s\dottl \ldots \dottl z_{k-1} \dottl z_k =v
\end{equation*}
Otherwise, we apply again Lemma \ref{lem_tle_int} with $z_0=w$ and $t=t_1$. Once again, if $z_2=z_1s$, we have done and if not we 
proceed as before. In this way we get a chain from $us$ to $v$ of length $k$.

Observe that all the $z_is$ are still in the interval, since $z_is\trianglerighteq v$, by (PO3) (i) with $x=vs$ and $y=z_i$.
\endproof

For any $u,v\in\WA$ with $w\trianglelefteq v$ let us denote by  $d(u,v)$ the supremum of the lengths of all maximal chains from $u$
 to $v$.

\begin{cor}\label{cor_tle_chain}
Let $u,v\in\WA$ be such that $u\triangleleft v$ and let $s\in\SRA$ be such that $vs\triangleleft v$ and $us\triangleleft u$. Let
$z\in (us,v]\setminus \{\trianglerighteq u\}$ be not minimal, then $d(z,v) < d(u,v)$.
\end{cor}

\begin{prop}
 Let $N\in\Z'\!-\!\text{mod}^{\text{f}} $ be such that the adjunction morphism 
$N\mapsto \Gamma(\mathscr{L}(N))$ is an isomorphism. If $N$ is $\trianglelefteq$-flabby,
 then also $\theta^{s,out}(N)\in \Zmod$ is $\trianglelefteq$-flabby. 
\end{prop}
\proof  Let us keep the same notation as before. We claim that, for any  set $\I\subseteq \WA$ upwardly closed 
 wrt $\trianglelefteq$ the canonical map $\Gamma(\mathscr{M}))\rightarrow
 \Gamma( \I_{\trianglelefteq},\mathscr{M})$ is surjective.

If $\I s=\I$, the proof is very similar to the one of Lemma \ref{lem_thson_flabby}, therefore we omit it.

Let us set $\{ \trianglerighteq u\}=\{ v\in \WA \mid  v\trianglerighteq u \}$, and analogously for $\{ \triangleright u\}$ then by
 \cite[Proposition 4.2]{Fie08a}  it is enough to demonstrate the surjectivity of the following map  for all $u\in \WA$ 
\begin{equation*}
 \Gamma(\{\trianglerighteq u\},\mathscr{M})\rightarrow
 \Gamma( \{\triangleright u\},\mathscr{M})
\end{equation*}

Let us consider $m\in\Gamma(\{\triangleright u\}, \mathscr{M}\})$.

We want to show that it is possible to extend it to a section $\tilde{m}\in \Gamma(\{\trianglerighteq u\},\mathscr{M})$. 

If $us\triangleright u$, then, from (PO3) (ii) with $x=us$ and $y=v$, if $v\triangleright u$ it holds  
$vs\triangleright u$ too. It follows that $\{\trianglerighteq u\}\setminus\{u,us\}$ is $s$-invariant. Moreover it is upwardly closed
because the interval $[u,us]=\{u,us\}$ by (PO2). Hence we can apply what we have already proven and get that 
if we restrict $m$ to $\{\trianglerighteq u \}\setminus \{u,us\}$ it extends to a global section. We restricted 
ourselves to the case in which $m$ is  supported at most on the vertex $us$, that is $m_v=0$ for all
 $x\in\{\triangleright u\}\setminus\{us\}$. By definition of pullback functor and Lemma \ref{lem_pullths},
 $\mathscr{M}^u=\mathscr{M}^{us}$ and so we can extend $m$ to the section $\tilde{m}\in\{\trianglerighteq u\}$,
by the setting $\tilde{m}_u:=\tilde{m_{us}}$. 

Let us now assume $us\triangleleft u$ and consider the following set
\begin{equation*}
 \text{supp}(\mathscr{M})=\{v\in\WA \mid \mathscr{M}^v\neq 0 \}.
\end{equation*}
Observe that $\text{supp}(\mathscr{M})$ is finite, since otherwise $\Gamma(\mathscr{M})$ would not be a finitely generated
 $S_k$-module. Therefore, by Lemma \ref{lem_tle_directed}, we find a $v\in \WA$ such that 
$\text{supp}(\mathscr{M})\subseteq \{\trianglelefteq 
v\}$. Let us observe that we may assume $vs\triangleright v$ without loss of generality.

Let us now proceed by induction on $d(u, v)$. The base step is $u=v$ and it is trivial. 
Otherwise, by  Corollary \ref{lem_tle_chain}, we may suppose by induction that $m$ extends to any vertex 
$z\in (us, v]\setminus \{us,u\}$ which is not minimal. We get in this way a section $m'$.

 Next, let us suppose $z\dottl us$ and $z\neq u$. In this case $z\triangleleft zs$. Indeed, suppose it were not the case. By (PO1) $z$ and $zs$ 
are comparable and by (PO3) (ii) with $x=u$ and $y=z$, it holds $us\trianglelefteq zs$.
 This contradicts the assumptions $z\dottl us$ and $z\neq u$. We have 
already proven that it is possible to extend $m'_{|{\triangleright z}}$ to the vertex $z$. Since the minimal vertices are not comparable,
 there are not edges between them and hence we extended $m'$ to a section $m''\in\Gamma(\{\trianglerighteq us\}\setminus\{u,us\}, \mathscr{M})$. 
But $\{\trianglerighteq us\}\setminus\{u,us\}$ is now an $s$-stable  set and therefore $m''$ extends to a global section and in particular
to the vertex $u$.
\endproof

Finally, as $\theta^{s}=\theta^{s,out}\circ \theta^{s,on}$ and since Lusztig proved that the generic order has 
also properties (PO1), (PO2) and (PO3) (cf. \cite{Lu80}), we get the following theorem.
\begin{theor}\label{theor_sthper_flabby}$\theta^s:\Zmod \rightarrow \Zmod$ preserves both $\leq$-flabby and
 $\preccurlyeq$-flabby objects.
\end{theor}

\subsubsection{Special modules}
Let $B_e\in\Zmod$ be the free $S_k$-module of rank one on which $z=(z_{x})_{x\in \WA}$ acts via multiplication by $z_e$.

\begin{defin} \item[(i)] The category of (right) \emph{special modules} is the full subcategory $\H$ of $\Zmod$ whose objects 
 are isomorphic to a direct summand of a direct sum of modules of the form
$\theta^{s_{i_1}}\circ\ldots \circ\theta^{s_{i_r}}(B^J_e)\langle n\rangle$, where $s_{i_1}, \ldots, s_{i_r}\in\SRA$ and 
$n\in\mathbb{Z}$.
\item[(ii)] The category of  \emph{parabolic special modules} is the full subcategory $\Hp$ of $\Zpmod$ whose objects  are
 isomorphic to a direct summand of a direct sum of modules of the form ${}^{s_{i_1}}\theta^{\text{par}}\circ\ldots \circ{}^{s_{i_r}}\theta^{\text{par}}(B^J_e)\langle n\rangle$, where $s_{i_1}, \ldots, s_{i_r}\in\SRA$ and $n\in\mathbb{Z}$.
\end{defin}

A fundamental characterisation of special modules is the following one.

\begin{theor}[cf. \cite{Fie08b}]\label{theor_Fie08bSM_BMP} Let $M\in\Zmod$. Then $M\in\H$  if and only if it is isomorphic to the
 space of global sections of a  BradenMacPherson sheaf on $\MGa$.
\end{theor}

Thanks to the above result, we are now able to prove the following, which will make clear why we are dealing with such objects.

\begin{prop}\label{prop_flabbyStab}Let $\F$ be a Braden-MacPherson sheaf on $\MGa$  then $\F^{\text{per}}$ is a flabby sheaf 
on $\MGper$.
\end{prop}
\proof We want to show that $F=\Gamma(\F)$ is a flabby object in $\Zpermod$. By Theorem \ref{theor_Fie08bSM_BMP},
 we know that $F\in \H$, so we may prove our result by induction. If $F=B_e$, there is nothing to prove. We have to show now that,
 if the claim is true for $M\in\H$, then it holds also for $\theta_s(M)\in\H$, that, again by Theorem \ref{theor_Fie08bSM_BMP}, 
is still isomorphic to the global sections of a Braden-MacPherson sheaf on $\MGa$. But
 now by the inductive hypothesis we get that $M$ is a flabby object in $\Zpermod$ and so, by applying Theorem 
\ref{theor_sthper_flabby}, $\theta_s(M)=\theta_s^{\text{per}}(M)$  is also a  flabby object in $\Zpermod$.

\endproof

\subsubsection{Decomposition of the functor ${\cdot}^{\text{stab}}$}

The functor ${\cdot^{\text{stab}}}$ may be obtained by composing the five following functors.
\begin{displaymath}\begin{xymatrix}{
\textbf{Sh}_{\MGp_{|_{\I}}}\ar[r]^{i_{*}}\ar@/_1.6pc/[rrrrr]<-1.2ex>_{{\cdot}^{\text{stab}}}&\textbf{Sh}_{\MGp}\ar[r]^{p_{\text{par}}^{*}}
&\textbf{Sh}_{\MGa} \ar[r]^{{\cdot}^\text{per}}&\textbf{Sh}_{\MGper}\ar[r]^{j^*}&\textbf{Sh}_{\MGper_{|_{\I}}}\ar[r]^{{\cdot}^{\text{opp}}}&\textbf{Sh}_{\MGst_{|_{\I}}}\\
}\end{xymatrix}
\end{displaymath}
Where
\begin{itemize}
\item
$i:\MGp_{|_{\I}}\hookrightarrow \MGp$  and $j:\MGper_{|_{\I}}\hookrightarrow\MGper
$
are the morphisms of moment graphs induced by the corresponding inclusions of subgraphs
\item $ p_{\text{par}}:\MG\rightarrow\MGp$ is the quotient homomorphism defined by
 $p_{\text{par}, \V}:x\mapsto x^{\text{par}}$, the minimal representative of $x\W$, and $p_{\text{par}, l,x}=\id_{\widehat{\coQ}}$ for all 
$x$ 
\item ${\cdot}^{\text{opp}}$ is the pullback of the isomorphism of moment graphs $f:\MGst_{|_{\I}}\rightarrow \MGper_{|_{\I}}$ defined as
 $f_{\V}=\id_{\widehat{\coQ}}$ and $f_{l,x}(\lambda)=x^{-1}(\lambda)$ for all $x\in\I$ and $\lambda\in\coQ$ 
\end{itemize}

Some properties of the above functors are needed.

\begin{prop} The functors $p_{\text{par}}^*$ and ${\cdot}^{\text{opp}}$ preserve indecomposable BMP-sheaves.
\end{prop}
\proof
The fact that  $p_{\text{par}}^*$ maps BMP-sheaves to BMP-sheaves is just a particular case of Theorem  6.2 of \cite{L11} in the 
reformulation we gave in \cite{L12}, Theorem 6.1.

 In Lemma 5.2 of \cite{L11}, we proved that $f$ defined as above is an isomorphism and hence, by Lemma \ref{pullbackBMP}, the pullbak 
 ${\cdot}^{\text{opp}}$ preserves Braden-MacPherson sheaves.
\endproof

\begin{theor} Let $\F\in\textbf{Sh}_{\MGp_{|_{\I}}}$ be a Braden-MacPherson sheaf, then ${\F}^{\text{per}}\in\textbf{Sh}_{\MGst_{|_{\I}}}$ is a flabby sheaf.
\end{theor}
\proof The statement follows by combining the results of this section, together with the fact that the functors $i_*$ and  $j^*$ 
clearly map flabby sheaves to flabby sheaves.
\endproof

\subsection{Indecomposability}

Here we prove the only step missing in the proof of Theorem \ref{theor_stabfunctor}. From now on,
 we will denote by $\WAp$ the set of minimal representatives for the equivalence classes of $\WA/\W$.

\subsubsection{Localisation of special $\Zp$-modules}\label{sssec_loc_SM}

Let $\beta\in\Delta_{+}$, we consider the following localisation of the symmetric algebra $S_k$:
\begin{equation}S_{k}^{\beta}:=S_k[(\al+n\delta)^{-1}\,\vert\,\al\in\Delta_{+}\!\setminus\{\beta\},\,n\in\mathbb{Z}]
\end{equation}

Fiebig used this localisation in \cite{Fie07a}, in order to relate the category of regular special modules to a category 
introduced by Andersen, Jantzen and Soergel in \cite{AJS}. 

Let us denote by  $\WA_{\beta}$ the subgroup of $\WA$ generated by the affine reflections $s_{\beta,n}$, for $n\in\mathbb{Z}$, 
and by $\WA^{\beta}$ the set of orbits for the left action of $\WA_{\beta}$ on $\WAp$. Remark that the group $\WA_{\beta}$ 
is isomorphic to $\widetilde{A_1}$.  For any subset $\Omega\subseteq \WAp$, let us write moreover
 $\Z^{\text{par},\beta}(\Omega):=\Z^{\text{par}}(\Omega)\otimes_{S_k}S_{k,\beta}$.

\begin{lem}[cf. \cite{Fie07a}, Lemma 3.1]Let $\Omega\subset \WAp$ be finite, then 
\begin{equation*}\Z^{\text{par},\beta}(\Omega)=\left\{(z_x)\in\bigoplus_{x\in\Omega}S_k^{\beta}\,\Big\vert 
\begin{array}{c}
z_{x}\equiv z_{y} \,\,\,(\!\!\!\!\mod \check{(\beta+n\delta)})\\
\text{\small{if }} \exists \,w\in\W,n\in\mathbb{Z} \text{\small{ s.t. }}y\,w\,x^{-1}=s_{\beta,n} 
\end{array}\right \}=\bigoplus_{\Theta\in\W_{\beta}}\Z^{\text{par},\beta}(\Omega\cap \Theta)
\end{equation*}
\end{lem}
\proof Omitted, since Fiebig's proof of \cite[Lemma 3.1]{Fie07a} works exactly the same  in this parabolic setting too.
\endproof
For $M\in\Hp$, we set $M^{\beta}:=M\otimes_{S_k}S_{k,\beta}$. Because any special module is a module on $\Z(\Omega)$ 
for some $\Omega\subset \WAp$ finite (see \cite[Lemma 4.3]{L12}), the decomposition of the previous lemma  induces the following decomposition.
\begin{equation}M^{\beta}=\bigoplus_{\Theta\in\W^{\beta}}M^{\beta, \Theta}
\end{equation}

Next we are going to show that this localisation procedure preserves special modules.
 In particular, we prove that, under the localisation, a special module having support on a finite interval far enough in the 
fundamental chamber splits in a direct sum of special modules for the subgeneric parabolic structure algebra, that is the one
 corresponding to the case $\g=\mathfrak{sl}_2$.

\begin{lem}\label{lem_locPSM} Let $M\in\Hp$ such that $\Zp$ acts on it via $\Zp(\I)$, for $\I$ a finite interval far enough
 in $\C^+$  and  $M^{\beta}=\bigoplus_{\Theta\in\W^{\beta}} M^{\beta,\Theta}$, then, for any $\Theta\in\W^{\beta}$, $M^{\beta,\Theta}$
 is isomorphic to a $\Z^{\text{par}}(\widehat{\frak{sl}_2})$-special module.
\end{lem}
\proof We prove by induction that any $M^{\beta, \Theta}$ is a special module for the structure algebra of
 $\MGp_{|_{\Theta}}$. If $M=B_e$, there is nothing to prove. Suppose the lemma holds for $M\in\Hp$; we have to show that it is 
true also for $\sth(M)=\bigoplus_{\Theta\in\W^{\beta}}\sth(M)^{\beta, \Theta}$.
 Thus it is enough to demonstrate it for one module $M^{\beta, \Theta}$. 

In order to do this, we follow the proof of \cite[Lemma 3.5]{Fie07a}.
  If $\Theta=\Theta s$, 
then $\sth(M)^{\beta, \Theta}=M^{\beta, \Theta}\otimes_{{}^s\Z^{\text{par}\beta}(\Theta)}\Z^{\text{par},\beta}(\Theta)$, 
since, by \cite[Lemma 4.3]{L12}, the inclusion ${}^s\Z^{par,\beta}(\Omega)\subset \Z^{par,\beta}(\Omega)$ contains
 ${}^s\Z^{\text{par}, \beta}(\Theta)\subset \Z^{\text{par}, \beta}(\Theta)$ as a direct summand. Otherwise,
  $\Theta\neq \Theta s$ and the inclusion
 ${}^s\Z^{\text{par}, \beta}(\Theta\cup \Theta s)\subset \Z^{\text{par}, \beta}(\Theta)\oplus \Z^{\text{par}, \beta}(\Theta s)$
 is an isomorphism on each direct summand. Therefore we get the equality
 $\sth(M)^{\text{par},\beta}=M^{\beta, \Theta}\oplus M^{\beta, \Theta s}$. 
In both cases, we obtain the claim by applying the inductive hypothesis because $\Z^{\text{par}, \beta}$ acts on
 $M^{\beta, \Theta}$ via $\Z^{\text{par}, \beta}(\I\cap \Theta)$ and  clearly 
$\Z^{\text{par}, \beta}(\I\cap \Theta)=\Z^{\text{par}}(\I\cap \Theta)$.

Now the statement follows because by Lemma \ref{PMGlimit}, for any finite interval $\I$ far enough in the fundamental chamber 
and any $\Theta\in\W^{\beta}$, the interval $\I\cap \Theta$ is isomorphic (as moment graph) to a finite parabolic interval of 
the  subgeneric moment graph.
\endproof

\subsubsection{A property of indecomposable BMP-sheaves} Before proving our last result, we first  have to recall an intrinsic property of indecomposable canonical sheaves, 
which has been used already by Braden and MacPherson in \cite{BM01} to show the existence of these objects. We follow a 
reformulation due to Fiebig.

 For any $k$-moment graph $\MG=(\V,\E,\tle,l)\in\MGYk$ and for any $x\in\V$, we denote (cf. \cite[\S 4.2]{Fie08a})
\begin{equation*}\E_{\delta x}:=\big\{  E\in\E\,\,|\,\, E:x\rightarrow y         \big\}
\end{equation*}
\begin{equation*}\V_{\delta x}:=\big\{ y\in \mathcal{V}\,\,|\,\, \exists E\in\E_{\delta x}  \text{ such that } E:x\rightarrow y       \big\}
\end{equation*}

Consider  $\F\in\ShMGk$ and define $\F^{\delta x}$ to be the image of $\Gamma(\{\triangleright x\},\F)$ under the composition $u_x$ of the following maps
\begin{equation} \label{Eqn_ux}
\begin{xymatrix}{
\Gamma(\{\triangleright x\},\F) \ar@{^{(}->}[r]\ar@/_2pc/[rrr]_{u_x}&\bigoplus_{y\triangleright  x}\!\F^y\ar[r]&\bigoplus_{y\in \V_{\delta x}}\!\!\F^y\ar[r]^{\oplus \rho_{y,E}}&\bigoplus_{E\in \E_{\delta x}} \!\!\F^E
}\end{xymatrix}
\end{equation}

Moreover, let us denote 

\begin{equation*}\begin{xymatrix}{
d_x:=(\rho_{x,E})_{E\in\E_{\delta x}}^T: \F^x\ar[r]&\bigoplus_{E\in \E_{\delta x}} \!\!\F^E
}\end{xymatrix}
\end{equation*}

If $\BMP$ is an indecomposable \emph{BMP}-sheaf, that is $\BMP=\BMP(w)$ for some $w\in \V$, then conditions \textit{(BMP3)} 
and \textit{(BMP4)} of Definition \ref{defin_BMP} may be replaced by the following condition (cf. \cite[Theorem 1.4]{BM01})

\vspace{2.3mm}
\textit{\hspace{-5mm}(BMP3') for all $x\in\V$, with $x\triangleleft w$,  $d_x: \BMP(w)^x\rightarrow \BMP(w)^{\delta x}$ is a projective cover in the category of graded $S_k$-modules}
\vspace{2.3mm}

Since a (graded) finitely generated $S_k$-module is projective if and only if it  is free, $\BMP(w)^x$ is isomorphic to the 
graded free $S_{k}$-module with minimal number of generators which maps surjectively to $\BMP(w)^{\delta x}$.

\begin{prop}\label{prop_indp_BMPst} Let $\I$ be a finite interval of $\MGp$ far enough in $\C^+$ and let 
$\BMP\in\textbf{Sh}_{\MGp_{|_{\I}}}$ be an indecomposable Braden-MacPherson sheaf. Then $\st\BMP$ is also indecomposable as 
sheaf on $\MGst_{|_{\I}}$.
\end{prop}
\proof Since $\BMP$ is indecomposable, by Theorem \ref{theor_FW_BMPindc}, $\BMP=\BMP(w)$ for some $w\in\I$, that implies 
$\BMP(w)^x=0=\stab{\BMP}{x}$ for all   $x> w$ ($x\in\I$) and $\BMP(w)^w\cong S_k\cong\stab{\BMP}{w}$.  Suppose that
 $\st\BMP=\C\oplus\mathcal{D}$, then for what we have just observed, we may take $\C$ and $\mathcal{D}$ such that 
$\C^x=\mathcal{D}^x=0$ for all $x>w$, $\C^w\cong S$ and $\mathcal{D}^w=0$. Let $y\in\I$ be a maximal vertex such that
$\mathcal{D}^y\neq 0$. For any $E:y\rightarrow z\in\E_{\delta y}$, by definition of Braden-MacPherson sheaf,
 $\rho_{z,E}:\BMP(w)^z=\stab{\BMP}{z}=\C^z\rightarrow \BMP^E=\stab{\BMP}{E}$ is surjective with kernel 
$l(E)\cdot\BMP^z=l(E)\C^z$ and this implies $\mathcal{D}^E=0$. 

We now localise $\Gamma(\BMP)$ at a finite simple root $\beta$. Remark that, since we are representing the parabolic Bruhat graph using alcoves, we are 
taking the quotient of $\MG^{\text{opp}}$ instead of $\MG$. It means that we have to twist the action of $S_k$ on any 
vertex $x$ by $x^{-1}$. However, once the action of the symmetric algebra is twisted, the two previous results still work 
in the same way.  By combining Theorem \ref{theor_Fie08bSM_BMP} and Lemma \ref{lem_locPSM} we know that 
$\mathscr{L}(\Gamma(\BMP)^{\beta})$ is a direct sum of Braden-MacPherson sheaves on certain moment graphs, each one of 
them is  isomorphic to a finite interval of the parabolic Bruhat graph for $\widetilde{A_1}$. From the definition of $\mathscr{L}$, it follows that $\mathscr{L}(\Gamma(\st\BMP)^{\beta})=(\mathscr{L}(\Gamma(\BMP)^{\beta}))^{\text{stab}}$.

We have already proven that $\rho_{y,E}(\mathcal{D}^y)=0$ for any $E\in\E_{\delta y}\cap \E_{S}$ and we want to show that
 $\rho_{y,E}(\mathcal{D}^y)=$ for any $E\in\E_{\delta y}$. If it were not the case, there would be a non-stable edge 
$F\in\E_{\delta y}\cap \E_{NS}$ such that $\rho_{y,E}(D^y)\neq 0$. Let $\beta\in\Delta_{+}$ be such that $l(F)=\beta+n\delta$ for some $n\in\mathbb{Z}$. Localising at $\beta$, we would get $\rho_{y,F}^{\beta}(\mathcal{D}^{y,\beta})\neq 0$ and from the $\widetilde{A_1}$ case, it follows that $\rho_{y,E}^{\beta}(\mathcal{D}^{y,\beta})\neq 0$ for all $E\in\E_{\delta y}$ in $\beta$-direction, but we proved that this is not the case. 


We are now ready to conclude. From what we showed, it follows that $u_y(\C^y)=\BMP^{\delta y}$ and this implies $\mathcal{D}^y=0$, since $(\BMP^{y}, u_y)$ is a projective cover of $\BMP^{\delta y}$.

\endproof

\appendix
\section{Finite parabolic intervals in the subgeneric case}

In this appendix we want to study the behaviour of indecomposable BMP-sheaves on finite intervals of $\MGp$ in the subgeneric
 case, that is $\gaff=\widehat{\mathfrak{sl}_2}$. In particular, we want to interpret in the moment graph setting the fact that the corresponding parabolic Kazhdan-Lusztig polynomials are all trivial. More precisely, we want to show that the structure sheaf $\A$ and the canonical sheaf $\BMPp$  are in this case isomorphic.
 Hence to prove our claim, it will be enough to define a surjective map $S_k\rightarrow {\BMPp(w)}^{\delta x}$ for any pair
 of vertices $x< w$ in $\MGp$. 

 Recall that the set of vertices is in this case totally ordered, so we may enumerate the vertices as follows,
 once identified the finite root $\al$ with the corresponding coroot $\check{\al}$: $v_{0}=0$, $v_{1}=\alpha$, $v_{2}=-\alpha$, ... , $v_{h}=(-1)^{h+1}[{\frac{h+1}{2}}]\alpha$, $\ldots$ .

From now on, we denote the edges as $E_{h,k}:v_h-\!\!\!-\!\!\!-v_k$ and the labels as  $l_{h,k}:=l(E_{h,k})$; we 
write moreover $l_{h,k}=\alpha+n_{h,k} c$. Actually, the label of an edge $E_{h,k}$ is by definition $\pm l_{h,k}$; however, there exists  an isomorphic $k$-moment graph with same sets of vertices  and edges, but this other label function and, by Lemma \ref{pullbackBMP}, the corresponding indecomposable canonical sheaves are isomorphic.

We will prove in several steps that, if $v_j\leq v_i$ and $(\MGp_{|[v_{j}, v_{i}]},k)$ is a GKM-pair, then $(\BMPp(v_i))^{v_j}\cong S_k$ by induction on $i-j$.

Let us fix once and for all $\I=\{v_i, v_{i-1}, \ldots, v_{j+1}\}$.

\begin{lem}\label{subgen_lemma1} Let $r\in \NN$ be such that $r<i-j$. If $(\MGp, k)$ is a GKM-pair, and $z\in \Gamma(\I, \BMPp(v_i))_{\{ r\}}$, then $z$ is uniquely determined by its first $r+1$ components, that is the restriction map 
\begin{equation*}\Gamma(\I, \BMPp(v_i))_{\{r\}}\longrightarrow \Gamma(\{v_i, v_{i-1}, \ldots, v_{i-r}\}, \BMPp(v_i))_{\{ r\}}
\end{equation*}
 is injective.
\end{lem}
\proof   Let $z\in \Gamma(\I, \BMPp(v_i))_{\{ r\}}$ such that $z_{v_i}=z_{v_{i-1}}=\ldots=z_{v_{i-r}}=0$. Observe that for any $j+1\leq h<i-r \leq k \leq i$ one has $z_{v_h}\equiv z_{v_k}=0\, (\!\!\mod -\alpha + n_{h,k}c )$.

From  the GKM-property it follows that  $GCD(-\alpha+n_{h,k}c, -\alpha+ n_{h,l}c)=1$ for any $i-r\leq k\neq l \leq i$. Since $S_k$ is an UFD, $z_{v_h}$ has to be divisible by $(-\alpha+n_{h,i-r}c)(-\alpha+n_{h,i-r+1}c)\ldots (-\alpha+n_{h,i}c)$. This is a polynomial of degree $r+1$ while $z_{v_h}$ was a polynomial of degree $r$, so $z_{v_h}=0$.
 \endproof

\begin{lem}\label{subgen_lemma2} Let $r\in \NN$ be such that $r<i-j$. We have $\text{dim}_{k} \Gamma(\I, \BMPp(v_i))_{\{ r\}}=\binom{r+2}{2}$.
\end{lem}
\proof By Lemma \ref{subgen_lemma1}, $\text{dim}_{k} \Gamma(\I, \BMPp(v_i))_{\{ r\}}=\text{dim}_{k} \Gamma(\{v_i, v_{i-1}, \ldots, v_{i-r}\}, \BMPp(v_i))_{\{ r\}}$.

Clearly, $\Gamma(\{v_i, v_{i-1}, \ldots, v_{i-r}\}, \BMPp(v_i))_{\{ r\}}\subseteq \bigoplus_{0}^r (S_k)_{\{r\}}$ and $\text{dim}_{k}\bigoplus_{0}^r (S_k)_{\{r\}}=(r+1)^2$.

By definition an element $m\in \bigoplus_{0}^r (S_k)_{\{r\}}$ is in $\Gamma(\{v_i, v_{i-1}, \ldots, v_{i-r}\}, \BMPp(v_i))_{\{ r\}}$ if it satisfies some (linear) conditions given by the labels of the edges. If we prove that such conditions are linearly independent, then we know that \begin{equation*}\text{dim}_{k} \Gamma(\{v_i, v_{i-1}, \ldots, v_{i-r}\}, \BMPp(v_i))_{\{ r\}}=\text{dim}_{k}\bigoplus_{0}^r (S_k)_{\{r\}}- \sharp \text{ edges}.
\end{equation*}

We noticed in \S\ref{sssec_sl2PMG} that in the $\widehat{\mathfrak{sl}_2}$ case all the vertices are connected, so the number of edges is equal to the  number of pairs of different elements in a set with $r+1$ elements, that is $\binom{r+1}{2}$. Then,
\begin{equation*}\text{dim}_{k} \Gamma(\{v_i, v_{i-1}, \ldots, v_{i-r}\}, \BMPp(v_i))_{\{ r\}}=(r+1)^2-\binom{r+1}{2}=\binom{r+2}{2}.
\end{equation*}

It is left to show that the conditions are linearly independent.

Let $i-r\leq h<k\leq i$ and define the element $(m^{(h,k)})\in \bigoplus_{0}^r(S_k)_{\{ r\}}$ in the following way:
\begin{equation*}m^{(h,k)}_{v_l}:=\Big\{
\begin{array}{ll}
c\prod_{m\in\{i, i-1, \ldots, i-r\}\setminus\{h,k\}}(-\alpha+n_{h,m}c)&\text{if } l =h\\
0&\text{otherwise}\\
\end{array}
\end{equation*}

Now $m^{(h,k)}_{v_l}=m^{(h,k)}_{v_m}$ for any $l,m\neq h$ and $c\prod(-\al+n_{h,m}c)\equiv 0\, (\!\!\!\!\mod \al+n_{h,m}c )$. By the GKM-property, $l_{h,k}$ does not divide  $m^{(h,k)}_{v_h}$, while  $m^{(h,k)}_{v_k}=0$. 

So for any condition coming from the edge $E_{l,m}$ we built a $r+1$-tuple which verifies all conditions except the $E_{l,m}$-th. It follows that all conditions are linearly independent.

\endproof

Let us denote by $m_{\alpha}, m_{c}\in \Gamma(\I, \BMPp(v_i))_{\{ 1\}}$ the constant  sections $m_{\al,v}=\alpha$, $m_{c,v}=c$ for all $v\in \I$. Denote moreover by  $u_{v_j}:=\oplus \rho_{v_h, E_{h,j}}$, where $\rho_{v_h, E_{h,j}}:S_k\rightarrow S_k/ (E_{h,j}\cdot S_k)$ are the canonical quotient maps.

\begin{lem}\label{subgen_lemma3}Let $r\in \NN$ and let $(\MGp, k)$ be a GKM-pair. The vector subspace of  $(\BMPp_{v_i})^{ v_j}$ generated by 
\begin{equation*}u_{v_j}(\ma^{r}), u_{v_j}(\ma^{r-1}\mc) \ldots  u_{v_j}(\ma\mc^{r-1}),  u_{v_j}(\mc^{r})
\end{equation*}
has dimension equal to $r+1$  if $r<i-j$  or dimension equal to $i-j$ otherwise.
\end{lem}
\proof 
 As first notice that $(\BMPp(v_i))^{E_{j,k}}=S_k/(l_{j,k}\cdot S_k)\cong k[c] $ by the  mapping $\al\mapsto n_{j,k} c$. Then 
\begin{equation*}u_{v_j}(\ma^k\mc^{t-k})=(n_{j,i}^k,n_{j,i-1}^k, \ldots, n_{j,j+1}^k)c.
\end{equation*}

We obtain the following matrix
\begin{equation*}N=\left(\begin{matrix}
1&1&\ldots&1\\
n_{j,i}& n_{j,i-1}&\ldots &n_{j,j+1}\\
n_{j,i}^2& n_{j,i-1}^2&\ldots &n_{j,j+1}^2\\
\vdots&\vdots& &\vdots\\
n_{j,i}^t& n_{j,i-1}^t&\ldots &n_{j,j+1}^t
\end{matrix}\right)
\end{equation*}

By the GKM-property it follows that $n_{j,k}\neq n_{j,h}$ for all pair $j+1\leq k\neq h\leq i$ and $N$ is a Vandermonde matrix. In particular, such a matrix is not singular and so it has maximal rank, that is $rk(N)=t+1$ if $t<i-j$ and $rk(N)=i-j$ otherwise.
\endproof

\begin{lem}\label{subgen_lemma4}There exists a section $m_0\in \Gamma(\I, \BMPp(v_i))_{\{ 1\}}$ such that $u_{v_j}(m_0)=0$ and $m_{0,v}\neq 0$ for all $v\in \I$.
\end{lem}
\proof Let $v_j=r\alpha$. Define $m_{0,v_h}:=(r-s)l_{j,h}=(r-s)(-\alpha+(r+s)c)$ if $v_h=s \alpha$.

Notice that $(m_0)\in \Gamma(\I, \BMPp(v_i))$; indeed for any pair of vertices $v_h=s\alpha$, $v_k=t\alpha$, one has $l_{h,k}=\alpha+ (s+t)c$ and 
\begin{align*}
m_{0,v_h}-m_{0,v_k}&=(r-s)(-\alpha+(r+s)c)-(r-t)(-\alpha+(r+t)c)\\
&=s\alpha-s^2c-t\alpha+t^2 c\\
&=\alpha(s-t)+c (t^2-s^2)\\
&=(s-t)(-\alpha+(s+t)c)\\
&\equiv 0 \,\,\,\,(\!\!\!\!\mod -\alpha+ (s+t)c).
\end{align*}
 
Finally, by definition $m_{0,v_h}\neq 0$ for any $v_h\in \I$ and $u_{v_j}((m_0))=0$.

\endproof

\begin{lem}\label{subgen_lemma5}Let $r\in \NN$ be such that $r<i-j$. The collection of monomials $\{ m_{\alpha}^{l}m_{c}^h m_0^k\,|\,l,h,k\geq 0,\, l+h+k=r\}$ is a basis of $\Gamma(\I,\BMPp(v_i))_{\{r\}}$ as $k$-vector space.
\end{lem}
\proof

Since the number of monomials in three variables of degree $r$ is $\binom{r+2}{2}$ and by Lemma \ref{subgen_lemma2} 
$\text{dim}_k\Gamma(\I, \BMPp(v_i))= \binom{r+2}{2} $  as well, it is enough to prove that all monomial in $\ma$, $\mc$, $m_0$ are linearly independent. We prove the claim by induction on $r$.

Let $r=1$. If $x\ma+y\mc+zm_0=0$, then clearly $0=u_{v_j}(x\ma+y\mc+zm_0)=xu_{v_j}(\ma)+yu_{v_j}(\mc)+zu_{v_j}(m_0)$. By Lemma \ref{subgen_lemma4}  $u_{v_j}(m_0)=0$, so $xu_{v_j}(\ma)+yu_{v_j}(\mc)=0$. But by Lemma \ref{subgen_lemma3} $u_{v_j}(\ma)$ and $u_{v_j}(\mc)$ generate a vector space of dimension $2$, then $x=y=0$. Finally, from $zm_0=0$ and Lemma \ref{subgen_lemma4} it follows $z=0$.

 Now let $r>1$. Let $z=\sum_{l+m+n=r}x_{l,m,n}\ma^l\mc^m m_0^n=0$. We can write $z=z_1+z_0 m_0$,  where $z_1$ is such that $m_0$ does not appear. Then by Lemma \ref{subgen_lemma3} $ u_{v_j}(z)=u_{v_j}(z_1)+u_{v_j}(z_0)u_{v_j}(m_0)=u_{v_j}(z_1)=0$. From Lemma \ref{subgen_lemma3} we know that all $u_{v_j}(\ma^l \mc^{r-l})$ are linearly independent and so $$0=u_{v_j}(z_1)=u_{v_j}(\sum_{l+m=r}x_{l,m,0}\ma^l \mc^m)=\sum_{l+m=r}x_{l,m,0}u_{v_j}(\ma^{l}\mc^m)$$ implies $x_{l,m,0}=0$ for all pair $l,m$,that is $c_1=0$. Thus we obtain $c_0m_0=0$ and we conclude by Lemma \ref{subgen_lemma4} that  $c_0=0$. Finally, $0=c_0=\sum_{l+m+n=r-1} x_{l,m,n+1}\ma^l\mc^m m_0^n$ is a linear combination of monomials in $\ma, \mc, m_0$ of degree $r-1$ and so by the inductive hypothesis we have $x_{l,m,n+1}=0$ for all $l,m,n$.

\endproof

\begin{theor}\label{theor_tA1}If $v_j\leq v_i$ and $(\MGp_{|_{[v_{j}, v_{i}]}},k)$ is a GKM-pair, then $(\BMPp(v_i))^{v_j}\cong S_k$. 
\end{theor}
\proof We prove that $(\BMPp(v_i))^{\delta v_i}$ coincides with the $u_{v_i}$ image of the ring generated by $\ma$ and $\mc$. If $r<i-j$, by \ref{subgen_lemma5}, $\Gamma(\I,\BMPp(v_i))_{\{r\}}$ is generated by $\{ m_{\alpha}^{l}m_{c}^h m_0^k\,|\,l,h,k\geq 0,\, l+h+k=r\}$. From \ref{subgen_lemma4} it follows $(\BMPp(v_i))^{\delta v_i}=u_{v_i}(\Gamma(\mathcal{I},\BMPp(v_i))_{\{r\}})$ is contained in the ring generated by $u_{v_i}((\ma))$ and $u_{v_i}((\mc))$.

Otherwise, $r\geq i-j$ and $\bigoplus_{E_{i,k}\in\E_{\delta v_i}}(\BMPp(v_i))^{E_{i,k}}\cong k[c]^{i-j}$, having dimension $i-j$.  Then by Lemma \ref{subgen_lemma3} $u_{v_i}(\ma)$ and $u_{v_i}(\mc)$ generate $(\BMPp(v_i))^{\delta v_i}$

Thus we have a surjective  map $S_k\rightarrow \bigoplus_{E_{i,k}\in \E_{\delta v_i}} (\BMPp(v_i))^{E_{i,k}}$ given by the mapping $\alpha\mapsto \ma$ and $c\mapsto \mc$. Then $(\BMPp(v_i))^{v_j}\cong S_k$.
\endproof

\end{document}